\documentclass[a4paper]{article}

\usepackage[T1]{fontenc}

\usepackage{amsmath, amssymb, amsthm}

\usepackage{geometry}
\usepackage{graphicx}
\usepackage{cancel}
\usepackage{dsfont}

\newtheorem{theo}{Theorem}[section]
\newtheorem{prop}[theo]{Proposition}
\newtheorem{lem}[theo]{Lemma}
\newtheorem{cor}[theo]{Corollary}

\theoremstyle{definition}
\newtheorem{defi}[theo]{Definition}

\theoremstyle{remark}

\newtheorem*{rem}{Remark}

\renewcommand\leq{\leqslant}
\renewcommand\geq{\geqslant}

\newcommand\N{\mathbb N}
\newcommand\Z{\mathbb Z}

\newcommand\R{\mathbb R}
\newcommand\C{\mathbb C}

\newcommand{\supp}{\operatorname{supp}}

\newcommand{\ds}{\displaystyle}
\newcommand{\abs}[1]{\left\lvert#1\right\rvert}
\newcommand{\norme}[1]{\left\lVert#1\right\rVert}
\newcommand{\pare}[1]{\left(#1\right)}
\newcommand{\prods}[2]{\left\langle#1,#2\right\rangle}

\newcommand{\harm}{\pare{s+\frac{1}{2}}}

\DeclareMathAlphabet{\mathonebb}{U}{bbold}{m}{n}

\numberwithin{equation}{section}

\usepackage[intoc,refpage]{nomencl}
\usepackage[toc,page]{appendix} 
\usepackage{hyperref}
\usepackage{etoolbox}
\apptocmd{\sloppy}{\hbadness 10000\relax}{}{}

\begin{document}

\title{Scattering theory for the Dirac equation in Schwarzschild-Anti-de Sitter space-time}
\author{Guillaume \textsc{Idelon-Riton}}

\maketitle

\abstract{We show asymptotic completeness for linear massive Dirac fields on the Schwarzschild-Anti-de Sitter spacetime. The proof is based on a Mourre estimate. We also construct an asymptotic velocity for this field.}

\setcounter{tocdepth}{2} \tableofcontents

\section{Introduction}

The aim of this paper is to show asymptotic completeness for the massive Dirac equation on the Anti-de Sitter Schwarzschild space-time.\\
\indent When studying a physical system for which the dynamics is described by a Hamiltonian, one of the fundamental properties we want to prove is asymptotic completeness. Roughly speaking, it states that, for large time, our dynamics behave, modulo possible eigenvalues, like the well-understood dynamics described by what we call a free Hamiltonian.  \\
The first asymptotic completeness results in General Relativity were obtained by J. Dimock and B. Kay in $1986$ and $1987$ (\cite{DimKay86b},\cite{DimKay86a}, \cite{DimKay87}) for classical and quantum scalar fields. This study was pursued in the $1990$'s by A. Bachelot for classical fields. He obtains scattering theories for Maxwell fields in $1991$ \cite{Bachelot2} and Klein-Gordon fields in $1994$ \cite{Bachelot3}. After that, J-P. Nicolas obtained a scattering theory for massless Dirac fields in $1995$ \cite{JPN} and F. Melnyk obtained a complete scattering for massive charged Dirac fields \cite{Me03} in $2003$. In all these works, the authors used trace class perturbation methods. On the other hand, new techniques, using Mourre estimates, were applied to the wave equation on the Schwarzschild space-time in $1992$ by S. De Bièvre, P. Hislop and I.M Sigal \cite{DeBHS}. Using this method, a complete scattering theory for the wave equation on stationary asymptotically flat space-times was obtained by D. Häfner in $2001$ \cite{Ha01} and D. Häfner and J-P. Nicolas obtained a scattering theory for massless Dirac fields outside slowly rotating Kerr black holes in $2004$ \cite{HaNi04}, making use of a positive conserved quantity which exists for the Dirac equation and not for the Klein-Gordon equation. In $2004$, T. Daudé obtains a scattering theory for Dirac fields on Reissner-Nordström black holes \cite{Dau10} and on Kerr-Newman black holes in \cite{Dau04}. Using an integral representation for the Dirac propagator, D. Batic gives a new approach to the time-dependent scattering for massive Dirac fields on the Kerr metric in $2007$. Recently, V. Georgescu, C. Gérard and D. Häfner obtained an asymptotic completeness result for the Klein-Gordon equation in the De-Sitter Kerr black hole, see \cite{GGH14}. See also M. Dafermos, G. Holzegel and I. Rodnianski for scattering results for the Einstein equations \cite{DaHoRod14} and M. Dafermos, I. Rodnianski and Y. Shlapentokh-Rothman for a scattering theory for the wave equation on Kerr black holes exteriors \cite{DaRodRoth14}. One of the principal motivation for all these works is the study of the Hawking effect. That kind of results are needed to give a mathematically rigorous description of the Hawking effect, see \cite{Ba99} and \cite{Ha09}.\\
\indent In our work, we are concerned with problems that arise from the Anti-de Sitter background. Indeed, the Schwarzschild Anti-de Sitter space-time is a solution of the Einstein vacuum equations with cosmological constant $\Lambda <0$ containing a spherically symmetric black hole. This space-time has a non-trivial causality. In fact, it is not globally hyperbolic, that is to say, Cauchy data defined on a slice $\{t=constant\}\times ]r_{SAdS},+\infty[ \times \mathbb{S}^{2}$ (where $r_{SAdS}$ correspond to the horizon) do not uniquely determine the evolution of the field in all the space-time. So, first of all, there's a difficulty in defining the dynamic. This is due to the fact that, when studying the geodesics in Boyer-Lindquist coordinates, null geodesics can reach timelike infinity in finite time. This suggests that we will need to put asymptotic conditions as $r \to + \infty$ in order to determine the dynamic uniquely. This problem was first studied by Breitenlohner and Freedman (\cite{BrFr82b},\cite{BrFr82a}) for scalar fields. They showed that the need to put boundary conditions depends on the comparison between the mass of the field and the cosmological constant and discovered two critical values known as B-F bounds. More recently, A. Bachelot (\cite{Bachelot}) showed a similar bound for the Dirac equation in the Anti-de Sitter space-time using a spectral approach. This approach uses the fact that, in an appropriate coordinate system, the equation can be written as $i\partial_{t} \psi = i H_{m} \psi$ with $H_{m}$ independent of $t$. We thus have to construct a self-adjoint extension of $H_{m}$. In order to put the right boundary condition, we will understand the asymptotic behavior of the states in the natural domain of $H_{m}$. This kind of method was also used by Ishibashi and Wald (\cite{IshWa03},\cite{IshWa04}) for integer spin fields.\\
\indent Using other techniques, there has been some recent advances concerning scalar fields. We first mention the works of G. Holzegel and J. Smulevici who proved, using vectorfield methods, a result of asymptotic stability of the Schwarzschild-AdS space-time with respect to spherically symmetric perturbations thanks to an exponential decay rate of the local energy \cite{Holsmu}. However, looking at the solutions of the linear wave equation on the Schwarzschild-AdS black hole with arbitrary angular momentum $l$, resonances with imaginary part $e^{-\frac{C}{l}}$ appear (see \cite{Gan14} for details) and the local energy only decays logarithmically. The same phenomenon appear in the Kerr-AdS space-time, see \cite{Holsmu2}. Thus Kerr-AdS is supposed to be unstable. In these papers, it was supposed that the Dirichlet boundary condition holds. More recently, G. Holzegel and C.M. Warnick considered other boundary conditions for the wave equation on asymptotically AdS black hole \cite{HolWar}. This includes some boundary conditions considered in the context of AdS-CFT correspondence. This correspondence was also in mind of A. Bachelot in his paper about the Klein-Gordon equation in the $\text{AdS}^{5}$ space-time \cite{Bachelot7} and of A. Enciso and N. Kamran when they study the Klein-Gordon equation in $\text{AdS}^{5} \times Y^{p,q}$ where $Y^{p,q}$ is a Sasaki-Einstein $5$-manifold \cite{EncKam}.

~\\
\indent We now present our results. We denote the natural domain of $H_{m}$ by
\begin{equation*}
D(H_{m}) = \left \{ \phi \in \mathcal{H} ; \hspace{2mm} H_{m}\phi \in \mathcal{H} \right \},
\end{equation*}
and we will use $l^{2} = -\frac{3}{\Lambda}$ where $\Lambda<0$ is the cosmological constant. We obtain:
\begin{prop} 
For $2ml \geq 1$, the operator $H_{m}$ is self-adjoint on $D(H_{m})$.
\end{prop}
For the case $2ml<1$, we will put MIT boundary conditions. This defines an operator $H_{m}^{MIT}$ with natural domain $D\pare{H_{m}^{MIT}}$. Then we obtain:
\begin{prop}
The operator $H_{m}^{MIT}$ is self-adjoint on $D\pare{H_{m}^{MIT}}$.
\end{prop}
The Cauchy problem is then well-posed by Stone's theorem.\\
We then turn our attention to the scattering theory. By means of a Mourre estimate, we are able to prove velocity estimates. We then introduce the comparison operator $H_{c}=i \gamma^{0} \gamma^{1} \partial_{x}$ with domain $D\pare{H_{c}}= \left \{ \varphi \in \mathcal{H}_{s,n}; H_{c} \varphi \in \mathcal{H}_{s,n}, \hspace{1mm} \varphi_{1} \pare{0} = - \varphi_{3} \pare{0}, \hspace{1mm} \varphi_{2} \pare{0} = \varphi_{4} \pare{0} \right \}$. Making use of the velocity estimates, we obtain the following asymptotic completeness result:
\begin{theo}[Asymptotic completeness] 
For all $m>0$ and all $\varphi \in \mathcal{H}$, the limits:
\begin{flalign}
\hphantom{A} & \lim_{t \to \infty} e^{it H_{c}} e^{-itH_{m}} \varphi \\
& \lim_{t \to \infty} e^{itH_{m}} e^{-itH_{c}} \varphi
\end{flalign}
exist. If we denote these limits by $\Omega \varphi$ and $W \varphi$ respectively, then we have $\Omega^{*} = W$.
\end{theo}
We eventually study the asymptotic velocity. We will say that $B =\mathrm{s}- \mathrm{C_{\infty}}-\underset{n \to \infty}{\lim} B_{n}$ if, for all $J \in C_{\infty}\pare{\R}$, we have $J\pare{B} = \mathrm{s}-\underset{t \to \infty}{\lim} J\pare{B_{n}}$ (where $C_{\infty}\pare{\R}$ is the set of continuous functions which go to $0$ at $\pm \infty$). Then, we obtain the following:
\begin{theo}[Asymptotic velocity for $H_{m}$]
Let $J \in C_{\infty} \pare{\R}$ and $A= -\gamma^{0} \gamma^{1}x$ where $\gamma^{0}$, $\gamma^{1}$ are Dirac matrices. Then, for all $m >0$, the limit:
\begin{equation}
\mathrm{s}-\underset{t \to \infty}{\lim} e^{itH_{m}} J\pare{\frac{A}{t}} e^{-itH_{m}}
\end{equation}
exists. Moreover, if $J\pare{0} = 1$, then
\begin{equation}
\mathrm{s}- \underset{R\to \infty}{\lim} \pare{\mathrm{s} - \underset{t \to \infty}{\lim} e ^{itH_{m}} J\pare{\frac{A}{Rt}} e^{-itH_{m}}} = \mathds{1}.
\end{equation}
If we define
\begin{equation}
\mathrm{s}- \mathrm{C_{\infty}}-\underset{t \to \infty}{\lim} e^{itH_{m}} \frac{A}{t} e^{-itH_{m}} =: P^{+}_{m},
\end{equation}
then the self-adjoint operator $P^{+}_{m}$ is densely defined and commute with $H_{m}$. The operator $P^{+}_{m}$ is called the asymptotic velocity and is in fact the identity operator.
\end{theo}
The paper is organized as follows.\\
\indent In section $2$, we present the Schwarzschild-AdS geometry and, due to the lack of global hyperbolicity, the fact that radial null geodesics go to infinity in finite time. Using the Newman-Penrose formalism, we then obtain the Dirac equation on this space-time and give a spectral formulation of this equation for a coordinate system $\pare{t,x,\theta,\varphi}$ where the horizon corresponds to $x$ goes to $-\infty$ and the Anti-de Sitter infinity corresponds to $x=0$. We eventually generalize this equation by giving asymptotic behaviors of the potentials and we ensure that the Dirac equation in the Schwarzschild-AdS space-time is part of our generalization. In the rest of the paper, we will work with this generalization.

\indent In section $3$, we obtain the self-adjointness of our operator for all $m>0$. First, we present the spinoidal spherical harmonics and then we use this tool to decompose our operator (in fact, we diagonalize the Dirac operator on the sphere) which leads us to a $1+1$ dimensional problem for the operator now denoted $H_{m}^{s,n}$. Then we study the states in the natural domain $D\pare{H_{m}^{s,n}} = \{ \varphi \in \mathcal{H}_{s,n} \lvert H_{m}^{s,n} \varphi \in \mathcal{H}_{s,n} \}$. The problem is coming from the Anti-de Sitter infinity where the potential behaves badly. Nevertheless, the potential behaves like in the result of A. Bachelot on the Anti-de Sitter space. After a unitary transform we can use his result. In this way, we see that the states behave well when $2ml\geq 1$ but it degenerates at $0$ when $2ml<1$. When $2ml\geq 1$, we prove that our operator is essentially self-adjoint on $C_{0}^{\infty}\pare{]-\infty,0[}$ and, using an elliptic estimate and a Hardy-type inequality, we give a precise description of the domain. In the case $2ml<1$, we need to put a boundary condition to obtain the self-adjointness of our operator. In this paper, we have chosen the MIT boundary condition. This allows us to solve the Cauchy problem. We finally prove the absence of eigenvalues for this operator.

\indent In section $4$, we prove a compactness result. We use an approximation of our resolvent, separating the behavior close to the black hole horizon and close to $x=0$. We then obtain that $f\pare{x}\pare{H_{m}^{s,n}-\lambda}^{-1}$ is compact if $f$ goes to $0$ at the horizon and has a finite limit at $x=0$.

\indent In section $5$, we obtain a Mourre estimate for $H_{m}^{s,n}$ using $A = \Gamma_{1} x$, where $\Gamma_{1}$ is the matrix $\text{diag}\pare{1,-1,-1,1}$, as conjugate operator.

\indent In section $6$, we obtain some propagation estimates. First, making use of the Mourre estimate and of an abstract result about minimal velocity estimates, we prove that the minimal velocity is $1$. Then, using a standard observable and a general result which uses Heisenberg derivative to obtain velocity estimates, we prove that the maximal velocity is also $1$.

\indent In section $7$, we are now able to prove asymptotic completeness for our hamiltonian. This result is first proved for fixed harmonics and then we prove that we can sum over all harmonics. It is proved by making use of the two velocity estimates and a similar reasoning as in the propagation estimates.

\indent In section $8$, we first prove the existence of the asymptotic velocity for $H_{c}$ and then deduce the same result for $H_{m}$ using the wave operators. We see that the asymptotic velocity operator is the identity.

\begin{center}
\bf{Aknowledgments}
\end{center}

This work was partially supported by the ANR project AARG.

\section{The Schwarzschild Anti-de Sitter space-time and the Dirac equation}

\indent In this section, we present the Schwarzschild Anti-de Sitter space-time and give the coordinate system that we will work with in the rest of the paper. We quickly study the radial null geodesics and then formulate the Dirac equation as a system of partial differential equations which are derived from the two spinor component expression of this equation by use of the Newman-Penrose formalism. We finally give a generalization of our equation by just considering a potential that have the same asymptotic behavior as in the case of the Schwarzschild Anti-de Sitter space-time.

\subsection{The Schwarzschild Anti-de Sitter space-time}

Let $\Lambda <0$. We define $l^{2}= \frac{-3}{\Lambda}$. We denote by $M$ the black hole mass.

In Boyer-Lindquist coordinates, the Schwarzschild-Anti-de Sitter metric is given by:
\begin{equation}
g_{ab}=\pare{1-\frac{2M}{r}+\frac{r^{2}}{l^{2}}} dt^{2} - \pare{1-\frac{2M}{r}+\frac{r^{2}}{l^{2}}}^{-1} dr^{2}- r^{2} \pare{d\theta^{2} + \sin^{2} \theta d\varphi^{2}}
\end{equation}
We define $F(r)= 1- \frac{2M}{r}+ \frac{r^{2}}{l^{2}}$. We can see that $F$ admits two complex conjugate roots and one real root $r=r_{SAdS}$. We deduce that the singularities of the metric are at $r=0$ and $r=r_{SAdS}=p_{+}+p_{-}$ where $p_{\pm}=\pare{Ml^{2} \pm \pare{M^{2}l^{4} + \frac{l^{6}}{27}}^{\frac{1}{2}}}^{\frac{1}{3}}$. (See \cite{Holsmu}) The exterior of the black hole will be the region $r \geq r_{SAdS}$ and our spacetime is then seen as $\R_{t} \times ]r_{SAdS},+\infty[ \times S^{2}$. It is well-know that the metric can be extended for $r \leq r_{SAdS}$ by a coordinate change which gives the maximally extended Schwarschild-Anti-de Sitter spacetime. In this paper, we are only interested in the exterior region.

In order to have a better understanding of this geometry, we study the outgoing (respectively ingoing) radial null geodesics (that is to say for which $\frac{dr}{dt}>0$ (respectively $\frac{dr}{dt}<0$)). Using the form of the metric we can see that along such geodesics, we have:
\begin{equation}
\frac{dt}{dr} = \pm F\pare{r}^{-1}.
\end{equation}
We thus introduce a new coordinate $r_{*}$ such that $t-r_{*}$ (respectively $t+r_{*}$) is constant along outgoing (respectively ingoing) radial null geodesics. In other words:
\begin{equation}
\frac{\mathrm dr_{*}}{\mathrm dr}= F(r)^{-1}.
\end{equation}
The coordinate system $\pare{t,r_{*},\theta,\varphi}$ is called Regge-Wheeler coordinates. $r_{*}$ is given by:
\begin{flalign}
r_{*}(r)= \notag & ln\pare{\pare{r-r_{SAdS}}^{\alpha_{1}}\pare{r^{2}+r_{SAdS}r+r^{2}_{SAdS}+l^{2}}^{-\frac{\alpha_{1}}{2}}}\\
\hphantom{A} &
+C\arctan\pare{\frac{2r+r_{SAdS}}{\pare{3r^{2}_{SAdS}+4l^{2}}^{\frac{1}{2}}}}.
\end{flalign}
where:
\begin{equation}
\alpha_{1}=\frac{r_{SAdS}l^{2}}{3r^{2}_{SAdS}+l^{2}} = \frac{1}{2\kappa}; \hspace{5mm} C=\frac{l^{2}\pare{3r^{2}_{SAdS}+2l^{2}}}{\pare{3r^{2}_{SAdS}+l^{2}}\pare{3r^{2}_{SAdS}+4l^{2}}^{\frac{1}{2}}}
\end{equation}
We obtain $ \lim_{r \to r_{SAdS}} r_{*}(r)=- \infty $ and $ \lim_{r \to \infty} r_{*}(r)= C \frac{\pi}{2}$.
We will consider the coordinate $x= r_{*} - C \frac{\pi}{2}$ rather than $r_{*}$. We thus have:
\begin{flalign}
\hphantom{A}& \lim_{r \to r_{SAdS}} x\pare{r}=- \infty \\
\hphantom{A}& \lim_{r \to \infty} x\pare{r}= 0.
\end{flalign}
This limit proves that, along radial null geodesic, a particle goes to timelike infinity in finite Boyer-Lindquist time (recall that along these geodesic, $t-r_{*}$ and $t+r_{*}$ are constants). This geometric property will be a major issue in our problem. This implies that our space-time is not globally hyperbolic, so that we cannot use the standard result by Leray about the global existence of solution of hyperbolic equations. A similar situation has been encountered by A.Bachelot in his article \cite{Bachelot} concerning the Dirac equation on the Anti-de Sitter space-time. We expect to do a similar study concerning the self-adjoint extension.

\subsection{The Dirac equation on Schwarzschild Anti-de Sitter space-time}

In the two components spinor notation, the Dirac equation takes the following form:
\begin{flalign} \label{eqDirSpin}
\hphantom{A} & \begin{cases}
\nabla_{AA'} \phi^{A} = -\mu \chi_{A'}\\
\nabla_{AA'} \chi^{A'} = -\mu \phi_{A}
\end{cases}
\end{flalign}
where $\nabla_{AA'}$ is the Levi-Civita connection, $\phi^{A}$ is a two-spinor, $\mu = \frac{m}{\sqrt{2}}$ and $m \geq 0$ is the mass of the field. \\
Thanks to the Newman-Penrose formalism, we can obtain the equation in the form of a system of partial differential equations. In this formalism, we introduce a null tetrad $\pare{l^{a},n^{a},m^{a},\bar{m}^{a}}$, that is
\begin{equation}
l_{a}l^{a}=n_{a}n^{a}=m_{a}m^{a}=\bar{m}_{a}\bar{m}^{a}= l_{a}m^{a} = n_{a}m^{a} =0,
\end{equation}
which is a basis of the complexified of the tangent space. We'll say that the tetrad is normalized if:
\begin{equation}
l_{a}n^{a}=1 \hspace{3mm} m_{a}\bar{m}^{a}=-1.
\end{equation}
The two vectors $l^{a}$ and $n^{a}$ correspond to the directions along which the light goes to infinity (we can choose $l^{a}$ as an outgoing null vector and $n^{a}$ as an ingoing null vector). The vector $m^{a}$ admits bounded integral curves. The vectors $m^{a}$ and $\bar{m}^{a}$ will generate rotations. In our case, we will consider:
\begin{flalign*}
l^{a} \partial x_{a} = \frac{1}{\sqrt{2}} F(r)^{- \frac{1}{2}} \pare{\partial_{t} + \partial_{x}}, & \hspace{3mm} n^{a} \partial x_{a} =  \frac{1}{\sqrt{2}} F(r)^{- \frac{1}{2}} \pare{\partial_{t} - \partial_{x}}\\
m^{a} \partial x_{a} = \frac{1}{\sqrt{2} r} \pare{\partial_{\theta}- \frac{i}{\sin \theta} \partial_{\varphi}}, & \hspace{3mm} \bar{m}^{a} \partial x_{a} = \frac{1}{\sqrt{2} r} \pare{\partial_{\theta}+ \frac{i}{\sin \theta} \partial_{\varphi}}.
\end{flalign*}
We remark that this tetrad is normalized and since $t \pm x$ is constant along null geodesics, the vector $l^{a} \partial x_{a}$ and $n^{a} \partial x_{a}$ are null. Moreover, using the equation of radial null geodesics with $\lambda$ as our affine parameter, we deduce that $\frac{dt}{dr}= \frac{dt}{d\lambda}\frac{d\lambda}{dr} = F\pare{r}^{-1}$ which gives us an outgoing real null vector. We see as well that $m^{a}$ is linked to rotations. We give the associated dual vectors:
\begin{flalign*}
l_{a} dx^{a} = \frac{1}{\sqrt{2}} F(r)^{\frac{1}{2}} \pare{dt - dx}, & \hspace{3mm} n_{a} dx^{a} = \frac{1}{\sqrt{2}} F(r)^{\frac{1}{2}} \pare{dt + dx} \\
m_{a} dx^{a} = \frac{r}{\sqrt{2}} \pare{-d\theta + i \sin(\theta) d\varphi}, & \hspace{3mm} \bar{m}_{a} dx^{a} = \frac{r}{\sqrt{2}} \pare{-d\theta - i \sin(\theta) d\varphi}.
\end{flalign*}
Using this tetrad, it is then possible to decompose the covariant derivative in directional derivatives along these directions. We introduce the following symbols:
\begin{equation*}
D = l^{a} \nabla_{a}, \hspace{2mm} D'=n^{a}\nabla_{a}, \hspace{2mm} \delta = m^{a} \nabla_{a}, \hspace{2mm} \delta' = \bar{m}^{a} \nabla_{a}.
\end{equation*}
We have twelve spin coefficients that are defined by the following expressions:
\begin{flalign*}
\hphantom{A} & \hat{\kappa} = m^{a} Dl_{a}, \hspace{5mm} \rho = m^{a} \delta' l_{a}, \hspace{5mm} \sigma = m^{a} \delta l_{a}, \hspace{5mm} \tau = m^{a} D' l_{a},\\
& \epsilon = \frac{1}{2} \pare{n^{a} D l_{a} + m^{a} D \bar{m}_{a}}, \hspace{5mm} \alpha = \frac{1}{2} \pare{n^{a} \delta' l_{a} + m^{a} \delta' \bar{m}^{a}}, \\
& \beta = \frac{1}{2} \pare{n^{a} \delta l_{a} + m^{a} \delta \bar{m}^{a}}, \hspace{5mm} \gamma = \frac{1}{2} \pare{n^{a} D' l_{a} + m^{a} D' \bar{m}_{a}},\\
& \pi = - \bar{m}^{a} D n_{a}, \hspace{5mm} \lambda = - \bar{m}^{a} \delta' n_{a}, \hspace{5mm} \mu = - \bar{m}^{a} \delta n_{a}, \hspace{5mm} \nu = - \bar{m}^{a} D' n_{a},
\end{flalign*}
where $\hat{\kappa}$ is the spin coefficient usually denoted $\kappa$, since $\kappa$ is the surface gravity in our convention. We can now give the equation \eqref{eqDirSpin} as a system of partial differential equations. These equations act on the components of the spinor $\phi^{A},\chi^{A'}$ in a normalized spinorial basis $(o^{A},\iota^{A})$ (that is such that $o_{A}\iota^{A}=1$). To choose our spinorial basis, we use the null tetrad above. Indeed, we can define the spinorial basis $(o^{A},\iota^{A})$, uniquely up to an overall sign, using the following conditions:
\begin{equation*}
o^{A}\bar{o}^{A'} = l^{a}, \hspace{3mm} \iota^{A} \bar{\iota}^{A'} = n^{a}, \hspace{3mm} o^{A} \bar{\iota}^{A'} = m^{a}, \hspace{3mm} \iota^{A} \bar{o}^{A'} = \bar{m}^{a}, \hspace{3mm} o_{A}\iota^{A} =1.
\end{equation*}
The dual basis is $\epsilon_{A}^{0}=- \iota_{A}$, $\epsilon_{A}^{1}= o_{A}$. Let $\phi^{0},\phi^{1},\chi^{0'},\chi^{1'}$ such that $\phi^{A} = \phi^{0} o^{A} + \phi^{1} \iota^{A} $ and $\chi^{A'} = \chi^{0'} o^{A'} + \chi^{1'} \iota^{A'}$ where $(o^{A'},\iota^{A'})$ is the conjugate basis of $(o^{A},\iota^{A})$. In this basis, the components of $\phi_{A}$ and $\chi_{A'}$ are respectively:
\begin{equation*}
\phi_{0}= -\phi^{1}, \hspace{3mm} \phi_{1}=\phi^{0}, \hspace{3mm} \chi_{0'}=-\chi^{1'}, \hspace{3mm} \chi_{1'}=\chi^{0'}.
\end{equation*}
We obtain the following system of partial differential equations:
\begin{equation} \label{eqDirComp}
\begin{cases}
l^{a} \partial x_{a} \phi_{1} - \bar{m}^{a} \partial x_{a} \phi_{0} + \pare{\epsilon - \rho} \phi_{1} - \pare{\pi - \alpha} \phi_{0} = \frac{m}{\sqrt{2}} \chi^{1'} \\
m^{a} \partial x_{a} \phi_{1} - n^{a} \partial x_{a} \phi_{0} + \pare{\beta - \tau} \phi_{1} - \pare{\mu - \gamma} \phi_{0} = - \frac{m}{\sqrt{2}} \chi^{0'} \\
l^{a} \partial x_{a} \chi^{0'} + m^{a} \partial x_{a} \chi^{1'} + \pare{\bar{\epsilon} - \bar{\rho}} \chi^{0'} + \pare{\bar{\pi} - \bar{\alpha}} \chi^{1'} = - \frac{m}{\sqrt{2}} \phi_{0} \\
\bar{m}^{a} \partial x_{a} \chi^{0'} + n^{a} \partial x_{a} \chi^{1'} + \pare{\bar{\beta} - \bar{\tau}} \chi^{0'} + \pare{\bar{\mu} - \bar{\gamma}}\chi^{1'} = - \frac{m}{\sqrt{2}} \phi_{1}.
\end{cases}
\end{equation}
Using the 4-component spinor $
\psi=\begin{pmatrix}
\phi_{A}\\
\chi^{A'}
\end{pmatrix}$, we obtain:
\begin{equation} \label{opé1}
\left( \partial_{t} + \gamma^{0}\gamma^{1} \pare{F(r)\partial_{r} + \frac{F\pare{r}}{r} + \frac{F'\pare{r}}{4}} +\frac{F(r)^{\frac{1}{2}}}{r} \cancel{D}_{\mathbb{S}^{2}} + im \gamma^{0} F(r)^{\frac{1}{2}}\right) \psi = 0. 
\end{equation}
where $m$ is the mass of the field and $\cancel{D}_{\mathbb{S}^{2}}$ is the Dirac operator on the sphere. In the coordinate system given by $\pare{\theta,\varphi}\in [0;2\pi]\times [0;\pi]$, we obtain: $\cancel{D}_{\mathbb{S}^{2}}= \gamma^{0} \gamma^{2} \pare{ \partial_{\theta} + \frac{1}{2} \cot \theta} + \gamma^{0} \gamma^{3} \frac{1}{\sin \theta} \partial_{\varphi}$ where singularities appear, but we just have to change our chart in this case. We will now work in these coordinates.\\
Recall that Dirac matrices $\gamma^{\mu}$, $0\leq \mu \leq 3$, unique up to unitary transform, are given by the following relations:
\begin{equation}
\gamma^{0^{*}} = \gamma^{0}; \hspace{3mm} \gamma^{j^{*}} = -\gamma^{j},\hspace{3mm} 1\leq j \leq 3;\hspace{3mm} \gamma^{\mu} \gamma^{\nu} + \gamma^{\nu} \gamma^{\mu} = 2 g^{\mu \nu}\mathbf{1},\hspace{3mm} 0\leq \mu, \nu \leq 3.
\end{equation}
In our representation, the matrices take the form:
\begin{equation} \label{MatDir}
\gamma^{0} = i\begin{pmatrix}
0 & \sigma^{0} \\
-\sigma^{0} & 0
\end{pmatrix}, \hspace{2mm} 
\gamma^{k} = i \begin{pmatrix}
0 & \sigma^{k} \\
\sigma^{k} & 0
\end{pmatrix}, \hspace{2mm} k=1,2,3
\end{equation}
where the Pauli matrices are given by:
\begin{equation}
\sigma^{0}=\begin{pmatrix}
1&0 \\
0 & 1
\end{pmatrix}, \hspace{1mm} 
\sigma^{1}= \begin{pmatrix}
1 & 0 \\
0 & -1
\end{pmatrix}, \hspace{1mm}
\sigma^{2}= \begin{pmatrix}
0 & 1 \\
1 & 0
\end{pmatrix}, \hspace{1mm}
\sigma^{3}= \begin{pmatrix}
0 & -i \\
i & 0
\end{pmatrix}.
\end{equation}
We thus obtain:
\begin{equation}
\gamma^{0}\gamma^{1} = \begin{pmatrix} 
-\sigma^{1}& 0 \\
0& \sigma^{1} \\
\end{pmatrix}; \hspace{2mm} \gamma^{0}\gamma^{2} = \begin{pmatrix}
-\sigma^{2}& 0 \\
0 &\sigma^{2}
\end{pmatrix};\hspace{2mm} \gamma^{0} \gamma^{3} = \begin{pmatrix}
-\sigma^{3}& 0 \\
0& \sigma^{3}
\end{pmatrix}.\label{refGamma0,1,2,3}
\end{equation}
We introduce the matrix:
\begin{equation}
\gamma^{5} = -i\gamma^{0}\gamma^{1}\gamma^{2} \gamma^{3}
\end{equation}
which satisfies the relations:
\begin{equation}
\gamma^{5}\gamma^{\mu} + \gamma^{\mu}\gamma^{5} = 0,\hspace{5mm} 0\leq \mu \leq 3. \label{relGamma5}
\end{equation}
We make the change of spinor $\phi(t,x,\theta,\varphi) = r F(r)^{\frac{1}{4}} \psi (t,r,\theta,\varphi)$ and obtain the following equation:
\begin{equation}
\partial_{t} \phi = i \left(i\gamma^{0}\gamma^{1} \partial_{x}  + i\frac{F(r)^{\frac{1}{2}}}{r} \cancel{D}_{\mathbb{S}^{2}} -m \gamma^{0} F(r)^{\frac{1}{2}} \right) \phi.
\end{equation}
We set:
\begin{equation}
H_{m} = i\gamma^{0}\gamma^{1} \partial_{x}  + i\frac{F(r)^{\frac{1}{2}}}{r} \cancel{D}_{\mathbb{S}^{2}} -m \gamma^{0} F(r)^{\frac{1}{2}}.
\end{equation}
We introduce the Hilbert space:
\begin{equation}
\mathcal{H} := \left [L^{2}\pare{\left]-\infty,0\right[_{x} \times S^{2}_{\omega}, dx d\omega} \right ]^{4}
\end{equation}

\subsection{Generalization}

Let $q \in \R$ and $n \in \N$, and define the spaces $T^{q,n}$ by:
\begin{equation}
T^{q,n}= \left \{ f \in C^{\infty}\pare{]-\infty;0[} \hspace{1.5mm} \vert \hspace{1.5mm} \forall \alpha \in \N, \hspace{1mm} \abs{\partial^{\alpha}_{x} f(x)} \lesssim
\begin{cases}
e^{qx} & ,\hspace{1mm} \text{when} \hspace{2mm} x \to - \infty \\
\pare{-x}^{n} & , \hspace{1mm} \text{when} \hspace{2mm} x \to 0
\end{cases} \right \}
\end{equation}
We consider two smooth functions $A_{0},B_{0}$ such that:
\begin{equation*}
A_{0} = \begin{cases}
0 & \text{if} \hspace{2mm} x \leq -2  \\
\frac{1}{l} & \text{if} \hspace{2mm} x \geq -1
\end{cases}; \hspace{5mm} B_{0} = \begin{cases}
0 & \text{if} \hspace{2mm} x \leq -2 \\
\frac{l}{-x} & \text{if} \hspace{2mm} x \geq -1.
\end{cases}
\end{equation*}
We will consider the following operator:
\begin{equation} \label{ExprHmDir}
H_{m} = \Gamma^{1} D_{x} + A(x) \cancel{D}_{\mathbb{S}^{2}} -m \gamma^{0} B(x)
\end{equation}
where $m$ is the mass of the field and, for two positive numbers $\vartheta,\beta$:
\begin{flalign}
\hphantom{A} & A-A_{0} \in T^{\vartheta,2}  \label{CompAsymptA}\\
& B- B_{0} \in T^{\beta,1}. \label{CompAsymptB}
\end{flalign} 
We also recall that $\Gamma^{1} = - \gamma^{0} \gamma^{1} = \text{diag} (1,-1,-1,1)$ and $D_{x} = \frac{1}{i} \partial_{x}$.\\ \indent We then check that the Schwarzschild Anti-de Sitter case enters in our abstract model. For $x$ going to $- \infty$, we have:
\begin{flalign*}
\hphantom{A} & r- r_{SAdS}  = \pare{3r^{2}_{SAdS} +  l^{2}}^{\frac{1}{2}}e^{- 2\kappa C \arctan \pare{\frac{3 r_{SAdS}}{\pare{3 \pare{r_{SAdS}}^{2} + 4 l^{2} }}}+C\pi \kappa}e^{2\kappa x} - C_{1}e^{4\kappa x} + o\pare{e^{4\kappa x}} \\
& F\pare{r}^{\frac{1}{2}}  = \frac{\pare{3r^{2}_{SAdS} + l^{2}}^{\frac{3}{4}}D_{4}^{\frac{1}{2}}}{r^{\frac{1}{2}}_{SAdS}l} e^{\kappa x} + C_{2}e^{3\kappa x}+ o \pare{e^{3\kappa x}},\\
& \frac{F\pare{r}^{\frac{1}{2}}}{r}  = \frac{\pare{3r^{2}_{SAdS} + l^{2}}^{\frac{3}{4}}D_{4}^{\frac{1}{2}}}{r^{\frac{3}{2}}_{SAdS}l} e^{\kappa x} + C_{3} e^{3\kappa x}  + o\pare{e^{3\kappa x} }
\end{flalign*}
where $C_{1},C_{2},C_{3}$ are constants. Then, for $x$ in a neighbourhood of $0$, we have:
\begin{flalign*}
\hphantom{A} & r = -\frac{l^{2}}{x} + \frac{1}{3} \pare{x} + o \pare{-x} \\
& F\pare{r}^{\frac{1}{2}} = -\frac{l}{x}-\frac{x}{6l}+ o\pare{x} \\
& \frac{F\pare{r}^{\frac{1}{2}}}{r} = \frac{1}{l} + \frac{x^{2}}{2 l^{3}} + o \pare{x^{2}}.
\end{flalign*}
The Schwarzschild Anti-de Sitter model is thus a particular case of our generalized model with $A= \frac{F(r)^{\frac{1}{2}}}{r}$ and $B=F(r)^{\frac{1}{2}}$.

\section{Study of the hamiltonian}

\indent In this section, we first present the spinoidal spherical harmonics. This allows us to reduce our problem to the study of a $1+1$ dimensional equation with a new hamiltonian denoted $H_{m}^{s,n}$. We then use the fact that, at AdS infinity, the potential looks like the one considered by A. Bachelot in \cite{Bachelot}. By means of a unitary transform and a cut-off near AdS infinity, we are able to make use of his result and obtain the asymptotic behavior of the elements in the natural domain of our operator. As in \cite{Bachelot}, the need or not to put a boundary condition is linked to the comparison between the mass of the field and the cosmological constant. For $2ml \geq 1$ (where $m$ is the mass of the field and $l$ is linked to the cosmological constant), there's no need to put boundary conditions. When $2ml <1$, we consider the generalized MIT-bag boundary condition in order to determine the dynamic uniquely. We then prove the self-adjointness of our operators. Using an elliptic inequality, we are able to give the domain of our operator for $2ml > 1$. Using Stone's theorem, we can solve the Cauchy problem for our equation. At last, we give a proof of the absence of eigenvalue for all $m>0$ which will be useful for the propagation estimates.

\subsection{Description of the domain}

\subsubsection{The spinoidal spherical harmonics}

In the rest of this paper, we will often make use of spinoidal spherical harmonics (we can refer to \cite{Bachelot} for a more complete presentation of these harmonics) which will permit us to decompose $\mathcal{H}$ as follows:
\begin{equation}
\mathcal{H} = \underset{(s,n) \in I}{\bigoplus} \pare{ \pare{L^{2}(x,dx)}^{4} \otimes \begin{pmatrix}
T^{s}_{-\frac{1}{2}, n} \\
T^{s}_{\frac{1}{2}, n} \\
T^{s}_{-\frac{1}{2}, n} \\
T^{s}_{\frac{1}{2}, n}
\end{pmatrix}} \label{decompS2}
\end{equation}
where: 
\begin{equation}
I := \left \{ (s,n); \hspace{2mm} s \in \N +\frac{1}{2}, \hspace{2mm} n \in \Z + \frac{1}{2}, \hspace{2mm} s - \abs{n} \in \N \right \}.
\end{equation}
These functions satisfy the following relations:
\begin{flalign}
\left (\frac{\partial}{\partial \theta} + \frac{1}{2 \tan \theta} \right ) T^{s}_{\pm \frac{1}{2},n}& = \pm \frac{n}{\sin \theta} T^{s}_{\pm \frac{1}{2},n} - i \left ( s+ \frac{1}{2} \right) T^{s}_{\mp \frac{1}{2},n}, \\
\frac{\partial}{\partial \varphi} T^{s}_{\pm \frac{1}{2},n} &= -in T^{s}_{\pm \frac{1}{2},n}. 
\end{flalign}
Since $\pare{T^{s}_{\frac{1}{2}, n}}_{(s,n)\in I}$ and $\pare{T^{s}_{-\frac{1}{2}, n}}_{(s,n)\in I}$ both span $L^{2}\pare{\mathcal{S}^{2}}$, we can decompose $f\in L^2(S^2)$ as follows:
\begin{equation*}
f(\theta,\varphi) = \underset{(s,n)\in I}{\sum} u^{s}_{\pm,n}(f) T^{s}_{\pm \frac{1}{2},n} (\theta,\varphi), \hspace{2mm} u^{s}_{\pm ,n}(f) \in \C.
\end{equation*}
Let us introduce the Hilbert spaces $W^{d}_{\pm}$ for $d \in \R$ as the closure of the space:
\begin{equation}
W^{\pm}_{f} := \left \{ \underset{finite}{\sum} u^{s}_{\pm,n} T^{s}_{\pm \frac{1}{2},n} ; \hspace{2mm} u^{s}_{\pm,n} \in \C \right \} \label{402}
\end{equation}
for the norm
\begin{equation*}
||f||^2_{W^{d}_{\pm}} := \underset{(s,n) \in I}{\sum} \left ( s + \frac{1}{2} \right )^{2d} |u^{s}_{\pm,n}(f)|^2.
\end{equation*}
Using Plancherel's formula, $L^{2}\pare{S^{2}}$ is just $W^{0}$. We give some properties of these spaces (for a more complete presentation, we refer to \cite{Bachelot}). We have: 
\begin{flalign*}
\hphantom{A}& d \geq 0  \Longrightarrow W^{d}_{\pm} = \left \{ f\in L^2 \left (S^2 \right) ; \hspace{2mm} ||f||_{W^{d}_{\pm}} < \infty \right \},\\
& \pare{W^{d}_{\pm}}' = W^{-d}_{\pm} \hspace{3mm} \text{and} \hspace{3mm} C^{\infty}_{0}\pare{]0,\pi[_{\theta} \times ]0,2 \pi[_{\varphi}} \subset W^{d}_{\pm}.
\end{flalign*}
We must remark that $T^{s}_{\pm \frac{1}{2},n}(\theta,2 \pi) = - T^{s}_{\pm \frac{1}{2},n} (\theta,0) \neq 0$. Consequently, these functions are not smooth on the sphere $S^2$. In correspondence with the decomposition \eqref{decompS2}, we introduce the Hilbert spaces:
\begin{equation}
\mathcal{W}^{d} = W^{d}_{-} \times W^{d}_{+} \times W^{d}_{-} \times W^{d}_{+}
\end{equation}
equipped with the norm:
\begin{equation}
\norme{\Phi}^{2}_{\mathcal{W}^{d}} = \overset{4}{\underset{j=1}{\sum}} \underset{\pare{s,n}\in I}{\sum} \pare{s+\frac{1}{2}}^{2d} \abs{u^{s}_{j,n}}^{2}
\end{equation}
where:
\begin{equation*}
\Phi\pare{\theta,\varphi} = \underset{\pare{s,n}\in I}{\sum} \begin{pmatrix}
u^{s}_{1,n} T^{s}_{-\frac{1}{2}, n} \pare{\theta,\varphi} \\
u^{s}_{2,n} T^{s}_{+\frac{1}{2}, n} \pare{\theta,\varphi} \\
u^{s}_{3,n} T^{s}_{-\frac{1}{2}, n} \pare{\theta,\varphi}\\
u^{s}_{4,n} T^{s}_{+\frac{1}{2}, n} \pare{\theta,\varphi}
\end{pmatrix}.
\end{equation*}

\subsubsection{A result due to A.Bachelot}

We recall a result obtained by A.Bachelot (see \cite{Bachelot}). In this article, the hamiltonian considered was:
\begin{equation} \label{OpDirBach}
H_m^{B} = i\gamma_{B}^{0}\gamma_{B}^{1} \pare{F_{B}(r)\partial_{r} + \frac{F_{B}\pare{r}}{r} + \frac{F_{B}'\pare{r}}{4}} +i\frac{F_{B}(r)^{\frac{1}{2}}}{r} \cancel{D}_{\mathbb{S}^{2}} -m \gamma_{B}^{0} F_{B}(r)^{\frac{1}{2}}
\end{equation}
in $\pare{r,\theta,\varphi}$ coordinates where $F_{B}\pare{r} = 1 + \frac{r{2}}{l^{2}}$. Here, $m$ is $\tilde{m} \sqrt{\frac{3}{\Lambda}}$ with $\tilde{m}$ the mass of the field and $-\Lambda$ the cosmological constant. Moreover, the space $\mathcal{L}^{2}$ is defined by $\mathcal{L}^2 := \left[L^2\left([0,\frac{\pi}{2}[_\zeta \times [0,\pi]_{\theta} \times [0,2 \pi[_{\varphi}, \sin \theta \mathrm d\zeta \mathrm d\theta \mathrm d\varphi \right) \right ]^4 $ where $\zeta = \arctan\pare{\sqrt{\frac{\Lambda}{3}}r}$.. Using a change of spinor and a change of coordinates such that $\phi(t,\zeta,\theta,\varphi) = r F_{B}(r)^{\frac{1}{4}} \psi (t,r,\theta,\varphi)$, he obtains:
\begin{equation}
H_m^{B}  :=  i \gamma^0_{B} \gamma^1_{B} \frac{\partial}{\partial \zeta} + \frac{i}{\sin \zeta} \left [ \gamma^0_{B} \gamma^2_{B} \left( \frac{\partial}{\partial \theta} + \frac{1}{2 \tan \theta} \right) + \frac{1}{\sin \theta} \gamma^0_{B} \gamma^3_{B} \frac{\partial}{\partial \varphi} \right] - \frac{m}{\cos \zeta} \gamma^0_{B}. \label{337}
\end{equation}
where he uses the natural domain:
\begin{equation}
D(H_m^{B}) := \left \{ \Phi \in \mathcal{L}^2; H_m^{B} \Phi \in \mathcal{L}^2 \right \}. \label{501}
\end{equation}
At last, we recall that the Dirac matrices $\gamma^0_{B},\gamma^{1}_{B},\gamma^{2}_{B},\gamma^{3}_{B}$ take the form:
\begin{equation} \label{MatDirBach}
\gamma^{0}_{B} = \begin{pmatrix}
I & 0\\
0 & -I
\end{pmatrix}, \hspace{2mm} 
\gamma^{k}_{B} =  \begin{pmatrix}
0 & \sigma^{k}_{B} \\
-\sigma^{k}_{B} & 0
\end{pmatrix}, \hspace{2mm} k=1,2,3
\end{equation}
where the Pauli matrices are given by:
\begin{equation*}
I = \begin{pmatrix}
1 & 0 \\
0 & 1
\end{pmatrix}, \hspace{1mm}
\sigma^{1}_{B}= \begin{pmatrix}
1 & 0 \\
0 & -1
\end{pmatrix}, \hspace{1mm}
\sigma^{2}_{B}= \begin{pmatrix}
0 & 1 \\
1 & 0
\end{pmatrix}, \hspace{1mm}
\sigma^{3}_{B} = \begin{pmatrix}
0 & -i \\
i & 0
\end{pmatrix}
\end{equation*}
The result is then the following (see Theorem V.1 in \cite{Bachelot}):
\begin{theo}\label{5.1}
For all $\Phi \in D(H_m^{B})$, we have:
\begin{equation}
\Phi \in C^0 \left ( \left [0,\frac{\pi}{2} \right [_\zeta ; \mathcal{W}^{\frac{1}{2}} \right )\hspace{3mm} \text{with} \hspace{3mm} || \Phi(\zeta,.) ||_{\mathcal{W}^{\frac{1}{2}}} = O (\sqrt{\zeta}), \hspace{2mm} \zeta \to 0,
\end{equation} 
and for $m >0$, we have 
\begin{equation}
\int^{\frac{\pi}{2}}_0 ||\Phi(\zeta,.)||^2_{\mathcal{W}^1} \frac{d\zeta}{\sin \zeta} \leq ||H_m \Phi ||^2_{\mathcal{L}^2}. \label{504}
\end{equation}
\\For $m > \frac{1}{2}$, we have 
\begin{equation}
|| \Phi(\zeta,.)||_{L^2(S^2)} = O \left( \sqrt{\frac{\pi}{2} -\zeta} \right), \hspace{2mm} \zeta \to \frac{\pi}{2}.\label{505}
\end{equation}
\\For $m= \frac{1}{2}$, we have 
\begin{equation}
||\Phi(\zeta,.)||_{L^2(S^2)} = O \left(\sqrt{\left(\zeta-\frac{\pi}{2} \right)ln\left( \frac{\pi}{2} -\zeta\right)} \right), \hspace{2mm} \zeta \to \frac{\pi}{2}. \label{506}
\end{equation}
\\For $0 < m < \frac{1}{2}$, there exist functions $\psi_{-} \in W^{\frac{1}{2}}_{-}$, $\chi_{-} \in W^{\frac{1}{2}}_{+}$, $\psi_{+},\hspace{0.5mm} \chi_{+} \in L^2(S^2)$ and $\phi \in C^0 \left([0,\frac{\pi}{2}]_{\zeta}; L^2(S^2;\C^4) \right)$ satisfying
\begin{equation}
\Phi(\zeta,\theta,\varphi) = \left(\frac{\pi}{2}-\zeta \right)^{-m}
\begin{pmatrix}
\psi_{-}(\theta,\varphi)\\
\chi_{-}(\theta,\varphi)\\
-i\psi_{-}(\theta,\varphi)\\
i \chi_{-}(\theta,\varphi)
\end{pmatrix}  + \left(\frac{\pi}{2} - \zeta \right)^m
\begin{pmatrix}
\psi_{+}(\theta,\varphi)\\
\chi_{+}(\theta,\varphi)\\
i\psi_{+}(\theta,\varphi)\\
-i\chi_{+}(\theta,\varphi)
\end{pmatrix} +\phi(\zeta,\theta,\varphi),\label{507}
\end{equation}
\begin{equation}
||\phi(\zeta,.)||_{L^2(S^2)}=o \left(\sqrt{\frac{\pi}{2}-\zeta} \right), \hspace{2mm} x \to \frac{\pi}{2}. \label{508}
\end{equation}
\\Conversely, for all $\psi_{-} \in W^{\frac{1}{2}+m}_{-}$, $\chi_{-} \in W^{\frac{1}{2}+m}_{+}$, $\psi_{+} \in W^{\frac{1}{2}-m}_{-}$, $\chi_{+} \in W^{\frac{1}{2}-m}_{+}$ there exists $\Phi \in D(H_m^{B})$ satisfying \eqref{507} and \eqref{508}.
\end{theo}

\begin{rem} This result concerning the asymptotic behavior of elements in the domain of the operator $H_{m}^{B}$ is first proved for fixed harmonics (i.e fixed $\pare{s,n} \in I$). In the next sections, we will often make use of the result obtained for fixed harmonics.
\end{rem}
The condition on the mass is a consequence of the fact that the states in the natural domain of our operator have to be in $L^{2}$. When the mass is sufficiently large, the term $\left(\frac{\pi}{2}-\zeta \right)^{-m}$ in \eqref{507} is not in $L^{2}$ so it cannot appear in the development of the states near $\frac{\pi}{2}$. In this case, we do not need to put boundary conditions to obtain the self-adjointness of this operator and well-posedness of the Cauchy problem.\\
Unfortunately, for a mass too small compared to the cosmological constant, we see that the term $\left(\frac{\pi}{2}-\zeta \right)^{-m}$ in \eqref{507} is in $L^{2}$ which is problematic for the symmetry of our operator. We thus need to put boundary conditions to get rid of this term and solve the Cauchy problem.

\subsubsection{Unitary transform of \texorpdfstring{$H_{m}$}{Hm}}

Let us introduce the following domains:
\begin{enumerate}
\item[-] If $ 2ml \geq 1$:
\begin{equation}
D(H_{m}) = \left \{ \phi \in \mathcal{H} ; \hspace{2mm} H_{m}\phi \in \mathcal{H} \right \}.
\end{equation}
\item[-] If $2ml < 1$, we consider the operator equipped with the domain whose elements satisfy a generalized MIT-bag condition (where $\alpha \in \R$ is called the Chiral angle and $\gamma^{5} = -i \gamma^{0} \gamma^{1} \gamma^{2} \gamma^{3}$ (see \cite{Bachelot})):
\begin{equation}
D(H_{m}) = \left \{ \phi \in \mathcal{H} ; \hspace{2mm} H_{m}\phi \in \mathcal{H}, \hspace{1mm} \norme{\pare{\gamma^{1} + i e^{i \alpha \gamma^{5}}} \phi}_{2} = o\pare{\sqrt{- x}}, \hspace{1mm} x \to 0 \right \}.
\end{equation}
\end{enumerate}
First, we'll try to remove $\alpha$ in the case $2ml < 1$. We introduce the following operator:
\begin{equation}
H^{\alpha}_{m} = e^{i \frac{\alpha}{2} \gamma^{5}} H_{m} e^{-i \frac{\alpha}{2} \gamma^{5}}.
\end{equation}
Since $e^{i \alpha \gamma^{5}}$ is unitary and $e^{i \alpha \gamma^{5}} \gamma^{1} =\gamma^{1} e^{-i \alpha \gamma^{5}}$, we see that $\varphi \in D\pare{H_{m}}$ if and only if $e^{i \frac{\alpha}{2} \gamma^{5}} \varphi \in D\pare{H_{m}^{\alpha}}$ where:
\begin{equation*}
D\pare{H_{m}^{\alpha}} = \left \{ \phi \in \mathcal{H} ; \hspace{2mm} H_{m}^{\alpha}\phi \in \mathcal{H}, \hspace{1mm} \norme{\pare{\gamma^{1} + i} \phi}_{2} = o\pare{\sqrt{-x}}, \hspace{1mm} x \to 0 \right \}.
\end{equation*}
So we can restrict to the case $\alpha = 0$ which we will do in the following. \vspace{2mm} \\
We will now modify our hamiltonian in order to exploit the result of A.Bachelot. We introduce a new time variable $\tilde{t} = -t$ (and we will continue to denote by $t$) which gives:
\begin{equation}
\partial_{t} \phi = i \pare{- H_{m}} \phi. \label{2}
\end{equation}
Let:
\begin{equation}
\tilde{H}_{m} = \gamma^{5}_{B} P^{-1} (-H_{m}) P \gamma^{5}_{B} \label{UnitTransH}
\end{equation}
where:
\begin{flalign*}
\hphantom{A}& P = \frac{1}{\sqrt{2}} e^{i \frac{\pi}{4}} \begin{pmatrix}
\mathrm{Id} & \mathrm{Id} \\
-i \mathrm{Id} & i \mathrm{Id}
\end{pmatrix}, \hspace{2mm}
 P^{*} = P^{-1} = \frac{1}{\sqrt{2}} e^{-i \frac{\pi}{4}} \begin{pmatrix}
\mathrm{Id} & i \mathrm{Id} \\
\mathrm{Id} & -i \mathrm{Id}
\end{pmatrix}, \hspace{2mm} \gamma^{5}_{B} = \begin{pmatrix}
0 & \mathrm{Id} \\
\mathrm{Id} & 0
\end{pmatrix},
\end{flalign*}
and $\mathrm{Id}$ is the identity matrix of order $2$. The matrix $P$ satisfies the following relations:
\begin{equation} \label{relP}
\gamma^{0} = P \gamma^{0}_{B} P^{-1}; \hspace{5mm} \gamma^{j} = - P \gamma^{j}_{B} P^{-1}, \hspace{3mm} 1 \leq j \leq 3.
\end{equation}
where the Dirac matrices are defined by \eqref{MatDirBach} and \eqref{MatDir}. The matrix $\gamma^{5}_{B}$ satisfies the same relations as $\gamma^{5}$ in \eqref{relGamma5}. We obtain:
\begin{equation}
\tilde{H}_{m} = i \gamma^{0}_{B} \gamma^{1}_{B} \partial_{x} + i \gamma^{0}_{B} \gamma^{2}_{B} A(x) \pare{\partial_{\theta} + \frac{1}{2} \cot \theta} + i \gamma^{0}_{B} \gamma^{3}_{B} A(x) \frac{1}{\sin \theta} \partial_{\varphi}- m \gamma^{0}_{B} B(x).
\end{equation}

\subsubsection{Asymptotic behavior of elements of the domain}

We introduce the projection $P_{s,n}$ from $\mathcal{H}$ to $\mathcal{H}_{s,n}$ and the operators $\tilde{H}_{m}^{s,n} = \tilde{H}_{m|\mathcal{H}_{s,n}}$, $H^{s,n,B}_{m} = H^{B}_{m|\mathcal{H}_{s,n}}$ for $(s,n) \in I$. We denote $\psi_{s,n} = P_{s,n}(\psi)$ with components $\psi^{s}_{i,n}$ for $i= 1,\cdots,4$. Furthermore, the domain of $H_{m}^{s,n,B}$ is given by:
\begin{enumerate}
\item[-] If $2ml \geq 1$:
\begin{equation*}
D\pare{H_{m}^{s,n,B}} = \left \{ \varphi_{s,n} \in \mathcal{H}_{s,n}; \hspace{1mm} H_{m}^{s,n,B} \varphi_{s,n} \in \mathcal{H}_{s,n} \right \}
\end{equation*}
\item[-] If $2ml < 1$, we add the condition that $\norme{\pare{\gamma^{1}_{B}+i} \varphi_{s,n} (x,.)}_{\mathcal{W}^{0}} = o\pare{\sqrt{-x}}$ when $x$ goes to $0$.
\end{enumerate}
We then have the:
\begin{lem}\label{LemDomain}
Let $\psi \in D\pare{\tilde{H}_{m}}$ and $\chi \in C^{\infty}_{0}\pare{\left ] -2\epsilon, 0 \right ]}$ such that $\chi = 1$ on $\left ]-\epsilon,0 \right]$ with $\epsilon>0$. Then $\chi \psi \in D\pare{H^{B}_{m}}$.
\end{lem}

\begin{proof}
Recall that the operator obtained by A. Bachelot in \cite{Bachelot} is given by \eqref{OpDirBach} where $F_{B}\pare{r} = 1 + \frac{r^{2}}{l^{2}}$. This operator has the same form as in \eqref{opé1}. Moreover, when $r>>r_{SAdS}$, $F_{B}$ and $F$ have the same behavior ($F$ is defined by $F\pare{r}=1 + \frac{r^{2}}{l^{2}} - \frac{2M}{r}$). We make the change of variable $r \to x $ where $\frac{dx}{dr} = F\pare{r}^{-1}$ and $F$ is defined on $]r_{SAdS},+\infty[$. We obtain:
\begin{equation*}
H_{m}^{B} = i \gamma^{0}_{B} \gamma^{1}_{B} g\pare{x} \partial_{x} + i \gamma^{0}_{B} \gamma^{1}_{B} \pare{\frac{F\pare{r}}{r} + \frac{F'\pare{r}}{4}} + \frac{3M}{2r^{2}} + A_{B}\pare{x} D_{S^{2}} - m \gamma^{0}_{B}B_{B}\pare{x}
\end{equation*}
where $r$ is understood as a function of $x$ and:
\begin{flalign*}
\hphantom{A} & g\pare{x} = 1 + \frac{2M}{l^{4}}\pare{-x}^{3} + o\pare{\pare{- x}^{3}}, \hspace{3mm} A_{B}\pare{x} = \frac{1}{l} + \frac{1}{2l^{3}} \pare{-x}^{2}+ o\pare{\pare{-x}^{2}}  \\
& B_{B}\pare{x} = \frac{l}{-x} + \frac{1}{6l} \pare{-x} + o \pare{-x}, \hspace{3mm} \frac{F\pare{r}}{r} = \frac{1}{-x} + \frac{2}{3l^{2}}\pare{-x} + o\pare{-x} \\
& F'\pare{r} = \frac{2}{-x} - \frac{2}{3l^{2}}\pare{-x} + o\pare{-x}
\end{flalign*}
when $x$ goes to $0$. Since $P_{s,n}(\chi \psi) = \chi \psi_{s,n}$, we have:
\begin{flalign}
H_{m}^{s,n,B} P_{s,n}(\chi \psi) & = g(x) \tilde{H}_{m}^{s,n} P_{s,n}(\chi \psi) + i \gamma^{0}_{B} \gamma^{1}_{B} \pare{\frac{F\pare{r}}{r} + \frac{F'\pare{r}}{4}} \pare{1-g\pare{x}} \chi \psi_{s,n} \notag \\
& \quad + \frac{3M}{2r^{2}} \chi \psi_{s,n} +  \gamma^{0}_{B} \gamma^{2}_{B} \pare{A_{B}(x)-g(x)A(x)} \pare{s + \frac{1}{2}} \chi \psi_{s,n} \notag \\
& \quad - m \gamma^{0}_{B}  \pare{B_{B}(x)-g(x)B(x)}\chi\psi_{s,n}  \label{2.42}
\end{flalign}
Since $\psi \in D(\tilde{H}_{m})$, $g$ is bounded in a neighborhood of $0$ and $\chi \in C^{\infty}_{0}\pare{]-1, 0]_{x}}$, the first term is in $L^{2}(x, dx)$. Using the behavior at $0$ of $g$, the terms $ A_{B}(x)-g(x)A(x)$, $B_{B}(x)-g(x)B(x) $ and $\pare{\frac{F\pare{r}}{r} + \frac{F'\pare{r}}{4}} \pare{1-g\pare{x}}$ are bounded near $0$. We deduce that $H_{m}^{s,n,B} P_{s,n}(\chi \psi) \in \mathcal{H}_{s,n}$. In particular, $\chi\psi_{s,n} \in D\pare{H_{m}^{s,n,B}}$.\\
\indent To be able to sum over $\pare{s,n}$, we need to know that $\pare{s+\frac{1}{2}}^{2} \norme{\pare{\chi \psi_{s,n}}}^{2}_{L^{2}\pare{-\frac{1}{2},0}}$ is summable. Since $\psi \in D\pare{\tilde{H}_{m}}$, $f=\tilde{H}_{m}\psi$ admits a decomposition $f = \underset{\pare{s,n}\in I}{\ds \sum} f^{s}_{n}$. We denote $f^{s}_{i,n}$ ($i=1,\cdots,4$) the components of $f^{s}_{n}$. We obtain four differential equations: 
\begin{flalign*}
 \hphantom{A} & i\overline{\chi \psi^{s}_{4,n}}\left(\chi\psi^s_{3,n} \right)'+ \left(s+\frac{1}{2} \right)A\pare{x} \abs{\chi\psi^s_{4,n}}^{2} - B\pare{x}\overline{\chi\psi^{s}_{4,n}}\chi\psi^s_{1,n}  = \overline{\chi\psi^{s}_{4,n}}f^{s}_{1,n}, \\
& -i \overline{\chi\psi^{s}_{3,n}}\left(\chi\psi^s_{4,n} \right)'+ \left(s+\frac{1}{2} \right)A\pare{x} \abs{\chi\psi^s_{3,n}}^{2}- B\pare{x} \overline{\chi\psi^{s}_{3,n}}\chi\psi^s_{2,n} =  \overline{\chi\psi^{s}_{3,n}}f^{s}_{2,n}, \\
& i \overline{\chi\psi^{s}_{2,n}} \left(\chi\psi^s_{1,n} \right)' + \left(s+\frac{1}{2} \right)A\pare{x} \abs{\chi\psi^s_{2,n}}^{2} + B\pare{x} \overline{\chi\psi^{s}_{2,n}} \chi\psi^s_{3,n} =  \overline{\chi\psi^{s}_{2,n}}f^{s}_{3,n}, \\
&-i\overline{\chi\psi^{s}_{1,n}}\left( \chi\psi^s_{2,n} \right)' + \left(s+\frac{1}{2} \right)A\pare{x} \abs{\chi\psi^s_{2,n}}^{2} + B\pare{x} \overline{\chi\psi^{s}_{1,n}}\chi\psi^s_{4,n} = \overline{\chi\psi^{s}_{1,n}}f^{s}_{4,n}.
\end{flalign*}
where we have multiply by $\overline{\chi\psi^{s}_{j,n}}$ for $j=1,\cdots,4$. Adding these equations and taking the real part, we obtain:
\begin{flalign}
\hphantom{A} & \frac{d}{dx} \Im \pare{\chi\psi^{s}_{1,n}\overline{\chi\psi^{s}_{2,n}} + \chi\psi^{s}_{3,n}\overline{\chi\psi^{s}_{4,n}}}
+ \pare{s +\frac{1}{2}} A \pare{x} \underset{j=1}{\overset{4}{\ds \sum}} \abs{\chi\psi^s_{j,n}}^{2} \notag \\
& = \Re \pare{\overline{\chi\psi^{s}_{4,n}}f^{s}_{1,n} + \overline{\chi\psi^{s}_{3,n}}f^{s}_{2,n} + \overline{\chi\psi^{s}_{2,n}}f^{s}_{3,n} + \overline{\chi\psi^{s}_{1,n}}f^{s}_{4,n}}. \label{2.48}
\end{flalign}
Using that:
\begin{equation}
\underset{x \to 0}{\lim} \Im \pare{\chi\psi^{s}_{1,n}\overline{\chi\psi^{s}_{2,n}} + \chi\psi^{s}_{3,n}\overline{\chi\psi^{s}_{4,n}}} = 0.
\end{equation}
and that $\chi\psi^{s}_{j,n}$ is $0$ at $1$ for all $j=1,\cdots,4$, we obtain:
\begin{flalign*}
\pare{s+\frac{1}{2}} \ds \int_{-\frac{1}{2}}^{0} A \pare{x} \underset{j=1}{\overset{4}{\ds \sum}} \abs{\chi\psi^s_{j,n}}^{2} dx & = \ds \int_{-\frac{1}{2}}^{0}  \Re \left ( \overline{\chi\psi^{s}_{4,n}}f^{s}_{1,n} + \overline{\chi\psi^{s}_{3,n}}f^{s}_{2,n} + \overline{\chi\psi^{s}_{2,n}}f^{s}_{3,n} \right. \\
& \quad \left . + \overline{\chi\psi^{s}_{1,n}}f^{s}_{4,n} \right ) dx.
\end{flalign*}
After some calculations, this gives:
\begin{equation*}
\pare{s+\frac{1}{2}}^{2} \ds \int_{-\frac{1}{2}}^{0} \pare{2 l A \pare{x}-1} \underset{j=1}{\overset{4}{\ds \sum}} \abs{\chi\psi^s_{j,n}}^{2} dx \leq \ds \int_{-\frac{1}{2}}^{0} \underset{j=1}{\overset{4}{\ds \sum}} l^{2} \abs{f^{s}_{j,n}}^{2} dx.
\end{equation*}
Using the asymptotic behavior of $A$ (see \eqref{CompAsymptA}), we can prove that $2lA\pare{x} - 1 \geq 1$ on the support of $\chi$ (for $\epsilon$ sufficiently small). Finally, we obtain:
\begin{equation}
\pare{s+\frac{1}{2}}^{2} \ds \int_{-\frac{1}{2}}^{0}  \underset{j=1}{\overset{4}{\ds \sum}} \abs{\chi\psi^s_{j,n}}^{2} dx \leq l^{2} \ds \int_{-\frac{1}{2}}^{0} \underset{j=1}{\overset{4}{\ds \sum}}  \abs{f^{s}_{j,n}}^{2} dx
\end{equation}
and the right hand side is summable because $f \in \mathcal{H}$. This gives the lemma.
\end{proof}
We can know apply Theorem \ref{5.1} to $\chi \psi$ and obtain the asymptotic behavior of $\psi$:
\begin{prop} \label{prop3.3}
If $2ml > 1$, we have: 
\begin{equation}
|| \psi(\zeta,.)||_{L^2(S^2)} = O \left( \sqrt{-x} \right), \hspace{2mm} x \to 0.\label{705}
\end{equation}
\\If $2ml= 1$, we have:
\begin{equation}
||\psi(x,.)||_{L^2(S^2)} = O \left(\sqrt{\pare{-x}ln\pare{-x}} \right), \hspace{2mm} x \to 0. \label{706}
\end{equation}
\\If $0 < 2ml < 1$, there exists functions $\psi_{-} \in W^{\frac{1}{2}}_{-}$, $\chi_{-} \in W^{\frac{1}{2}}_{+}$, $\psi_{+},\hspace{0.5mm} \chi_{+} \in L^2(S^2)$ and $\phi \in C^0 \left(]-\infty,0]_x; L^2(S^2;\C^4) \right)$ satisfying \eqref{507} and \eqref{508} with $\frac{\pi}{2} - \zeta$ replaced by $\pare{-x}^{l}$.
\\Conversely, for all $\psi_{-} \in W^{\frac{1}{2}+m}_{-}$, $\chi_{-} \in W^{\frac{1}{2}+m}_{+}$, $\psi_{+} \in W^{\frac{1}{2}-m}_{-}$, $\chi_{+} \in W^{\frac{1}{2}-m}_{+}$, there exists $\psi \in D(H_m)$ satisfying \eqref{507} and \eqref{508} with the same replacement as before.
\end{prop}

\begin{rem}
By restriction to $\mathcal{H}_{s,n}$, we obtain the same result for $s,n$ fixed. Moreover, if $\varphi_{s,n} \in D\pare{H_{m}^{s,n}}$, then it is in $H^{1}\pare{]-\infty,-c[}$ for a constant $c>0$. We conclude that $\varphi_{s,n} \in C^{0}\pare{]-\infty,-c[} \cap L^{2}\pare{]-\infty,-c[}$ and:
\begin{equation}
\norme{\varphi_{s,n}\pare{x,.}}_{\mathcal{W}^{0}} \to 0, \hspace{2mm} x \to - \infty.
\end{equation}
\end{rem}

\subsubsection{Description of the domain}

We now give a description of the domain of $H_{m}$ for fixed $\pare{s,n} \in I$. Recall that $H_{m}$ and $\tilde{H}_{m}$ are linked by a unitary transform, so it does not change the norm of the observables. We obtain:
\begin{flalign}
-D\pare{H_{m}^{s,n}}&= \left \{ \psi_{s,n} \in \mathcal{H}_{s,n} ; \hspace{1mm} H_{m}^{s,n}\psi_{s,n} \in \mathcal{H}_{s,n} \right \}, \hspace{2mm} \text{if} \hspace{2mm} 2ml \geq 1;\\
-D\pare{H_{m}^{s,n}} & =\left \{ \psi_{s,n} \in \mathcal{H}_{s,n} ; \hspace{1mm} H_{m}^{s,n} \psi_{s,n} \in \mathcal{H}_{s,n}, \hspace{1mm}\psi_{s,n}  =\pare{-x}^{-ml} \begin{pmatrix}
\psi^{s}_{-,n} (\theta,\varphi) \\
i\chi^{s}_{-,n} (\theta,\varphi) \\
-\psi^{s}_{-,n} (\theta,\varphi)\\
i\chi^{s}_{-,n} (\theta,\varphi)
\end{pmatrix} \right .  \notag \\
& \quad \left . + \phi^{s}_{n}\pare{x,\theta,\varphi}, \norme{\phi^{s}_{n}\pare{x,.,.}}_{\mathcal{W}^{0}}= o\pare{\sqrt{-x}} \right \}, \hspace{2mm} \text{if} \hspace{2mm} 2ml < 1. \label{Dom2mlPetit}
\end{flalign}

\subsection{Self-adjointness for fixed harmonic}

In this section, $s$ and $n$ are fixed.

\subsubsection{The case \texorpdfstring{$2ml\geq 1$}{2ml>1}}

\begin{lem}[Elliptic estimate]
We suppose that $2ml>1$. Then, there exists a constant $C>0$ such that, for all $\varphi \in C^{\infty}_{0}\pare{]-\infty,0[}$, we have:
\begin{equation}
\norme{-i \partial_{x} \varphi}^{2} \leq C \pare{\norme{H_{m}^{s,n} \varphi}^{2} + \norme{\varphi}^{2}}
\end{equation}
\end{lem}

\begin{proof}
We write $D_{x} = -i\partial_{x}$ and $\Gamma^{1} = - \gamma^{0}\gamma^{1}$. Recall that:
\begin{equation*}
H_{m}^{s,n} = \Gamma^{1} D_{x} +\pare{s+\frac{1}{2}} A\pare{x} \gamma^{0}\gamma^{2} - m B\pare{x} \gamma^{0}.
\end{equation*}
We will often denote $V\pare{x} = \pare{s+\frac{1}{2}} A\pare{x} \gamma^{0}\gamma^{2} - m B\pare{x} \gamma^{0}$. Choose a partition of unity $\chi_{1},\chi_{2}$ such that $\chi_{1}+\chi_{2} = 1$, $\supp \pare{\chi_{1}} \subset ]-\infty,-\epsilon[$ and $\chi_{1} =1$ on $]-\infty,-2\epsilon[$, $\supp \pare{\chi_{2}} \subset ]-2\epsilon,0[$ and $\chi_{2} =1$ on $]-\epsilon,0[$. We choose $\epsilon >0$ sufficiently small so that, if $\gamma^{5}_{B}$ and $P$ are unitary matrices defined as in \eqref{UnitTransH}, $\gamma^{5} P^{-1}\chi_{2} \varphi \in D\pare{H_{m}^{B}}$ when $\varphi \in D\pare{H_{m}^{s,n}}$ (it is possible by lemma \ref{LemDomain}). Recall that $m$ is the mass of the field and $l$ correspond to the cosmological constant. Using equation $III.32$ in theorem $III.4$ of \cite{Bachelot}, \eqref{UnitTransH} and \eqref{2.42}, we obtain:
\begin{equation*}
\norme{D_{x}  \pare{\gamma^{5}_{B} P^{-1} \chi_{2} \varphi}} \leq C_{m,l} \norme{g\pare{x}H_{m}^{s,n} \pare{\gamma^{5}_{B} P^{-1}\chi_{2} \varphi}} + \tilde{C}_{m,l} \norme{\chi_{2} \varphi},
\end{equation*}
where $C_{m,l}$ and $\tilde{C}_{m,l}$ are constants depending on $m$ and $l$. Since $\gamma^{5}_{B} P^{-1}$ is unitary and commute with $D_{x}$ and $g$ is bounded near $0$, we obtain:
\begin{equation}
\norme{D_{x} \pare{\chi_{2} \varphi}} \leq C_{m,l,\epsilon} \norme{H_{m}^{s,n} \pare{\chi_{2} \varphi}} + \tilde{C}_{m,l} \norme{\chi_{2} \varphi}.
\end{equation}
On the other hand, with $C_{V,\epsilon}$ constant, we have:
\begin{equation*}
\norme{D_{x}  \pare{ \chi_{1} \varphi}}  \leq \norme{H_{m}^{s,n} \pare{ \chi_{1} \varphi}} + C_{V,\epsilon} \norme{\varphi}.
\end{equation*}
Since $\chi_{1},\chi_{2}$ commute with $V$ and are bounded as are their derivatives, we obtain:
\begin{flalign*}
\norme{D_{x} \varphi}^{2} & \leq C \pare{\norme{H_{m}^{s,n} \pare{ \chi_{1} \varphi}}^{2} + \norme{H_{m}^{s,n} \pare{\chi_{2} \varphi}}^{2}} + C'\norme{\varphi}^{2} \\
& \leq \tilde{C} \norme{H_{m}^{s,n} \varphi}^{2} + \tilde{C}' \norme{\varphi}^{2}.
\end{flalign*}

\end{proof}

\begin{prop} \label{PropHaa2mlGr}
For $2ml \geq 1$, the operator $\tilde{H}_{m}^{s,n}$ is essentially self-adjoint on \\
$C^{\infty}_{0}\pare{\left ] -\infty, 0 \right [}$. Moreover, if $2ml>1$, the domain of this operator is given by $H^{1}_{0}\pare{]-\infty,0[}$.
\end{prop}

\begin{proof}
Recall that:
\begin{equation*}
\tilde{H}_{m}^{s,n} = i \gamma^{0}_{B} \gamma^{1}_{B} \partial_{x} + \gamma^{0}_{B} \gamma^{2}_{B} \harm A(x) - m \gamma^{0}_{B} B(x) 
\end{equation*}
with domain $D\pare{\tilde{H}_{m}^{s,n}} = \left \{ \psi_{s,n} \in \mathcal{H}_{s,n} ; \hspace{1mm} \tilde{H}_{m}^{s,n} \psi_{s,n} \in \mathcal{H}_{s,n} \right \}$ and if $\psi_{s,n} \in D\pare{\tilde{H}_{m}^{s,n}}$, then we have:
\begin{flalign}
\hphantom{A} & \norme{\psi_{s,n}(x,.)}_{L^{2}(S^{2})} = O\pare{\sqrt{\pare{-x}}} , \hspace{2mm} x \to 0, \hspace{2mm} \text{if} \hspace{2mm} 2ml >1;\\
& \norme{\psi_{s,n}(x,.)}_{L^{2}(S^{2})} = O \pare{\sqrt{x \ln \pare{-x}}}, \hspace{2mm} x \to 0, \hspace{2mm} \text{if} \hspace{2mm} 2ml=1;\label{Asymp2mlGr}\\
& \norme{\psi_{s,n}\pare{x,.}}_{\mathcal{W}^{0}} \to 0, \hspace{2mm} x \to - \infty.
\end{flalign}
Let us prove that $\tilde{H}_{m}^{s,n}$ is symmetric on its domain. We remark that $\pare{\gamma^{0}_{B} \gamma^{2}_{B}}^{*}= \gamma^{0}_{B}\gamma^{2}_{B}$, $\pare{\gamma^{0}_{B} \gamma^{1}_{B}}^{*} = \gamma^{0}_{B} \gamma^{1}_{B}$ and $\pare{\gamma^{0}_{B}}^{*} = \gamma^{0}_{B}$. So:
\begin{flalign*}
 \prods{ \gamma^{0}_{B} \gamma^{2}_{B} A\pare{x} \pare{s + \frac{1}{2}} \phi_{s,n}}{\psi_{s,n}}_{\mathcal{H}_{s,n}} &= \prods{  \phi_{s,n}}{\pare{\gamma^{0}_{B} \gamma^{2}_{B}} A\pare{x} \pare{s + \frac{1}{2}}\psi_{s,n}}_{\mathcal{H}_{s,n}}, \\
 \prods{\gamma^{0}_{B} B\pare{x} \phi_{s,n}}{\psi_{s,n}}_{\mathcal{H}_{s,n}} & =  \prods{ \phi_{s,n}}{\gamma^{0}_{B} B\pare{x} \psi_{s,n}}_{\mathcal{H}_{s,n}}.
\end{flalign*}
Thus, in the calculation of $\prods{\tilde{H}_{m}^{s,n} \phi_{s,n}}{\psi_{s,n}}_{\mathcal{H}_{s,n}} - \prods{\phi_{s,n}}{\tilde{H}_{m}^{s,n} \psi_{s,n}}_{\mathcal{H}_{s,n}} $, it remains only the boundary term due to integration by parts. Using \eqref{Asymp2mlGr}, this gives the symmetry of our operator on its domain.
\\ \indent We then use the same trick as in \cite{Bachelot}. Let us consider a new operator $H$ with the same expression as $\tilde{H}_{m}^{s,n}$ but defined on $D(H) = C^{\infty}_{0}\pare{\left ] - \infty, 0 \right [}$. Then $H^{*}$ is $\tilde{H}_{m}^{s,n}$ with domain $D(H^{*})$ included in $D\pare{ \tilde{H}_{m}^{s,n}}$. Let $\phi_{\pm} \in \ker \pare{H^{*} \pm i Id}$. Then, using the symmetry of $\tilde{H}_{m}^{s,n}$ and that $H^{*}=\tilde{H}_{m}^{s,n}$, we have:
\begin{equation}
0 = \prods{\tilde{H}_{m}^{s,n} \phi_{\pm}}{\phi_{\pm}} - \prods{\phi_{\pm}}{\tilde{H}_{m}^{s,n} \phi_{\pm}} = \prods{H^{*} \phi_{\pm}}{\phi_{\pm}} - \prods{\phi_{\pm}}{H^{*} \phi_{\pm}} = \mp 2i \norme{\phi_{\pm}}^{2}_{\mathcal{H}_{s,n}}.
\end{equation}
We conclude that $\phi_{\pm} = 0$. This proves that $\tilde{H}_{m}^{s,n}$ is essentially self-adjoint on $C^{\infty}_{0}\pare{\left ] -\infty, 0 \right [}$.\\
For the last part, using the last lemma, we see that, for $2ml > 1$, we have: $D\pare{H_{m}^{s,n}} \subset H^{1}_{0}\pare{]-\infty,0[}$. Indeed, if we take $\varphi \in D\pare{H_{m}^{s,n}}$, it is the limit of a sequence $\pare{\varphi_{n}}_{n\in \N} \in \pare{C^{\infty}_{0}}^{\N}$ for the graph norm. The last lemma gives that $\partial_{x}\varphi_{n}$ is a Cauchy sequence so that it converges in $H_{0}^{1}$. A distribution argument gives that this limit is $\partial_{x} \varphi$ which is in $L^{2}$ by the lemma.\\
Moreover, we have $H_{m}^{s,n} = i \gamma^{0} \gamma^{1} \partial_{x} + \gamma^{0} \gamma^{2} \pare{s+\frac{1}{2}} A\pare{x} - m \gamma^{0} B\pare{x}$ with $A$ having the behavior as in \eqref{CompAsymptA} and $B$ as in \eqref{CompAsymptB}. Using the fact that $B$ and $B_{B}$ have the same behavior when $x \to 0$ and the unitary transform, we can use the proof of Theorem $III.4$ in \cite{Bachelot} to prove a Hardy type inequality of the form:
\begin{equation}
\norme{B \chi_{2}^{2} \varphi}  \leq c \pare{ \norme{\varphi} + \norme{-i \partial_{x} \varphi}}.
\end{equation}
Using the fact that $A$ is bounded, we have a similar estimate for $\gamma^{0} \gamma^{2} \pare{s+\frac{1}{2}} A\pare{x} - m \gamma^{0} B\pare{x}$. Thus $H^{1}_{0} \subset D\pare{H_{m}^{s,n}}$. This proves the proposition.
\end{proof}

\subsubsection{The case \texorpdfstring{$2ml<1$}{2ml<1}}

Recall that if $0<2ml<1$, then, for all $\psi_{s,n} \in D\pare{\tilde{H}_{m}^{s,n}}$, there exists functions $\psi_{-} \in W^{\frac{1}{2}}_{-}$, $\chi_{-} \in W^{\frac{1}{2}}_{+}$, $\psi_{+},\hspace{0.5mm} \chi_{+} \in L^2(S^2)$ and $\phi \in C^0 \left([0,\frac{\pi}{2}]_x; L^2(S^2;\C^4) \right)$ such that $\norme{\sigma^{s}_{n}(x,\theta,\varphi)}_{\mathcal{W}^{0}} = o \pare{\sqrt{\pare{-x}}}$ as $x$ goes to $0$ and:
\begin{flalign}
 \psi_{s,n}\pare{x,\theta,\varphi} & = \pare{-x}^{-ml} \begin{pmatrix}
\psi^{s}_{-,n} (\theta,\varphi) \\
\chi^{s}_{-,n} (\theta,\varphi) \\
-i\psi^{s}_{-,n} (\theta,\varphi)\\
i\chi^{s}_{-,n} (\theta,\varphi)
\end{pmatrix} + \pare{-x}^{ml} \begin{pmatrix}
\psi^{s}_{+,n} (\theta,\varphi) \\
\chi^{s}_{+,n} (\theta,\varphi) \\
i\psi^{s}_{+,n} (\theta,\varphi)\\
-i\chi^{s}_{+,n} (\theta,\varphi)
\end{pmatrix}  + \sigma^{s}_{n}(x,\theta,\varphi), \notag \\
& := \pare{-x}^{-ml} \Psi^{s}_{-,n}(\theta,\varphi) + \pare{-x}^{ml} \Psi^{s}_{+,n}(\theta,\varphi) + \sigma^{s}_{n}(x,\theta,\varphi) \label{ExprPsi2mlPetit}
\end{flalign}
We denote by $\tilde{H}^{MIT}_{s,n}$ the operator $\tilde{H}_{m}^{s,n}$ with domain:
\begin{equation}
D(\tilde{H}^{MIT}_{s,n}) =\left \{ \psi_{s,n} \in \mathcal{H}_{s,n}; \hspace{2mm} \tilde{H}_{m}^{s,n} \psi_{s,n} \in \mathcal{H}_{s,n}, \hspace{2mm} \psi^{s}_{+,n} = \chi^{s}_{+,n}=0 \right \}.
\end{equation}
which is a consequence of the discussion after proposition $VI.2$ in \cite{Bachelot}. We have the:

\begin{prop}
The operator $\tilde{H}^{MIT}_{s,n}$ is self-adjoint on $D\pare{\tilde{H}^{MIT}_{s,n}}$.
\end{prop}

\begin{proof}
Let $\phi_{s,n},\psi_{s,n} \in D(\tilde{H}^{MIT}_{s,n})$. As in the proof of proposition \ref{PropHaa2mlGr}, when calculating 
\begin{equation*}
\prods{\tilde{H}^{MIT}_{s,n} \phi_{s,n}}{\psi_{s,n}}_{\mathcal{H}_{s,n}} - \prods{\phi}{\tilde{H}^{MIT}_{s,n} \psi}_{\mathcal{H}_{s,n}},
\end{equation*}
only boundary values of $\phi_{s,n},\psi_{s,n}$ are left. Using that
\begin{flalign*}
\hphantom{A} & \phi_{s,n}\pare{x,\theta,\varphi}= \pare{-x}^{-ml} \begin{pmatrix}
\phi^{s}_{-,n} (\theta,\varphi) \\
\xi^{s}_{-,n} (\theta,\varphi) \\
-i\phi^{s}_{-,n} (\theta,\varphi)\\
i\xi^{s}_{-,n} (\theta,\varphi)
\end{pmatrix} 
+ \varphi^{s}_{n}\pare{x,\theta,\varphi}:= \pare{-x}^{-ml} \Phi_{-,n}^{s}\pare{\theta,\varphi} + \varphi^{s}_{n}\\
& \norme{\varphi_{s,n}}_{L^{2}(S^{2})} = o\pare{\sqrt{\pare{-x}}}, \hspace{1mm} x \to 0,
\end{flalign*}
and a similar formula for $\psi_{s,n}$ with $\Phi^{s}_{-,n},\varphi^{s}_{n}$ replaced by $\Psi^{s}_{-,n},\sigma^{s}_{n}$ respectively, we can calculate these boundary values in a neighbourhood of $0$ (with the scalar product being the one of $L^{2}\pare{S^{2}}$ and we write $\phi_{s,n}\pare{x}$ for $\phi_{s,n}\pare{x,.}$):
\begin{equation*}
\prods{\phi_{s,n}\pare{x}}{\gamma^{0}_{B} \gamma^{1}_{B}\psi_{s,n}\pare{x}} = \pare{-x}^{-ml} \pare{ \prods{\Phi^{s}_{-,n} }{\sigma^{s}_{n}\pare{x}} + \prods{\varphi^{s}_{n}\pare{x}}{\Psi^{s}_{-,n}}} +\prods{\varphi^{s}_{n}\pare{x}}{\sigma^{s}_{n}\pare{x}}_{\mathcal{W}^{0}}.
\end{equation*}
Indeed, $\gamma^{0}_{B} \gamma^{1}_{B}$ arranges the terms such that $\prods{\pare{-x}^{-ml} \Phi^{s}_{-,n} }{\pare{-x}^{-ml} \Psi^{s}_{-,n}}_{\mathcal{W}^{0}} = 0$. Using the behavior at $0$ of $\varphi^{s}_{n},\sigma^{s}_{n}$ and at $-\infty$ of $\phi_{s,n},\psi_{s,n}$, we deduce that $\tilde{H}^{MIT}_{s,n}$ is symmetric. \\
\indent Let $\psi_{s,n} \in D(\tilde{H}_{m}^{s,n,MIT,*})$. Then, since $D(\tilde{H}_{m}^{s,n,MIT,*}) \subset D(\tilde{H}_{m}^{s,n})$, $\psi$ admits a decomposition, in a neighbourhood of $0$, as in \eqref{ExprPsi2mlPetit}. Moreover, $\tilde{H}_{m}^{s,n,MIT,*} = \tilde{H}_{m}^{s,n}$ on $D\pare{\tilde{H}_{m}^{s,n,MIT,*}}$ (using distributions). We have:
\begin{equation*}
0 = \prods{\tilde{H}_{m}^{s,n,MIT}\phi_{s,n}}{\psi_{s,n}} - \prods{\phi_{s,n}}{\tilde{H}_{m}^{s,n,MIT,*} \psi_{s,n}} = \underset{x \to 0}{\lim} \prods{\pare{-x}^{-ml}\Phi^{s}_{-,n} }{\pare{-x}^{ml} \Psi^{s}_{+,n}},
\end{equation*}
for all $\phi_{s,n} \in D\pare{\tilde{H}_{m}^{s,n,MIT}}$ and $\psi_{s,n} \in D\pare{\tilde{H}_{m}^{s,n,MIT,*}}$. In other words, we have:
\begin{equation}
2 \prods{\begin{pmatrix}
\phi^{s}_{-,n}\\
\xi^{s}_{-,n}
\end{pmatrix}}{\begin{pmatrix}
\psi^{s}_{+,n} \\
\chi^{s}_{+,n}
\end{pmatrix}}=0.
\end{equation}
But, for all $\phi^{s}_{-,n},\xi^{s}_{-,n} \in C^{\infty}_{0} ( Y_{s,n})$, we can find $\phi \in D(\tilde{H}_{m}^{s,n,MIT})$ admitting these components as coordinates. Thus $\psi^{s}_{+,n}=\chi^{s}_{+,n}=0$. We conclude that $D\pare{\tilde{H}_{m}^{s,n,MIT,*}} \subset D\pare{\tilde{H}_{m}^{s,n,MIT}}$ and that $\tilde{H}_{m}^{s,n,MIT}$ is self-adjoint on his domain.

\end{proof}

\subsection{Self-adjointness of \texorpdfstring{$\tilde{H}_{m}$}{Hm}}

\subsubsection{The case \texorpdfstring{$2ml\geq 1$}{2ml>1}}

We equip $\tilde{H}_{m}$ with the domain:
\begin{flalign}
D(\tilde{H}_{m}) & = \left \{u \in \mathcal{H}; \hspace{2mm} \tilde{H}_{m} u \in \mathcal{H} \right \} \notag \\
& = \left \{ \underset{(s,n)\in I}{\sum} \begin{pmatrix}
u^{s}_{1,n} T^{s}_{-\frac{1}{2},n} \\
u^{s}_{2,n} T^{s}_{\frac{1}{2},n} \\
u^{s}_{3,n} T^{s}_{-\frac{1}{2},n} \\
u^{s}_{4,n} T^{s}_{\frac{1}{2},n}
\end{pmatrix}; \hspace{2mm} \forall (s,n) \in I, \hspace{2mm} u^{s}_{n} \in L^{2}\pare{\left ]-\infty, 0 \right [_{x},dx}, \right . \notag \\
& \quad \left . \hspace{1mm} \tilde{H}_{m}^{s,n} \begin{pmatrix}
u^{s}_{1,n} T^{s}_{-\frac{1}{2},n} \\
u^{s}_{2,n} T^{s}_{\frac{1}{2},n} \\
u^{s}_{3,n} T^{s}_{-\frac{1}{2},n} \\
u^{s}_{4,n} T^{s}_{\frac{1}{2},n}
\end{pmatrix} \in L^{2},
\hspace{1mm} \underset{(s,n) \in I}{\sum} \norme{(\tilde{H}_{m}^{s,n} \pm i) \begin{pmatrix}
u^{s}_{1,n} T^{s}_{-\frac{1}{2},n} \\
u^{s}_{2,n} T^{s}_{\frac{1}{2},n} \\
u^{s}_{3,n} T^{s}_{-\frac{1}{2},n} \\
u^{s}_{4,n} T^{s}_{\frac{1}{2},n}
\end{pmatrix}}^{2}_{L^{2}} < \infty \right \}.
\end{flalign}
We then have:
\begin{prop} \label{proposition2.6}
Suppose that $2ml \geq 1$. Then the operator $\tilde{H}_{m}$ is self-adjoint on its domain.
\end{prop}

\begin{proof}
$\tilde{H}_{m}$ is symmetric. Indeed, let $\varphi,\psi \in D\pare{\tilde{H}_{m}}$. We can decompose $\varphi = \underset{(s,n) \in I}{\sum} \varphi_{s,n}$ and the same for $\psi$. Then:
\begin{equation}
\prods{\tilde{H}_{m} \varphi}{\psi} = \underset{(s,n) \in I}{\sum} \prods{\tilde{H}_{m}^{s,n} \varphi_{s,n}}{\psi_{s,n}} = \underset{(s,n) \in I}{\sum} \prods{\varphi_{s,n}}{\tilde{H}_{m}^{s,n} \psi_{s,n}} 
= \prods{\varphi}{\tilde{H}_{m} \psi}
\end{equation}
since $\tilde{H}_{m}^{s,n}$ is symmetric. We can prove that $\tilde{H}_{m}$ is closed in the same way.
Let $x = \underset{(s,n) \in I}{\sum} x_{s,n} \in \mathcal{H}$. Since $ \tilde{H}_{m}^{s,n}$ is self-adjoint, there exists $y_{s,n} \in D\pare{\tilde{H}_{m}^{s,n}}$ such that $(\tilde{H}_{m} \pm i) y_{s,n} = (\tilde{H}_{m}^{s,n}\pm i)y_{s,n}= x_{s,n}$. Thus $x = \underset{(s,n) \in I}{\sum} (\tilde{H}_{m} \pm i) y_{s,n}=\pare{\tilde{H}_{m} \pm i} y$ where $y=\underset{(s,n) \in I}{\sum} y_{s,n} \in D\pare{\tilde{H}_{m}}$ since:
\begin{equation*}
\underset{(s,n) \in I}{\sum} \norme{\tilde{H}_{m}^{s,n}y_{s,n}}^{2} + \norme{y_{s,n}}^{2} = \underset{(s,n) \in I}{\sum} \norme{(\tilde{H}_{m}^{s,n}\pm i)y_{s,n}}^{2} = \underset{(s,n) \in I}{\sum}\norme{x_{s,n}}^{2} < \infty.
\end{equation*}
Consequently, $(y_{s,n})_{(s,n) \in I}$ is summable and $x \in Im(\tilde{H}_{m} \pm i)$ so $Im(\tilde{H}_{m} \pm i) = \mathcal{H}$ and $\tilde{H}_{m}$ is self-adjoint.

\end{proof}

\subsubsection{The case \texorpdfstring{$2ml<1$}{2ml<1}}

Let us denote $\tilde{H}_{m}^{MIT}$ the operator $\tilde{H}_{m}$ with domain:
\begin{equation*}
D\pare{\tilde{H}_{m}^{MIT}} = \left \{ \underset{(s,n)\in I}{\sum} \begin{pmatrix}
u^{s}_{1,n} T^{s}_{-\frac{1}{2},n} \\
u^{s}_{2,n} T^{s}_{\frac{1}{2},n} \\
u^{s}_{3,n} T^{s}_{-\frac{1}{2},n} \\
u^{s}_{4,n} T^{s}_{\frac{1}{2},n}
\end{pmatrix}; \hspace{2mm} \forall (s,n) \in I, \hspace{2mm} u^{s}_{n} \in L^{2}\pare{\left ]-\infty, 0 \right [_{x},dx}, \right .
\end{equation*} 
\begin{flalign}
\hphantom{A} & \left . \tilde{H}_{m}^{s,n} \begin{pmatrix}
u^{s}_{1,n} T^{s}_{-\frac{1}{2},n} \\
u^{s}_{2,n} T^{s}_{\frac{1}{2},n} \\
u^{s}_{3,n} T^{s}_{-\frac{1}{2},n} \\
u^{s}_{4,n} T^{s}_{\frac{1}{2},n}
\end{pmatrix} \in L^{2},
\hspace{1mm} \underset{(s,n) \in I}{\sum} \norme{(\tilde{H}_{m}^{s,n} \pm i) \begin{pmatrix}
u^{s}_{1,n} T^{s}_{-\frac{1}{2},n} \\
u^{s}_{2,n} T^{s}_{\frac{1}{2},n} \\
u^{s}_{3,n} T^{s}_{-\frac{1}{2},n} \\
u^{s}_{4,n} T^{s}_{\frac{1}{2},n}
\end{pmatrix}}^{2}_{L^{2}} < \infty \right . \notag \\
& \quad \left .  \underset{(s,n) \in I}{\sum} \norme{ \pare{\gamma^{1}_{B} + i} \begin{pmatrix}
u^{s}_{1,n} T^{s}_{-\frac{1}{2},n} \\
u^{s}_{2,n} T^{s}_{\frac{1}{2},n} \\
u^{s}_{3,n} T^{s}_{-\frac{1}{2},n} \\
u^{s}_{4,n} T^{s}_{\frac{1}{2},n}
\end{pmatrix}}^{2}_{L^{2}}= o\pare{\sqrt{-x}}, \hspace{1mm} x \to 0  \right \} 
\end{flalign}

\begin{prop}
Suppose that $2ml < 1$. Then the operator $\tilde{H}_{m}^{MIT}$ is self-adjoint with domain $D\pare{\tilde{H}_{m}^{MIT}}$.
\end{prop}

\begin{proof}
Let us remark that, if the boundary condition is fulfilled for $\phi \in D\pare{\tilde{H}_{m}^{MIT}}$, then it is fulfilled for $\phi_{s,n} \in D\pare{\tilde{H}_{m}^{s,n,MIT}}$. We can now prove, as in the proof of proposition \ref{proposition2.6}, that $\tilde{H}_{m}^{MIT}$ is symmetric on its domain. Show that $\tilde{H}_{m}^{MIT}$ is closed will require more effort. Choose a sequence $\pare{\psi_{j}}_{j\in\N}$ of elements of $D\pare{\tilde{H}_{m}^{MIT}}$ such that $\psi_{j} \to \psi$ and $\tilde{H}_{m}^{MIT} \psi_{j} \to \varphi $ where $\psi,\varphi \in \mathcal{H}$ and the convergence is understood in the norm of $\mathcal{H}$. Using distributions, we have $\tilde{H}_{m}^{MIT} \psi = \varphi \in \mathcal{H}$ and we have to show that $\psi$ satisfies the boundary condition. We can write:
\begin{equation}
\psi_{j} = \underset{(s,n)\in I}{\sum} \psi_{j}^{s,n}, \hspace{5mm} \psi =  \underset{(s,n)\in I}{\sum} \psi^{s,n}, \hspace{5mm} \varphi =  \underset{(s,n)\in I}{\sum} \varphi^{s,n},
\end{equation}
and we obtain:
\begin{equation*}
\psi_{j}^{s,n} \to \psi^{s,n}; \hspace{5mm} \tilde{H}_{m}^{s,n,MIT} \psi_{j}^{s,n} \to \varphi^{s,n} 
\end{equation*}
in the norm of $\mathcal{H}_{s,n}$. Thus, $\psi^{s,n} \in D\pare{\tilde{H}_{m}^{s,n,MIT}}$ since $\tilde{H}_{m}^{s,n,MIT}$ is closed and $\psi^{s,n}$ admits a decomposition as in \eqref{Dom2mlPetit} where:
\begin{equation*}
 \underset{(s,n)\in I}{\sum} \pare{\norme{\phi_{1,n}^{s}}^{2}_{\mathcal{W}^{0}} + \norme{\phi_{2,n}^{s}}^{2}_{\mathcal{W}^{0}}+ \norme{\phi_{3,n}^{s}}^{2}_{\mathcal{W}^{0}}+ \norme{\phi_{4,n}^{s}}^{2}_{\mathcal{W}^{0}}} = o\pare{-x}
\end{equation*}
when $x$ goes to $0$, using the proof of theorem $V.1$ in \cite{Bachelot} and the fact that $\varphi$ is in the natural domain of $H_{m}$. Since $\gamma^{1}_{B}+i$ eliminates the terms containing $\pare{-x}^{-ml}$, we have:
\begin{equation}
\norme{\pare{\gamma^{1}_{B}+i} \varphi\pare{x,.}}^{2}_{L^{2}\pare{S^{2}}} \leq C \underset{(s,n)\in I}{\sum} \pare{\norme{\phi_{1,n}^{s}}^{2}_{\mathcal{W}^{0}} + \norme{\phi_{2,n}^{s}}^{2}_{\mathcal{W}^{0}}+ \norme{\phi_{3,n}^{s}}^{2}_{\mathcal{W}^{0}}+ \norme{\phi_{4,n}^{s}}^{2}_{\mathcal{W}^{0}} }
\end{equation}
where the last term is $o\pare{-x}$. This proves that the boundary condition is fulfilled and that the operator $\tilde{H}_{m}^{MIT}$ is closed. To prove the self-adjointness of $\tilde{H}_{m}^{MIT}$, we follow the same argument as in proposition \ref{proposition2.6} where we have to prove that $y=\underset{(s,n) \in I}{\sum} y_{s,n} \in D\pare{\tilde{H}_{m}^{MIT}}$. The only difference is that the boundary condition has to be fulfilled. Since $y_{s,n} \in D\pare{\tilde{H}_{m}^{s,n,MIT}}$, we can decompose $y_{s,n}$ as for $\varphi^{s,n}$ just above. A similar argument shows that $y$ satisfies the boundary condition. Thus $\tilde{H}_{m}^{MIT}$ is self-adjoint on $D\pare{\tilde{H}_{m}^{MIT}}$.

\end{proof}

\subsubsection{Self-adjointness of \texorpdfstring{$H_{m}$}{Hm}}

Recall that the domain of $H_{m}$ is:
\begin{enumerate}
\item[-] If $ 2ml \geq 1$:
\begin{equation*}
D(H_{m}) = \left \{ \phi \in \mathcal{H} ; \hspace{2mm} H_{m}\phi \in \mathcal{H} \right \}.
\end{equation*}
\item[-] If $0<m < \frac{1}{2l}$, we will denote by $H_{m}^{MIT}$ the operator $H_{m}$ with domain:
\begin{equation*}
D(H_{m}^{MIT}) = \left \{ \phi \in \mathcal{H} ; \hspace{2mm} H_{m}\phi \in \mathcal{H}, \hspace{1mm} \norme{\pare{\gamma^{1} + i} \phi\pare{x,.}}_{L^{2}\pare{\mathbb{S}^{2}}} = o\pare{\sqrt{- x}}, \hspace{1mm} x \to 0 \right \}.
\end{equation*}
\end{enumerate}
We obtain the following theorem:
\begin{theo}
\begin{enumerate} 
\item[-] For all $m \geq \frac{1}{2l}$, the operator $H_{m}$ with domain $D\pare{H_{m}}$ is self-adjoint.
\item[-] For all $m < \frac{1}{2l}$, the operator $H_{m}^{MIT}$ with domain $D(H_{m}^{MIT})$ is self-adjoint.
\end{enumerate}
\end{theo}

\begin{proof}
Recall that $\tilde{H}_{m} = \gamma^{5}_{B} P^{-1} \pare{-H_{m}} P \gamma^{5}_{B}$ where $\gamma^{5}_{B}$ and $P$ are unitary matrices. Thus $H_{m} = P \gamma^{5}_{B} \pare{-\tilde{H}_{m}} \gamma^{5}_{B} P^{-1} $. This is clear that $\psi \in D\pare{H_{m}}$ if and only if $ \gamma^{5}_{B} P^{-1}\psi \in D\pare{\tilde{H}_{m}}$ for $m \geq \frac{1}{2l}$. Moreover, recall that $\gamma^{1} = - P \gamma^{1}_{B} P^{-1}$ and $\gamma^{1}_{B} \gamma^{5}_{B} = - \gamma^{5}_{B}\gamma^{1}_{B}$ using \eqref{relP} and \eqref{relGamma5}. We then obtain:
\begin{equation*}
\norme{\pare{\gamma^{1} + i} \psi} = \norme{\pare{\gamma^{1}_{B} + i} \gamma^{5}_{B} P^{-1} \psi}.
\end{equation*}
Thus $\psi \in D\pare{H_{m}}$ if and only if $ \gamma^{5}_{B} P^{-1}\psi \in D\pare{\tilde{H}_{m}}$ for all $m>0$. This shows that $H_{m}$ is self-adjoint equipped with the convenient domain.
\end{proof}

\subsection{The Cauchy problem}

Using Stone theorem, we obtain:

\begin{theo}
Let $\psi_{0} \in \mathcal{H}$, there exists a unique solution $\psi$ to the equation:
\begin{equation}
\partial_{t} \psi = i H_{m} \psi
\end{equation}
such that
\begin{equation}
\psi \in C^{0}\pare{\R_{t};\mathcal{H}}
\end{equation}
and satisfying:
\begin{flalign}
\hphantom{A} & \psi\pare{t=0,.} = \psi_{0} \pare{.} \\
& \forall t\in \R, \hspace{1mm} \norme{\psi\pare{t,.}}_{\mathcal{H}} = \norme{\psi_{0}(.)}_{\mathcal{H}}.
\end{flalign}
\end{theo}

\subsection{Absence of eigenvalues}

\begin{prop}\label{27}
For all $m>0$ , the Dirac operator $H_{m}$, defined in \eqref{ExprHmDir}, does not admit any eigenvalues.
\end{prop}

\begin{proof}
Let us first show the absence of eigenvalues for $H_{m}^{s,n}$ for all $m>0$ and all $\pare{s,n} \in I$.
Since $H_{m}^{s,n}$ is self-adjoint on its domain, the eigenvalues (if they exist) are all real. So, suppose that there exists $\lambda \in \R$ and $\varphi \in D\pare{H_{m}^{s,n}}$ such that $H_{m}^{s,n} \varphi = \lambda \varphi$.

We define:
\begin{equation*}
w(x) = e^{i \lambda \gamma^{0}\gamma^{1} x} \varphi(x)
\end{equation*}
such that
\begin{equation*}
w'(x) = i\lambda \gamma^{0} \gamma^{1} w(x) + e^{i \lambda \gamma^{0} \gamma^{1} x} \varphi'(x).
\end{equation*}
But, with $ V(x) = \gamma^{0} \gamma^{2} A(x) \pare{s + \frac{1}{2}} - m \gamma^{0} B(x)$, we have:
\begin{equation*}
H_{m}^{s,n} \varphi - \lambda \varphi = 0  \Leftrightarrow i\gamma^{0}\gamma^{1} \varphi'(x)= \pare{\lambda - V(x)} \varphi(x)  \Leftrightarrow \varphi'(x) = i \gamma^{0} \gamma^{1} \pare{V(x) - \lambda} \varphi(x)
\end{equation*}
So, we obtain:
\begin{equation}
w'(x) = i\gamma^{0} \gamma^{1} e^{i \lambda \gamma^{0} \gamma^{1} x} V(x) e^{-i\lambda \gamma^{0} \gamma^{1} x} w(x).
\end{equation}
Write: $W(x) = i\gamma^{0} \gamma^{1} e^{i \lambda \gamma^{0} \gamma^{1} x} V(x) e^{-i\lambda \gamma^{0} \gamma^{1} x}$.
Let $T \in ]-\infty,0[$, we can then solve the preceding equation by:
\begin{equation*}
w(x) = e^{\int_{T}^{x} W(t) dt} w(T).
\end{equation*}
As in the remark after proposition \ref{prop3.3}, each component of $\varphi$ goes to $0$ at $-\infty$. Consequently, $w(x) \underset{x \to -\infty}{\to} 0$.\\
On the other hand, for all $x<0$, $\ds \int_{-\infty}^{x} \abs{W(t)} dt < \infty$ so:
\begin{equation*}
\underset{T \to -\infty}{\lim} e^{\int_{T}^{x} W(t) dt}= e^{\int_{-\infty}^{x} W(t) dt}
\end{equation*}
exists and is finished. As a consequence, we have:
\begin{equation*}
\underset{T \to -\infty}{\lim} e^{\int_{T}^{x} W(t) dt} w(T) = 0.
\end{equation*}
We then deduce that $w(x) = 0$ for all $x<0$ so it is the same for $\varphi$. Consequently, $H_{m}^{s,n}$ admits no eigenvalues.\\
We can now consider $H_{m}$. If $\lambda \in \R$ is an eigenvalue of $H_{m}$ then there exists $\varphi \in D\pare{H_{m}}$ such that $\pare{H_{m}-\lambda} \varphi = 0$. Using the decomposition of $\varphi$ in spherical harmonics, if $\varphi$ is non zero, there exists $\pare{s,n} \in I$ such that $\varphi_{s,n} \neq 0$ and $\varphi_{s,n}$ satisfies $\pare{H_{m}^{s,n}-\lambda} \varphi_{s,n} =0$. This is impossible since $H_{m}^{s,n}$ does not admit eigenvalues. Thus $\varphi$ is identically $0$. We deduce that $H_{m}$ does not admit any eigenvalue for all $m>0$.

\end{proof}

\section{Compactness results}
The purpose of this section is to prove that, for a well chosen function $f$, the operator $f\pare{x}\pare{H_{m}^{s,n}-z}^{-1}$ is compact for all $z \in \rho \pare{H_{m}^{s,n}}$. We will make use of this result for proving Mourre estimates in the following section. The key point here for the Mourre estimate is that $f$ only admits a finite limit at $0$.\\
This result is proved by separating our operator in two operators denoted $H_{+}$ and $H_{-}$. The operator $H_{+}$ has a potential which behaves like the one in $H_{m}^{s,n}$ at $0$ and is extended so that the potential becomes confining. Hence the resolvent of this operator is itself compact. For $H_{-}$, we preserve the behavior near the horizon of the black hole and extend it so that it decreases to $0$ at $0$. By extending the states and the potential, we are thus able to view the resolvent as the restriction of a resolvent on the entire line. For this last resolvent, we are able to use standard results about Hilbert-Schmidt operators.

~\\
We now enter into the details. We have:
\begin{equation}
H_{m}^{s,n}= \Gamma^{1} D_{x} + (s+\frac{1}{2})A(x)\gamma^{0} \gamma^{2} - m \gamma^{0} B(x).
\end{equation}
where $A$ and $B$ behave like:
\begin{equation*}
A-A_{0} \in T^{\kappa,2} ; \hspace{3mm} B- B_{0} \in T^{\tilde{\kappa},1}
\end{equation*}
with $\kappa,\tilde{\kappa} >0$. Moreover, $\Gamma^{1} = - \gamma^{0} \gamma^{1}$ where $\gamma^{0} \gamma^{1}$ is given in \eqref{refGamma0,1,2,3}.
The main result of this section is:

\begin{prop}\label{3.1}
Let $f\in C\pare{]-\infty,0]}$ such that $f$ goes to $0$ at $-\infty$. Let $z \in \rho(H_{m}^{s,n})$ where
$\rho(H_{m}^{s,n})$ is the resolvent set of $H_{m}^{s,n}$. Then the operator $f(x)\pare{H_{m}^{s,n}-z}^{-1}$ is compact on $\mathcal{H}$ for all $m>0$.
\end{prop}

\subsection{Asymptotic operators}

\subsubsection{Operator \texorpdfstring{$H_{-}$}{H-}}

Let us first introduce the operator $H_{c} = i \gamma^{0} \gamma^{1} \partial_{x}$ where $\gamma^{0} \gamma^{1}= \text{diag}\pare{-1,1,1,-1}$. We can thus prove the:
\begin{prop} \label{propH0aa}
The operator $H_{c} = i \gamma^{0} \gamma^{1} \partial_{x}$ is self-adjoint on the domain defined by:
\begin{equation*}
D\pare{H_{c}}= \left \{ \varphi \in \mathcal{H}_{s,n}; H_{c} \varphi \in \mathcal{H}_{s,n}, \hspace{1mm} \varphi_{1} \pare{0} = - \varphi_{3} \pare{0}, \hspace{1mm} \varphi_{2} \pare{0} = \varphi_{4} \pare{0} \right \}
\end{equation*}
\end{prop}

\begin{proof}
Since $D\pare{H_{c}} \subset H^{1}\pare{]-\infty,0[} \subset C^{0}\pare{]-\infty;0[}$, we can deduce that the elements of $D\pare{H_{c}}$ go to $0$ at $-\infty$ and from the boundary condition, we deduce the symmetry of $H_{c}$ on $D\pare{H_{c}}$. The closedness is also proven using the fact that $D\pare{H_{c}} \subset C^{0}\pare{]-\infty;0[}$.  \\
\indent On the other hand, since $C^{\infty}_{0}\pare{]-\infty,0[} \subset D\pare{H_{c}}$, we can prove (using distribution) that $H_{c}^{*} = H_{c}$ on $D\pare{H_{c}^{*}}$. We then study the default spaces. Let $\psi \in \ker \pare{H_{c}^{*} +i}$. Since $x \to e^{-x}$ is not in $L^{2}\pare{]-\infty,0[}$, we obtain:
\begin{equation*}
\ker \pare{H_{c}^{*} +i} = \text{vect} \left \{ \begin{pmatrix}
e^{x} \\
0 \\
0 \\ 
0
\end{pmatrix}, \begin{pmatrix}
0 \\
0 \\
0 \\ 
e^{x}
\end{pmatrix} \right \} \cap D\pare{H_{c}^{*}}.
\end{equation*}
But, if $\psi \in D\pare{H_{c}^{*}}$, then, for all $\varphi \in D\pare{H_{c}}$, we have:
\begin{flalign*}
0=\prods{H_{c} \varphi}{\psi} - \prods{\varphi}{H_{c}^{*} \psi}& = \underset{x\to 0}{\lim} \left ( -i\varphi_{1}\pare{x} \overline{\psi_{1}}\pare{x} + i \varphi_{2}\pare{x} \overline{\psi_{2}}\pare{x}  \right . \\
& \quad \left . + i \varphi_{3}\pare{x} \overline{\psi_{3}}\pare{x} - i \varphi_{4}\pare{x} \overline{\psi_{4}}\pare{x} \right).
\end{flalign*}
Choosing $\varphi$ such that $\varphi_{1}\pare{0} \neq 0$, we see that $\ker \pare{H_{c}^{*} +i} = \{0 \}$. The same is true for $H_{c}^{*} -i = \{0 \}$.
This shows that $H_{c}$ is self-adjoint on $D\pare{H_{c}}$.
 
\end{proof}
Now, let us define the operator $H_{-}$ by:
\begin{equation}
H_{-} = H_{c} + V_{-}(x)
\end{equation}
where
\begin{equation}
V_{-} (x) = \begin{cases}
x Id& , \hspace{1mm} for \hspace{1mm} x \geq d \\
\gamma^{0} \gamma^{2} A(x) \pare{s+\frac{1}{2}} - m \gamma^{0} B(x) &, \hspace{1mm} for \hspace{1mm} x \leq c.
\end{cases}
\end{equation}
with $c,d$ two negative constants such that $c\leq d$. We remark that $V_{-}$ is bounded on $\R_{-}^{*}$. Using the Kato-Rellich theorem, we obtain:
\begin{cor}
The operator $H_{-}$ equipped with $D\pare{H_{c}}$ is self-adjoint.
\end{cor}
\begin{rem}
Note that the potential of $H_{-}$ equals the potential of $H_{m}^{s,n}$ for $x$ negative and $\abs{x}$ large.
\end{rem}

\subsubsection{Operator \texorpdfstring{$H_{+}$}{H+}}
Let us define the operator $H_{+}$ by:
\begin{equation}
H_{+} = \Gamma^{1} D_{x} + V_{+}(x)
\end{equation}
where
\begin{equation}
V_{+}(x) = \begin{cases}
\gamma^{0} \gamma^{2} A(x) \pare{s+\frac{1}{2}} - m \gamma^{0} B(x)& , \hspace{1mm} for \hspace{1mm} x\geq b. \\
-x^{2} \gamma^{0} & , \hspace{1mm} for \hspace{1mm} x \leq a.
\end{cases}
\end{equation}
This time, the potential behaves like the potential in $H_{m}^{s,n}$ at $0$ and increases at $-\infty$. We then have a confining potential. This type of potential has been encountered in the article of A.Bachelot \cite{Bachelot}. For proving the self-adjointness of his operator, he uses the method we have recovered for proving the self-adjointness of our operator $H_{m}$. We just indicate the differents stages of the proof. We introduce the domain:
\begin{equation*}
 D(H_{+}) = \left \{\varphi \in L^{2}(\R_{-}^{*},\C^{4});H_{+} \varphi \in L^{2}(\R_{-}^{*}), \hspace{1mm} \norme{(\gamma^{1}+i) \varphi(x,.)}_{L^{2}\pare{S^{2}}}= o \pare{\sqrt{x}}, \hspace{1mm} x \to 0 \right \}
\end{equation*}
if $2ml<1$ and we remove the boundary condition for $2ml \geq 1$. In the following proof of compactness of $\pare{H_{+}-z}^{-1}$, we obtain estimates that allow us to prove the symmetry of this operator for $ml \geq\frac{1}{2}$. As before, we can do a unitary transform and obtain a result similar as lemma \ref{LemDomain}. We then obtain the asymptotic behavior of $\varphi$. This allows us to conclude in the case $ml\geq \frac{1}{2}$. If $ml<\frac{1}{2}$, we introduce the MIT boundary condition and a suitable partition of unity in order to separate the behavior at $0$ from the one at $-\infty$. We then obtain:
\begin{prop} \label{H+aa}
The operator $H_{+}$ equipped with $D\pare{H_{+}}$ is self-adjoint.
\end{prop}

\subsection{Compactness of \texorpdfstring{$f\pare{x}\pare{H_{-}-z}^{-1}$}{f(x)(H--z)-1}}

\begin{lem} \label{H-Comp}
Let $f\in C^{0}\pare{]-\infty,0]}$ such that $\underset{x\to -\infty}{\lim} f\pare{x} = 0$ and $z\in \rho\pare{H_{-}}$. Then $f\pare{.}\pare{H_{-}-z}^{-1}$ is compact.
\end{lem}

\begin{proof}
Let $\varphi \in D\pare{H_{c}}$ and $g = \pare{H_{c}-z} \varphi$ be defined on $]-\infty,0[$. Denote by $\varphi_{i}$ and $g_{i}$, $i=1,\cdots,4$, their components. We will extend these functions to $\R$ in the following way:
\begin{flalign*}
\tilde{\varphi_{1}} (x) = \begin{cases} 
\varphi_{1} (x) \hspace{1mm} if \hspace{1mm} x\leq 0,\\
- \varphi_{3}(-x) \hspace{1mm} if \hspace{1mm} x \geq 0
\end{cases} 
;& \hspace{5mm} \tilde{\varphi_{2}} (x) = \begin{cases} 
\varphi_{2} (x) \hspace{1mm} if \hspace{1mm} x\leq 0,\\
\varphi_{4}(-x) \hspace{1mm} if \hspace{1mm} x \geq 0
\end{cases}\\
\tilde{\varphi_{3}} (x) = \begin{cases} 
\varphi_{3} (x) \hspace{1mm} if \hspace{1mm} x\leq 0,\\
-\varphi_{1}(-x) \hspace{1mm} if \hspace{1mm} x \geq 0
\end{cases}
; &\hspace{5mm} \tilde{\varphi_{4}} (x) = \begin{cases} 
\varphi_{4} (x) \hspace{1mm} if \hspace{1mm} x\leq 0,\\
\varphi_{2}(-x) \hspace{1mm} if \hspace{1mm} x \geq 0.
\end{cases}
\end{flalign*}
The components are thus in $H^{1}\pare{\R}$. We also extend $g$ into $\tilde{g} \in \left [ L^{2}\pare{\R} \right]^{4}$ in the same way. Here, we have put $\tilde{H}^{c}$ for the operator with the same formula as $H_{c}$ but acting on functions defined on $\R$. Some calculation gives that $\pare{H_{c}-z}\varphi = g$ if and only if $\pare{\tilde{H}^{c} - z} \tilde{\varphi} = \tilde{g}$ for all $z$ in the resolvent set of $H_{c}$. \\
Let $f\in C^{0}\pare{]-\infty,0]}$ such that $\underset{x\to -\infty}{\lim} f\pare{x} = 0$. We consider a sequence $\pare{g_{n}}_{n \in \N} \in \pare{L^{2}\pare{\R_{-}^{*}}}^{\N}$ such that $g_{n} \rightharpoonup 0$ and we want to prove that $f(x) \pare{H_{c}-z}^{-1} g_{n}$ goes to $0$ strongly in $L^{2}$. We introduce $u_{n} = \pare{H_{c}-z}^{-1} g_{n}$ and extend $g_{n}$ and $u_{n}$ into $\tilde{g_{n}}$ and $\tilde{u_{n}}$ as before. Consequently, $\tilde{g_{n}} \rightharpoonup 0$ in $L^{2}\pare{\R}$ and $\tilde{u_{n}} = \pare{\tilde{H}^{c}-z}^{-1} \tilde{g_{n}}$. We mention here a consequence of theorem $IX.29$ in \cite{Reedsimon} which say that if $f,g \in L^{\infty}\pare{\R^{n}}$ and:
\begin{equation*}
\underset{\abs{x} \to \infty}{\lim} f\pare{x} = 0, \hspace{3mm} \underset{\abs{\xi} \to \infty}{\lim} g\pare{\xi} = 0,
\end{equation*}
then the operator $f\pare{x} g\pare{-i\nabla}$ is compact. Since $x \to \pare{x-z}^{-1} \in L^{\infty}$ and $\abs{x-z}^{-1} \underset{\abs{x} \to \infty}{\to} 0$, we deduce that:
\begin{equation*}
\tilde{f}(x) \pare{\tilde{H}^{c}-z}^{-1} \tilde{g_{n}} \underset{n \to \infty}{\overset{L^{2}\pare{\R}}{\to}} 0,
\end{equation*}
where $\tilde{f}$ is the extension of $f$ by symmetry on $\R_{+}$. Therefore, we have:
\begin{equation*}
\mathds{1}_{]-\infty,0[}(x) \tilde{f}(x) \pare{\tilde{H}^{c}-z}^{-1} \tilde{g_{n}} = \mathds{1}_{]-\infty,0[}(x) f(x)\tilde{u_{n}} = f(x) u_{n} = f(x) \pare{H_{c}-z}^{-1} g_{n}.
\end{equation*}
So $f(x) \pare{H_{c}-z}^{-1} g_{n} \underset{n\to \infty}{\overset{L^{2}\pare{\R^{*}_{-}}}{\to}} 0$ and the operator $f(x) \pare{H_{c}-z}^{-1}$ is compact.

Since $V_{-}$ goes to $0$ at $-\infty$ and $0$ and using the identity:
\begin{equation*}
f(x)\pare{H_{-}-z}^{-1} = - f(x) \pare{H_{-}-z}^{-1} V_{-}(x) \pare{H_{c}-z}^{-1} + f(x) \pare{H_{c}-z}^{-1},
\end{equation*}
we deduce that $\pare{H_{-}-z}^{-1}-\pare{H_{c}-z}^{-1}$ is compact and consequently that $f(x)\pare{H_{-}-z}^{-1}$ is also compact.
\end{proof}

\subsection{Compactness of \texorpdfstring{$\pare{H_{+}-z}^{-1}$}{(H+-z)-1}}

\begin{lem} \label{H+Comp}
The operator $\pare{H_{+}-z}^{-1}$ is compact.
\end{lem}

\begin{proof}
We follow the proof of the compactness result in \cite{Bachelot}. Let us show that the set:
\begin{equation}
K = \left \{ \varphi \in D(H_{+}) ; \hspace{1mm} \norme{\varphi}+ \norme{H_{+} \varphi} \leq 1 \right \}
\end{equation}
is compact. We consider a sequence $(\varphi_{n})_{n \in \N} \in K^{\N}$. Using the Banach-Alaoglu theorem and distributions, we obtain the existence of a sub-sequence $(\varphi_{\nu})$ such that:
\begin{equation*}
\varphi_{\nu} \underset{\nu \to \infty}{\rightharpoonup} \varphi ; \hspace{5mm} f_{\nu} =: H_{+} \varphi_{\nu} \underset{\nu \to \infty}{\rightharpoonup} H_{+} \varphi := f.
\end{equation*} 
Let:
\begin{equation*}
W(x) = \begin{cases}
m B(x) = - \frac{ml}{x} + O\pare{x}, & \hspace{1mm} for \hspace{1mm} x \geq b. \\
x^{2}, & \hspace{1mm} for \hspace{1mm} x \leq a,
\end{cases} 
\end{equation*}
so that $W$ is smooth on $]a,b[$. The equation $H_{+} \varphi_{\nu} = f_{\nu}$ can be written:
\begin{equation*}
\pare{\Gamma^{1} D_{x} - \gamma^{0} W\pare{x}} \varphi_{\nu} = -\gamma^{0} \gamma^{2} \pare{s+\frac{1}{2}} A\pare{x} \varphi_{\nu} + f_{\nu}.
\end{equation*}
We denote $g_{\nu}$ the right hand side of this equation. Then $g_{\nu}$ is in $L^{2}(]-\infty,0[)$ and $g_{\nu} \rightharpoonup g$ where $g$ is defined by replacing $\varphi_{\nu},f_{\nu}$ by $\varphi,f$ respectively.
We thus obtain four differential equations:
\begin{equation}
\begin{cases}
\partial_{x} \pare{\varphi^{1}_{\nu} + \varphi^{3}_{\nu}} + W(x) \pare{\varphi^{1}_{\nu} + \varphi^{3}_{\nu}} = i \pare{g^{1}_{\nu}-g^{3}_{\nu}} \\
\partial_{x} \pare{\varphi^{2}_{\nu} - \varphi^{4}_{\nu}} + W(x) \pare{\varphi^{2}_{\nu} - \varphi^{4}_{\nu}} = -i \pare{g^{2}_{\nu} + g^{4}_{\nu}} \\
\partial_{x} \pare{\varphi^{1}_{\nu} - \varphi^{3}_{\nu}} - W(x) \pare{\varphi^{1}_{\nu} - \varphi^{3}_{\nu}} = i \pare{g^{1}_{\nu}+g^{3}_{\nu}} \\
\partial_{x} \pare{\varphi^{2}_{\nu} + \varphi^{4}_{\nu}} - W(x) \pare{\varphi^{2}_{\nu} + \varphi^{4}_{\nu}} = i \pare{g^{4}_{\nu} - g^{2}_{\nu}}
\end{cases}
\end{equation}
For some constants $\lambda^{j}_{\nu}$, $j=1,\cdots,4$, the solutions are:
\begin{flalign}
\hphantom{A} & \pare{\varphi^{1}_{\nu}+\varphi^{3}_{\nu}}(x)= \lambda^{1}_{\nu} e^{-\int_{-1}^{x} W(u) \mathrm du} + i \ds \int_{-\infty}^{x} \pare{g^{1}_{\nu}-g^{3}_{\nu}} e^{\int_{-1}^{t} W(u) \mathrm du - \int_{-1}^{x} W(u) \mathrm du} \mathrm dt \label{3.10}\\
& \pare{\varphi^{2}_{\nu} - \varphi^{4}_{\nu}} (x) = \lambda^{2}_{\nu}  e^{-\int_{-1}^{x} W(u) \mathrm du} - i \ds \int_{-\infty}^{x} \pare{g^{2}_{\nu} + g^{4}_{\nu}} e^{\int_{-1}^{t} W(u) \mathrm du - \int_{-1}^{x} W(u) \mathrm du} \mathrm dt \label{3.11} \\
& \pare{\varphi^{1}_{\nu} - \varphi^{3}_{\nu}} (x) = \lambda^{3}_{\nu}  e^{\int_{-1}^{x} W(u) \mathrm du} + i \ds \int_{0}^{x}  \pare{g^{1}_{\nu}+g^{3}_{\nu}} e^{-\int_{-1}^{t} W(u) \mathrm du + \int_{-1}^{x} W(u) \mathrm du} \mathrm dt \label{3.12}\\
& \pare{\varphi^{2}_{\nu} + \varphi^{4}_{\nu}} (x) = \lambda^{4}_{\nu}  e^{\int_{-1}^{x} W(u) \mathrm du} + i \ds \int_{0}^{x} \pare{g^{4}_{\nu} - g^{2}_{\nu}} e^{-\int_{-1}^{t} W(u) \mathrm du + \int_{-1}^{x} W(u) \mathrm du} \mathrm dt \label{3.13}
\end{flalign}
\underline{Proof of the pointwise convergence of the integral terms.}\\
We have:
\begin{equation}
\ds \int_{-1}^{x} W(u) \mathrm du = \begin{cases} 
- ml \ln(-x) + \ds \int_{-1}^{x} O(u) \mathrm du & , \hspace{1mm} for \hspace{1mm} x \geq b. \\
\frac{x^{3}}{3} - \frac{a^{3}}{3} + \int_{-1}^{a} W(u) \mathrm du & , \hspace{1mm} for \hspace{1mm} x \leq a.
\end{cases}
\end{equation}
where $\ds \int_{-1}^{x} O(u) \mathrm du$ is bounded on $[b;0[$. We obtain:
\begin{flalign*}
\hphantom{A} & e^{\int_{-1}^{x} W(u) \mathrm du} = \begin{cases} 
\pare{-x}^{-ml} e^{\int_{-1}^{x} O(u) \mathrm du} &,  \hspace{1mm} for \hspace{1mm} x \geq b. \\
C_{1} e^{\frac{x^{3}}{3}} & , \hspace{1mm} for \hspace{1mm} x \leq a.
\end{cases} \\
& e^{-\int_{-1}^{x} W(u) \mathrm du} = \begin{cases}
\pare{-x}^{ml} e^{ - \int_{-1}^{x} O(u) \mathrm du} & , \hspace{1mm} for \hspace{1mm} x \geq b.\\
C_{2} e^{-\frac{x^{3}}{3}} & , \hspace{1mm} for \hspace{1mm} x \leq a.
\end{cases}
\end{flalign*}
where $C_{1},C_{2}$ are positive constants. We thus see that $e^{\int_{-1}^{t} W(u) \mathrm du}$ is square integrable on $]-\infty,x[$ and that $e^{-\int_{-1}^{t} W(u) \mathrm du}$ is square integrable on $]x,0[$. Consequently, since $g_{\nu}$ is weakly convergent, we deduce that:
\begin{equation*}
\ds \int_{-\infty}^{x} \pare{g^{1}_{\nu}-g^{3}_{\nu}} e^{\int_{-1}^{t} W(u) \mathrm du - \int_{-1}^{x} W(u) \mathrm du} \mathrm dt \underset{\nu \to \infty}{\to} \ds \int_{-\infty}^{x} \pare{g^{1}-g^{3}} e^{\int_{-1}^{t} W(u) \mathrm du - \int_{-1}^{x} W(u) \mathrm du} \mathrm dt
\end{equation*}
when $\nu \to \infty$. The same is true for the integral with $g^{1}_{\nu}+g^{3}_{\nu}$. \\
\underline{Majorations of integral terms by $L^{2}$ functions independent of $\nu$.}\\
In the following, we will only treat $\pare{\varphi^{1}_{\nu}+\varphi^{3}_{\nu}}$ and $\pare{\varphi^{1}_{\nu}-\varphi^{3}_{\nu}}$. The other functions can be treated in the same way. When $a \leq x \leq b$, the functions are smooth hence integrable. We study the other cases:
\begin{enumerate}
\item[a)] First, using the Cauchy-Schwarz inequality and that $g^{1}_{\nu}+g^{3}_{\nu}$ is bounded in $L^{2}$, we obtain:
\begin{equation}
\abs{\ds \int_{0}^{x}  \pare{g^{1}_{\nu}+g^{3}_{\nu}} e^{\int_{t}^{x} W(u) \mathrm du} \mathrm dt}^{2} \lesssim 
 \abs{\ds \int_{0}^{x} e^{-2\int_{-1}^{t} W(u) \mathrm du + 2\int_{-1}^{x} W(u) \mathrm du} \mathrm dt}
\end{equation}
Therefore, we prove that the right hand side is integrable:
\begin{enumerate}
\item[i)] If $x \geq b$, using the expression of $W$, the right hand side is integrable since:
\begin{equation*}
\abs{\ds \int_{x}^{0} e^{-2\int_{-1}^{t} W(u) \mathrm du + 2\int_{-1}^{x} W(u) \mathrm du} \mathrm dt}  \leq e^{2C}\abs{\ds \int_{x}^{0} \pare{-\frac{1}{x}}^{2ml} \pare{-t}^{2ml} \mathrm dt} = e^{2C} \frac{-x}{1+2ml}.
\end{equation*}
\item[ii)] If $x \leq a$, we have:
\begin{equation*}
\abs{\ds \int_{x}^{0} e^{ 2\int_{t}^{x} W(u) \mathrm du} \mathrm dt} = (C_{1})^{2} e^{2\frac{x^{3}}{3}} \pare{ \ds \int_{x}^{a} (C_{2})^{2} e^{-2\frac{t^{3}}{3}} \mathrm dt + \ds \int_{a}^{0} e^{-2\int_{-1}^{t} W(u) \mathrm du} \mathrm dt}.
\end{equation*}
The function $(C_{1})^{2} e^{2\frac{x^{3}}{3}}\pare{\ds \int_{a}^{0} e^{-2\int_{-1}^{t} W(u) \mathrm du} \mathrm dt}$ is integrable on $]-\infty,a]$ and:
\begin{equation*}
\ds \int_{x}^{a} e^{-2\frac{t^{3}}{3}} \mathrm dt  \leq - \frac{1}{2a^{2}}  e^{-2\frac{a^{3}}{3}} + \frac{1}{2x^{2}}  e^{-2\frac{x^{3}}{3}} - \frac{1}{a^{3}} \ds \int_{x}^{a} e^{-2\frac{t^{3}}{3}} \mathrm dt,
\end{equation*}
by integration by parts. Choosing $a$ such that $1 + \frac{1}{a^{3}}>0$, we deduce that $e^{2\frac{x^{3}}{3}} \int_{x}^{a} e^{-2\frac{t^{3}}{3}} \mathrm dt$ is integrable on $]-\infty,a]$ and goes to $0$ at $-\infty$.
\end{enumerate}

\item[b)] Secondly, as above, we study the integrability of $\int_{-\infty}^{x} e^{2\int_{-1}^{t} W(u) \mathrm du -2 \int_{-1}^{x} W(u) \mathrm du} \mathrm dt$:
\begin{enumerate}
\item[i)] If $x\geq b$, using the expression of $W$ and separating the integral from $-\infty$ to $b$ and from $b$ to $x$, we have to study
$\pare{-x}^{2ml} e^{-2\int_{-1}^{x} T(u) \mathrm du} \ds \int_{-\infty}^{b} e^{2\int_{-1}^{t} W(u) \mathrm du } \mathrm dt$ and $ \pare{-x}^{2ml} e^{-2\int_{-1}^{x} T(u) \mathrm du} \ds \int_{b}^{x} \pare{-\frac{1}{t}}^{2ml} e^{2\int_{-1}^{t} T(u) \mathrm du} \mathrm dt$. The first term is clearly integrable and since $e^{2\int_{-1}^{t} T(u) \mathrm du}$ is bounded on $[b,0[$, we can perform the second integral to see that it is also integrable.
\item[ii)] If $x\leq a$, since $\ds \int_{-\infty}^{x} \frac{1}{t^{3}} e^{\frac{2t^{3}}{3}} \mathrm dt \leq 0$, by integration by part, we have:
\begin{equation*}
\ds \int_{-\infty}^{x} e^{2\int_{-1}^{t} W(u) \mathrm du -2 \int_{-1}^{x} W(u) \mathrm du} \mathrm dt \leq C_{2}^{2} C_{1}^{2}\frac{1}{2x^{2}}.
\end{equation*}
\end{enumerate}
This ends the proof of the integrability.
\end{enumerate}

\underline{Convergence in $L^{2}$ of integral terms.}\\
We can use the dominate convergence theorem to obtain:
\begin{equation} \label{3.19}
\ds \int_{0}^{x}  \pare{g^{1}_{\nu}+g^{3}_{\nu}} e^{-\int_{-1}^{t} W(u) \mathrm du + \int_{-1}^{x} W(u) \mathrm du} \mathrm dt  \underset{\nu \to \infty}{\overset{L^{2}}{\to}} \ds \int_{0}^{x}  \pare{g^{1}+g^{3}} e^{-\int_{-1}^{t} W(u) \mathrm du + \int_{-1}^{x} W(u) \mathrm du} \mathrm dt.
\end{equation}
and the same for the integral with $g^{1}_{\nu} - g^{3}_{\nu}$.\\
\underline{Study of the sequences $\lambda^{i}_{\nu}$, $i=1,\cdots,4$.}
\begin{enumerate}
\item[a)] Let us study the convergence of $\lambda^{3}_{\nu}$ in \eqref{3.12} (we can do the same for $\lambda^{4}_{\nu}$). 
\begin{enumerate}
\item[-] If $ml<\frac{1}{2}$, using that $e^{\int_{-1}^{x} W(u) \mathrm du} \in L^{2}$, $\varphi_{\nu} \rightharpoonup \varphi$ and $\eqref{3.19}$, the term:
\begin{flalign*}
\hphantom{A} &\pare{\lambda^{3}_{\nu} - \lambda^{3}} \norme{e^{\int_{-1}^{x} W(u) \mathrm du}}_{L^{2}}^{2} = \\
& \prods{\pare{\pare{\varphi^{1}_{\nu} - \varphi^{3}_{\nu}}- \ds \int_{0}^{x}  \pare{g^{1}_{\nu}+g^{3}_{\nu}} e^{-\int_{-1}^{t} W(u) \mathrm du + \int_{-1}^{x} W(u) \mathrm du} \mathrm dt}}{e^{\int_{-1}^{x} W(u) \mathrm du}}_{L^{2}}\\
& -\prods{\pare{\pare{\varphi^{1} - \varphi^{3}} - \ds \int_{0}^{x}  \pare{g^{1}+g^{3}} e^{-\int_{-1}^{t} W(u) \mathrm du + \int_{-1}^{x} W(u) \mathrm du} \mathrm dt}}{e^{\int_{-1}^{x} W(u) \mathrm du}}_{L^{2}}
\end{flalign*}
goes to $0$ as $\nu \to -\infty$. We deduce that $\lambda^{3}_{\nu} \underset{\nu \to \infty}{\to} \lambda^{3}$. 
\item[-] If $ml \geq \frac{1}{2}$, $e^{\int_{-1}^{x} W(u) \mathrm du} \notin L^{2}$ and $\lambda^{3}_{\nu}=0$.
\end{enumerate}

\item[b)] We then study the convergence of $\lambda^{1}_{\nu}$ and $\lambda^{2}_{\nu}$. \\
Since $\varphi^{1}_{\nu}+\varphi^{3}_{\nu} \in L^{2}$, $e^{-\int_{-1}^{x} W(u) \mathrm du} \notin L^{2}$ and the other terms are in $L^{2}$, we deduce that $\lambda^{1}_{\nu}=\lambda^{2}_{\nu}=0$ for all $\nu \in \N$. 
\end{enumerate}
\underline{Convergence in $L^{2}$ of the sequences $\varphi^{1}_{\nu} - \varphi^{3}_{\nu}$, $\varphi^{2}_{\nu} + \varphi^{4}_{\nu}$, $\varphi^{1}_{\nu}+\varphi^{3}_{\nu}$, $\varphi^{2}_{\nu}-\varphi^{4}_{\nu}$.} \\
Using the dominate convergence theorem, we deduce that $\varphi^{1}_{\nu} - \varphi^{3}_{\nu} \underset{\nu \to \infty}{\overset{L^{2}}{\to}} \varphi^{1}- \varphi^{3}$. The same is true for the other functions. Thus, the sequence $\pare{\varphi_{n}}_{n\in \N}$ admits a converging sub-sequence which proves that $K$ is compact. Consequently, $\pare{H_{+}+i}^{-1}$ is compact and so is $\pare{H_{+}-z}^{-1}$ for all $z \in \rho(H_{+})$ using a resolvent identity.
\end{proof}

\subsection{Proof of proposition \ref{3.1}}

\begin{proof}
Let $j_{-}, j_{+} \in C^{\infty}$ such that $j_{-}^{2} + j_{+}^{2} = 1$, $\supp(j_{-}) \subset ]-\infty,c[$ and $\supp(j_{+}) \subset ]b,0[$. We define:
\begin{equation*}
Q(z) = j_{-}(x) \pare{H_{-}-z}^{-1} j_{-}(x) + j_{+}(x) \pare{H_{+}-z}^{-1} j_{+}(x).
\end{equation*}
Since $H_{m}^{s,n}-z = H_{-}-z$ on $]-\infty,c[$ and $H_{m}^{s,n}-z = H_{+} - z$ on $]b,0[$, we have:
\begin{equation*}
\pare{H_{m}^{s,n}-z} Q(z) = 1 - w(z)
\end{equation*}
where:
\begin{equation*}
w(z) =  -\pare{\left [ \pare{H_{m}^{s,n}-z}, j_{-}(x) \right ]\pare{H_{-}-z}^{-1} j_{-}(x) + \left [ \pare{H_{m}^{s,n}-z}, j_{+}(x) \right ]\pare{H_{+}-z}^{-1} j_{+}(x)}.
\end{equation*}
Since $ \left [ \pare{H_{m}^{s,n}-z}, j_{-}(x) \right ] = i \gamma^{0} \gamma^{1} j_{-}'(x)$ and $\left [ \pare{H_{m}^{s,n}-z}, j_{+}(x) \right ] = i \gamma^{0} \gamma^{1} j_{+}'(x)$ and $j_{-}',j_{+}'$ have compact support, we deduce that $w(z)$ is compact for all $z\in \rho\pare{H}$ using the last two sections. Moreover, $w: \rho \pare{H} \to \mathcal{L}\pare{L^{2}}$ is analytic.\\
\indent Since $j_{-}',j_{+}',j_{-},j_{+}$ are bounded, for some constant $C>0$, we have:
\begin{equation*}
\norme{w(z) \varphi}_{2} \leq \frac{C}{\abs{\Im z}}\norme{\varphi}_{2},
\end{equation*}
for all $\varphi \in L^{2}$. We then choose $z$ such that the imaginary part satisfies $\frac{C}{\abs{\Im z}} <1$. Therefore, $1-w(z)$ is invertible. Using the analytic Fredholm theorem, we have that $1-w(z)$ is invertible for all $z \in \rho \pare{H} \smallsetminus S$ where $S$ is a discrete set without accumulation points.\\
For these $z$, we deduce that:
\begin{equation}
\pare{H_{m}^{s,n}-z}^{-1} = Q(z) \pare{1-w(z)}^{-1}.
\end{equation}
Let $f$ be a continuous function going to $0$ at $-\infty$ and admitting a finite limit at $0$. Then $f(x) Q(z)$ is compact. Thus for $z \in \rho \pare{H} \smallsetminus S$, $f(x)\pare{H_{m}^{s,n}-z}^{-1}$ is compact. Using the analyticity of $z \to \pare{H_{m}^{s,n}-z}^{-1}$, we obtain the compactness for all $z \in \rho \pare{H_{m}^{s,n}}$.
\end{proof}

\section{Mourre estimates}

\subsection{Mourre theory}

We recall here some facts about the Mourre theory. Let $\mathcal{A}$ be a self-adjoint operator. We say that the pair $(\mathcal{A},H)$ satisfies the Mourre conditions if
\begin{flalign} 
\hphantom{A} & D(\mathcal{A}) \cap D(H) \hspace{2mm} \text{is dense in} \hspace{2mm} D(H)\label{M1} \\
& e^{it\mathcal{A}} \hspace{2mm} \text{preserves} \hspace{2mm} D(H) \hspace{2mm} \text{for t>0,} \hspace{2mm} \underset{\abs{t}\leq 1}{\sup} \norme{H e^{it\mathcal{A}}u}<\infty, \hspace{2mm} \forall u \in D(H) \label{M2} \\
& [iH,\mathcal{A}] \hspace{2mm} \text{defined as quadratic form on} \hspace{2mm} D(H) \cap D(\mathcal{A}) \notag \\
&\text{extend to a bounded operator from} \hspace{2mm} D(H) \hspace{2mm} \text{into} \hspace{2mm} \mathcal{H}.\label{M3}
\end{flalign}
The Mourre conditions are stronger than $C^{1}(\mathcal{A})$ regularity. We recall the definition of $C^{k}(\mathcal{A})$:
\begin{defi}
We say that $H \in C^{k}(\mathcal{A})$ if there exists $z \in \C \setminus \sigma(H)$ such that
\begin{equation}
\R \ni t \mapsto e^{it\mathcal{A}}\pare{z-H}^{-1} e^{-it\mathcal{A}}
\end{equation}
is $C^{k}$ for the strong topology of $\mathcal{L}(\mathcal{H})$.
\end{defi}
We then have the following lemma (see \cite[Proposition~5.1.2,~Theorem~6.3.4]{ABG}):
\begin{lem} \label{52}
Suppose that $(H,\mathcal{A})$ satisfies the Mourre conditions. Then $H \in C^{1}(\mathcal{A})$.
\end{lem}
We also recall a lemma concerning the $C^{2}(\mathcal{A})$ regularity:
\begin{lem} \label{53}
Suppose that $H\in C^{1}(\mathcal{A})$ and that the commutator $[i\mathcal{A},H]$ extends to a bounded operator from $D(H)$ into $\mathcal{H}$. We denote $[i\mathcal{A},H]_{0}$ this extension. If, in addition, the commutator $[i\mathcal{A},[i\mathcal{A},H]_{0}]$ defined as a quadratic form on $D(\mathcal{A}) \cap D(H)$ extends to a bounded operator from $D(H)$ into $D(H)^{*}$, then $H\in C^{2}(\mathcal{A})$.
\end{lem}

\subsection{Mourre estimate} \label{Cas2mlptt1Mou}

We will use $\mathcal{A} = \Gamma x$ as conjugate operator where $\Gamma = -\gamma^{0} \gamma^{1} = diag\pare{1,-1,-1,1}$. The operator $\mathcal{A}$ is self-adjoint when equipped with domain
\begin{equation}
D(\mathcal{A}) = \{ \varphi \in \mathcal{H}_{s,n}; \hspace{1mm} \mathcal{A} \varphi \in \mathcal{H}_{s,n} \}.
\end{equation}

\begin{lem} \label{LemMourCond2mlpetit}
For all $m>0$, the pair $\pare{H^{s,n}_{m},\mathcal{A}}$ satisfies the Mourre conditions. Consequently, $H^{s,n}_{m} \in C^{1}(\mathcal{A})$ 
\end{lem}

\begin{proof}

\indent Let us check \eqref{M1}: \\
\indent Case $2ml <1$:\\
\indent Let $\chi$ be a $C^{\infty}$ function such that $\chi = 1$ on $[-1,0]$, $\supp{\chi} \subset ]-2,0]$. We set $\chi_{k}(x) = \chi\pare{\frac{x}{k}}$ for all $k \in \N \smallsetminus \{0\}$. This implies that $\supp{\chi_{k}(x)} = 1$ on $]-k,0]$. We have $\chi_{k}'(x)=\frac{1}{k} \chi'\pare{\frac{x}{k}}$ so that it is bounded. Using these facts, we see that $\chi_{k} \varphi \in D(\mathcal{A}) \cap D\pare{H^{s,n}_{m}}$ if $\varphi \in D\pare{H^{s,n}_{m}}$. \\
\indent We now show that $\chi_{k} \varphi \underset{k \to \infty}{\to} \varphi$ for the norm: $\norme{\varphi}_{H^{s,n}_{m}} = \norme{\varphi}_{\mathcal{H}^{s,n}} + \norme{H^{s,n}_{m} \varphi}_{\mathcal{H}_{s,n}}$. By the dominate convergence theorem we have $\chi_{k} \varphi \overset{\mathcal{H}_{s,n}}{\underset{k \to \infty}{ \longrightarrow }} \varphi$. Moreover, $\abs{\chi_{k}'(x)} \leq \frac{1}{k} C$, so:
\begin{equation*}
\norme{H^{s,n}_{m} \varphi - H^{s,n}_{m} \chi_{k} \varphi} \leq \frac{C_{0}}{k} \norme{\varphi} + \norme{H^{s,n}_{m} \varphi - \chi_{k} H^{s,n}_{m} \varphi}.
\end{equation*}
which gives the desired result when $k$ goes to infinity for $\varphi \in D\pare{H^{s,n}_{m}}$. We deduce \eqref{M1}. 

We denote $D\pare{H^{s,n}_{m}}_{c} = \left \{ \chi_{k}\varphi; \hspace{1mm} \varphi \in D\pare{H^{s,n}_{m}}, \hspace{1mm} k \in \N \smallsetminus \{0\} \right \}$.

~\\
\indent Case $2ml \geq 1$:\\
In this case, $C^{\infty}_{0} \pare{]-\infty,0[}$ is a subset of $D\pare{\mathcal{A}} \cap D\pare{H_{m}^{s,n}}$ and is dense in $D\pare{H_{m}^{s,n}}$.

~\\
\indent Let us check \eqref{M2}: \\
For all $t>0$,
\begin{equation*}
e^{it\mathcal{A}} = \text{diag}(e^{itx},e^{-itx},e^{-itx},e^{itx}).
\end{equation*}
Let $\varphi \in D\pare{H^{s,n}_{m}}$, then:
\begin{enumerate}
\item[-] $e^{it\mathcal{A}} \varphi \in \mathcal{H}_{s,n}$.
\item[-] $H^{s,n}_{m} e^{it\mathcal{A}} \varphi = e^{it\mathcal{A}} H^{s,n}_{m} \varphi + te^{it\mathcal{A}} \varphi$. So $H^{s,n}_{m} e^{it\mathcal{A}} \varphi \in \mathcal{H}_{s,n}$ and $\underset{\abs{t}\leq 1}{\sup} \norme{H^{s,n}_{m} e^{it\mathcal{A}} \varphi} < \infty$.
\end{enumerate}
We need to check the boundary condition in the case $2ml <1$. We have:
\begin{equation*}
\norme{\pare{\gamma^{1} + i} e^{it\mathcal{A}} \varphi (x,.)}_{\mathcal{W}^{0}} = \norme{\begin{pmatrix}
ie^{itx} \varphi_{1} + i e^{-itx} \varphi_{3} \\
i e^{-itx} \varphi_{2} - i e^{itx} \varphi_{4} \\
i e^{itx} \varphi_{1} + i e^{-itx} \varphi_{3} \\
-i e^{-itx} \varphi_{2} + i e^{itx} \varphi_{4}
\end{pmatrix}}_{[L^{2}\pare{S^{2}}]^{4}}.
\end{equation*}
Let's consider: $\norme{ie^{itx} \varphi_{1} + i e^{-itx} \varphi_{3}}_{L^{2}\pare{S^{2}}}$ when x goes to $0$. By Taylor expansion, we must check that $-x \pare{\norme{\varphi_{1}(x,.)}_{L^{2}\pare{S^{2}}} + \norme{\varphi_{3}(x,.)}_{L^{2}\pare{S^{2}}}}$ is $o\pare{\pare{-x}^{\frac{1}{2}}}$. Since $\varphi \in D\pare{H_{m}^{s,n}}$, there exists functions $\psi_{-} \in W^{\frac{1}{2}}_{-}$, $\chi_{-} \in W^{\frac{1}{2}}_{+}$ and a function $\phi \in C^0 \left([0,\frac{\pi}{2}]_x; L^2(S^2;\C^4) \right)$, such that $\norme{\phi^{s}_{n}(r_{*},\theta,\varphi)}_{\mathcal{W}^{0}} = o \pare{\sqrt{\pare{-x}}}$ as $x \to 0$, satisfying:
\begin{equation*}
 \psi_{s,n}  = \pare{{-x}^{-ml}} \begin{pmatrix}
\psi^{s}_{-,n} (\theta,\varphi) \\
\chi^{s}_{-,n} (\theta,\varphi) \\
-i\psi^{s}_{-,n} (\theta,\varphi)\\
i\chi^{s}_{-,n} (\theta,\varphi)
\end{pmatrix} + \phi^{s}_{n}(r_{*},\theta,\varphi).
\end{equation*}
We thus obtain:
\begin{equation*}
-x \norme{\varphi_{1}(x,.)}_{L^{2}\pare{S^{2}}} \leq C_{s,n}\pare{-x}^{1-ml} -x \pare{o\pare{\pare{-x}^{\frac{1}{2}}}}.
\end{equation*}
Since $1-ml > \frac{1}{2}$ when $ml<\frac{1}{2}$, we have that $-2x \norme{\varphi_{1}(x,.)}_{L^{2}\pare{S^{2}}} = o\pare{\pare{-x}^{\frac{1}{2}}}$. Since $\varphi \in D\pare{H^{s,n}_{m}}$, this proves that the boundary condition is fulfilled and then \eqref{M2}.

~\\
Let us check \eqref{M3}: \\
First, we see that $xA(x)$ and $xB(x)$ are bounded functions on $]- \infty,0[$. Let $u,v \in D\pare{H^{s,n}_{m}}_{c}$ in the case $2ml<1$ and $u,v \in C^{\infty}_{0}\pare{]-\infty,0[}$ in the case $2ml \geq 1$, we have:
\begin{equation} \label{517}
\left [ H^{s,n}_{m},i\mathcal{A} \right ] (u,v) = \prods{u + 2i\pare{s+\frac{1}{2}} xA(x) \gamma^{2}\gamma^{1}u + 2i m xB(x) \gamma^{1}u}{v}.
\end{equation}
This shows that:
\begin{equation*}
\abs{\left [ H^{s,n}_{m},i\mathcal{A} \right ](u,v)} \leq C_{1} \norme{u}_{\mathcal{H}_{s,n}}\norme{v}_{\mathcal{H}_{s,n}}
\end{equation*}
for some constant $C_{1}$ and consequently, \eqref{M3} is satisfied.
\end{proof}

We then have the following:
\begin{prop} \label{54}
Recall that $\mathcal{A} = \Gamma x$. Let $I \subset \R$ be a compact non-empty interval. Then, for all $m>0$, we have:
\begin{equation}
\mathds{1}_{I}\pare{H^{s,n}_{m}} \left [ H^{s,n}_{m},i\mathcal{A} \right ] \mathds{1}_{I}\pare{H^{s,n}_{m}} \geq \mathds{1}_{I}^{2}\pare{H^{s,n}_{m}} + \mathds{1}_{I}\pare{H^{s,n}_{m}}K\mathds{1}_{I}\pare{H^{s,n}_{m}}
\end{equation}
where $\mathds{1}_{I}$ is the characteristic function of $I$ and $K$ is a compact operator.
\end{prop}

\begin{proof}
We remark that $x A(x) \underset{x \to -\infty,0}{\to} 0$, that $x B(x) \underset{x \to -\infty}{\to} 0 $ and that $x B(x) \underset{x \to 0}{\to} -l $ using the asymptotic behavior of $A$ and $B$ described in \eqref{CompAsymptA} and \eqref{CompAsymptB}. We obtain
\begin{equation*}
\left [ H^{s,n}_{m},i\mathcal{A} \right ] \geq \text{Id} - \pare{2s+1} x A(x) \gamma^{2}\gamma^{1} -  2m xB(x) \gamma^{1}.
\end{equation*}
Consider a compact non-empty interval $I \subset \R$ and $\tilde{I}$ a compact interval strictly containing $I$. Let $\varsigma \in C^{\infty}_{0}\pare{\tilde{I}}$ such that $\varsigma \equiv 1$ on $I$. We have:
\begin{equation}
\varsigma\pare{H^{s,n}_{m}} \left [ H^{s,n}_{m},i\mathcal{A} \right ]\varsigma\pare{H^{s,n}_{m}}  \geq \varsigma^{2}\pare{H^{s,n}_{m}} + K.
\end{equation}
where $K =  \varsigma\pare{H^{s,n}_{m}} \pare{-\pare{2s+1} x A(x) \gamma^{2}\gamma^{1} -  2m xB(x) \gamma^{1}} \varsigma\pare{H^{s,n}_{m}}$ is compact. Indeed, by proposition \ref{3.1} and the use of Helffer-Sjöstrand formula, we see that $\varsigma\pare{H^{s,n}_{m}}$ multiplied by a good function will be compact. The asymptotic behavior of $A$ and $B$ gives that $x A\pare{x}$ and $xB\pare{x}$ are bounded near $0$ and goes to $0$ at $-\infty$. This gives the compacity of K. Multiplying both sides by $\mathds{1}_{I}\pare{H^{s,n}_{m}}$, this gives the desired result since $\mathds{1}_{I}\varsigma = \mathds{1}_{I}$.
\end{proof}
Using the absence of eigenvalues, we deduce the following corollary:
\begin{cor} \label{corKless}
For all $m>0$, all $\lambda \in \R$ and all $0<\epsilon<1$, there exists a compact non-empty interval $I' \subset \R$ containing $\lambda$ such that:
\begin{equation}
\mathds{1}_{I'}\pare{H^{s,n}_{m}} \left [ H^{s,n}_{m},i\mathcal{A} \right ] \mathds{1}_{I'}\pare{H^{s,n}_{m}} \geq \pare{1-\epsilon}\mathds{1}_{I'}^{2}\pare{H^{s,n}_{m}}.
\end{equation}
Recall that $\mathds{1}_{I'}$ is the characteristic function of $I'$.
\end{cor}

\begin{proof}
We have the Mourre estimate with $I$ such that $\lambda \in I$. Let $I' \subset I$ such that $\lambda \in I'$. We can multiply both sides by $\mathds{1}_{I'}\pare{H^{s,n}_{m}}$ to obtain the same inequality with $I$ replaced by $I'$. Since $\lambda$ is not an eigenvalue of $H^{s,n}_{m}$, $\mathds{1}_{I'}\pare{H^{s,n}_{m}}$ tends strongly to $0$ when the size of $I'$ decreases. Then $\mathds{1}_{I'}\pare{H^{s,n}_{m}}K\mathds{1}_{I'}\pare{H^{s,n}_{m}}$ goes to $0$ in the operator norm ($K$ is compact). We can thus choose $I'$ sufficiently small such that the desired inequality holds.
\end{proof}

\section{Propagation estimates}

In this section, we first present abstract results about propagation estimates and the minimal velocity estimate. Then, we apply this to prove that our minimal and maximal velocity is $1$. This will be useful in the proof of asymptotic completeness.

\subsection{Abstract propagation estimates} \label{Estpropaabs}

We present the abstract theory of propagation estimates. Proofs can be found in \cite{DeGe}.\\
\indent Consider a Hilbert space $\mathcal{H}$ and $\pare{H,D\pare{H}}$ a self-adjoint operator on $\mathcal{H}$. Let $\Phi\pare{t}$ be a $C^{1}$ uniformly bounded function with values in $\mathcal{L}\pare{\mathcal{H}}$ defined on $\R^{+}$. We define the Heisenberg derivative of $\Phi$ by:
\begin{equation*}
\mathbb{D} \Phi\pare{t} := \frac{d}{dt} \Phi\pare{t} + i \left[H,\Phi\pare{t} \right].
\end{equation*}

\subsubsection{Basic principle} \label{princbase}
\begin{lem} \cite[Lemma~B.4.1,~B.4.2]{DeGe}
Let $\Phi\pare{t}$ be a $C^{1}$ uniformly bounded function with values in $\mathcal{L}\pare{\mathcal{H}}$ and defined on $\R^{+}$.
\begin{enumerate}
\item[i)] If there exists measurables functions with values in $\mathcal{L}\pare{\mathcal{H}}$ $B\pare{t},B_{i}\pare{t}$, $i=1,\cdots,n$ with
\begin{equation*}
\mathbb{D} \Phi\pare{t}\geq C_{0} B^{*}\pare{t} B\pare{t} - \underset{i=1}{\overset{n}{\sum}} B_{i}^{*}\pare{t} B_{i}\pare{t}
\end{equation*}
such that for all $i\in \left \{1,\cdots,n \right \}$
\begin{equation*}
\ds \int_{1}^{\infty} \norme{B_{i}\pare{t} e^{-itH}u}^{2} dt \leq C \norme{u}^{2}, \hspace{2mm} \forall u \in \mathcal{H}
\end{equation*}
then there exists a constant $C_{1}>0$ such that 
\begin{equation*}
\ds \int_{1}^{\infty} \norme{B\pare{t} e^{-itH}u}^{2} dt \leq C_{1} \norme{u}^{2}, \hspace{2mm} \forall u \in \mathcal{H}.
\end{equation*}
\item[ii)] Suppose that $B_{2,i}\pare{t}$ and $B_{1,i}\pare{t}$ are mesurable functions with value in $\mathcal{L}\pare{\mathcal{H}}$ and that the function $\Phi$ satisfies 
\begin{equation*}
\abs{\prods{\psi_{2}}{\mathbb{D}\Phi\pare{t} \psi_{1}}} \leq \underset{i=1}{\overset{n}{\sum}} \norme{B_{2,i}\pare{t}\psi_{2}}\norme{B_{1,i}\pare{t}\psi_{1}},
\end{equation*}
for all $\psi_{1},\psi_{2} \in \mathcal{H}$, with
\begin{equation*}
\ds \int_{1}^{\infty} \norme{B_{2,i}\pare{t} e^{-itH}u}^{2} dt \leq C_{1} \norme{u}^{2}, \hspace{2mm} \forall u \in \mathcal{H}
\end{equation*}
and
\begin{equation*}
\ds \int_{1}^{\infty} \norme{B_{1,i}\pare{t} e^{-itH}u}^{2} dt \leq C_{1} \norme{u}^{2}, \hspace{2mm} \forall u \in \mathcal{D},
\end{equation*}
where $\mathcal{D}$ is a dense subset of $\mathcal{H}$. Then the limit
\begin{equation*}
s- \underset{t\to \infty}{\lim} e^{itH}\Phi\pare{t}e^{-itH}
\end{equation*}
exists.
\end{enumerate}

\end{lem}

\subsubsection{Abstract minimal velocity estimates}

\begin{prop}\cite[Proposition~A.1]{GeNi}
Let $H \in C^{1+\epsilon}\pare{\mathcal{A}}$ for $\epsilon >0$. Let $\Delta$ be an interval such that
\begin{equation*}
\mathbf{1}_{\Delta} \pare{H} \left[H,i\mathcal{A}\right] \mathbf{1}_{\Delta} \pare{H} \geq c_{0} \mathbf{1}_{\Delta} \pare{H}.
\end{equation*}
Then, for all $g \in C^{\infty}_{0} \pare{\R}$, $\supp g \subset \pare{-\infty,c_{0}}$ and for $f \in C^{\infty}_{0}\pare{\Delta}$, we have
\begin{flalign*}
\hphantom{A} & \ds \int_{1}^{\infty} \norme{g\pare{\frac{\mathcal{A}}{t}} f\pare{H} e^{-itH}u}^{2} \frac{dt}{t} \leq C \norme{u}^{2}, \hspace{2mm} \forall u \in \mathcal{H}, \\
& s- \underset{t\to \infty}{\lim} g\pare{\frac{\mathcal{A}}{t}} f\pare{H} e^{-itH} = 0.
\end{flalign*}
\end{prop}

\subsection{Propagation estimates}

We have seen that $\left [H_{m}^{s,n},i\mathcal{A} \right ]$ admits a bounded extension from $D\pare{\mathcal{A}} \cap D\pare{H_{m}^{s,n}}$ to $D\pare{H_{m}^{s,n}}$. We denote this extension by $\left [H_{m}^{s,n},i\mathcal{A} \right ]_{0}$. We have:
\begin{equation}
\left [ \left [H_{m}^{s,n},i\mathcal{A} \right ]_{0}, i\mathcal{A} \right ]=4 \pare{\pare{s+\frac{1}{2}} x^{2} A\pare{x} \gamma ^{2} \gamma^{0} + m x^{2} B\pare{x} \gamma^{0}}
\end{equation}
so $\left [ \left [H_{m}^{s,n},i\mathcal{A} \right ]_{0}, i\mathcal{A} \right ]$ extends to a bounded operator to $D\pare{H_{m}^{s,n}}$ with values in $\mathcal{H}_{s,n}$. Using lemma \ref{53}, we deduce that $H \in C^{2}\pare{\mathcal{A}}$. Using the Mourre estimate and a partition of unity argument, this gives:
\begin{prop} \label{61}
For all $m>0$, $g\in C^{\infty}_{0}\pare{\R}$, $\supp\pare{g} \subset \pare{-\infty, 1-\delta}$ and $f \in C^{\infty}_{0}\pare{\R}$, we have:
\begin{flalign}
\hphantom{A} & \ds \int_{1}^{\infty} \norme{g\pare{\frac{\mathcal{A}}{t}} f\pare{H_{m}^{s,n}} e^{-itH_{m}^{s,n}}u}^{2} \frac{\mathrm{d}t}{t} \leq C \norme{u}^{2}, \hspace{2mm} \forall u \in \mathcal{H}_{s,n}, \\
& s- \underset{t\to \infty}{\lim} g\pare{\frac{\mathcal{A}}{t}} e^{-itH_{m}^{s,n}} = 0.
\end{flalign}
\end{prop}

\begin{proof}[Proof of proposition \ref{61}] \label{proof61}
Using the corollary \ref{corKless} where we denote $I$ our interval, we obtain
\begin{flalign*}
\hphantom{A} & \ds \int_{1}^{\infty} \norme{g\pare{\frac{\mathcal{A}}{t}} f\pare{H_{m}^{s,n}} e^{-itH_{m}^{s,n}}u}^{2} \frac{\mathrm{d}t}{t} \leq C \norme{u}^{2}, \hspace{2mm} \forall u \in \mathcal{H}_{s,n}, \\
& s- \underset{t\to \infty}{\lim} g\pare{\frac{\mathcal{A}}{t}} f\pare{H_{m}^{s,n}} e^{-itH_{m}^{s,n}} = 0,
\end{flalign*}
for $f\in C^{\infty}_{0}\pare{I}$ by the abstract velocity estimate. For $f \in C^{\infty}_{0}\pare{\R}$, we can cover $\supp \pare{f}$ by a finite number of intervals $I_{1},\cdots,I_{n}$ where a Mourre estimate holds. Then, we consider a partition of unity subordinate to this cover $\eta_{1},\cdots,\eta_{n}$ and we note $f_{i} = \eta_{i} f$ for all $i=1,\cdots,n$. Then:
\begin{flalign*}
\ds \int_{1}^{\infty} \norme{g\pare{\frac{\mathcal{A}}{t}} f\pare{H_{m}^{s,n}} e^{-itH_{m}^{s,n}}u}^{2} \frac{\mathrm{d}t}{t} & \leq  \underset{i=1}{\overset{n}{\ds \sum}} \ds \int_{1}^{\infty} \norme{g\pare{\frac{\mathcal{A}}{t}} f_{i}\pare{H_{m}^{s,n}} e^{-itH_{m}^{s,n}}u}^{2} \frac{\mathrm{d}t}{t}  \\
& \leq C_{n} \norme{u}^{2}, \hspace{2mm} \forall u \in \mathcal{H}_{s,n},
\end{flalign*}
and:
\begin{equation*}
s- \underset{t\to \infty}{\lim} g\pare{\frac{\mathcal{A}}{t}} f\pare{H_{m}^{s,n}} e^{-itH_{m}^{s,n}}  = \underset{i=1}{\overset{n}{\ds \sum}} s- \underset{t\to \infty}{\lim} g\pare{\frac{\mathcal{A}}{t}} f_{i}\pare{H_{m}^{s,n}} e^{-itH_{m}^{s,n}}= 0.
\end{equation*}
Thanks to a density argument, we obtain the desired limit.
\end{proof}
Proposition \ref{61} allows us to obtain:
\begin{lem} \label{lemVitMin}
Let $J_{-} \in C^{\infty}$ such that $\supp \pare{J_{-}} \subset ]-\infty,1-\epsilon[$ and $J_{-}\pare{x} = 1 $ for all $x \in ]-\infty,1-2\epsilon[$ and let $\chi \in C^{\infty}_{0}$. Then, for all $m>0$, we have:
\begin{flalign}
\hphantom{A} & \int_{1}^{\infty} \norme{J_{-} \pare{\frac{\mathcal{A}}{t}} \chi\pare{H_{m}^{s,n}} e^{-itH_{m}^{s,n}} u }^{2} \frac{dt}{t} \leq C \norme{u}^{2}, \hspace{1mm} \forall u \in \mathcal{H}_{s,n} \\
& \lim_{t\to \infty} J_{-} \pare{\frac{\mathcal{A}}{t}}e^{-itH_{m}^{s,n}} u = 0, \hspace{1mm} \forall u \in \mathcal{H}_{s,n}.
\end{flalign}
\end{lem}

\begin{proof}
\begin{enumerate}
\item[1)] Let $\theta_{1},\theta_{2} \in C^{\infty}$ such that $\supp \pare{\theta_{1}} \subset ]-\infty, -1 -\frac{\epsilon}{2}[$, $ \supp \pare{\theta_{2}} \subset ]-1 - \epsilon, 1-\epsilon[$ and $\theta_{1}+\theta_{2} = 1$. Then, using the triangular inequality and the minimal velocity estimate, we only need to prove the integral estimate for $\theta_{1} J_{-}$. \\
\indent So suppose that $K \in C^{\infty}$ such that $\supp \pare{K} \subset ]-\infty, -1 - \frac{\epsilon}{2}[$ and $K\pare{x} =1$ for all $x \in ]-\infty, -1 - \epsilon [$. We define $F\pare{s} = \int_{s}^{\infty} K^{2}\pare{t} dt$ and 
\begin{equation*}
\Phi\pare{t} = \chi\pare{H_{m}^{s,n}} F\pare{\frac{\mathcal{A}}{t}} \chi\pare{H_{m}^{s,n}}
\end{equation*}
such that $\Phi$ is $C^{1}$ uniformly bounded. We have:
\begin{equation*}
\mathbb{D} \Phi \pare{t} = \frac{1}{t} \chi\pare{H_{m}^{s,n}} \frac{\mathcal{A}}{t} K^{2}\pare{\frac{\mathcal{A}}{t}} \chi\pare{H_{m}^{s,n}} + i \chi\pare{H_{m}^{s,n}} \left [ H_{m}^{s,n} , F\pare{\frac{\mathcal{A}}{t}} \right ] \chi\pare{H_{m}^{s,n}},
\end{equation*}
where
\begin{flalign*}
 \left [ H_{m}^{s,n} , F\pare{\frac{\mathcal{A}}{t}} \right ]& = \frac{i}{t} K^{2}\pare{\frac{\mathcal{A}}{t}} + \pare{s+\frac{1}{2}} A\pare{x} \pare{F\pare{-\frac{x}{t}} - F\pare{\frac{x}{t}}} \gamma^{1}\gamma^{2} \\
& \quad - m B\pare{x} \pare{F\pare{-\frac{x}{t}} - F\pare{\frac{x}{t}}} \gamma^{1},
\end{flalign*}
with
\begin{equation*}
\abs{F\pare{\frac{-x}{t}} - F\pare{\frac{x}{t}}} \leq - \frac{2x}{t} \sup_{y\in \left [ \frac{x}{t}, - \frac{x}{t} \right ]} K^{2} \pare{y} \leq  - \frac{2x}{t} \mathds{1}_{\left \{x\leq \pare{-1 - \frac{\epsilon}{2}} t \right \}},
\end{equation*}
where $\mathds{1}$ is the characteristic function and $\sup_{y\in \left [ \frac{x}{t}, - \frac{x}{t} \right ]}K^{2} \pare{y}$ is thought as a function depending on the variables $x$ and $t$. We know that for $x<0$ and $\abs{x}$ sufficiently large, the functions $A$ and $B$ are exponentially decaying. If we fix $T$ sufficiently large, then, since $e^{x} \leq \frac{1}{-x^{3}}$ for $x$ sufficiently small, for all $t \geq T$, we have:
\begin{equation*}
\abs{A\pare{x} \pare{F\pare{-\frac{x}{t}} - F\pare{\frac{x}{t}}}} \leq \frac{C}{t^{2}} \zeta_{\left \{x\leq \pare{-1 - \frac{\epsilon}{2}} T \right \}}.
\end{equation*}
We can do the same thing with $B$. We obtain:
\begin{flalign*}
-\mathbb{D} \Phi \pare{t} & = \frac{1}{t} \chi\pare{H_{m}^{s,n}} \pare{1-\frac{\mathcal{A}}{t}} K^{2}\pare{\frac{\mathcal{A}}{t}} \chi\pare{H_{m}^{s,n}} + O\pare{t^{-2}} \notag \\
& \geq  \frac{2+\frac{\epsilon}{2}}{t} \chi\pare{H_{m}^{s,n}} K^{2}\pare{\frac{\mathcal{A}}{t}} \chi\pare{H_{m}^{s,n}} + O\pare{t^{-2}},
\end{flalign*}
since $\frac{\mathcal{A}}{t} \leq -1-\frac{\epsilon}{2}$ on the support of $K^{2}$. By lemma \ref{princbase}, this shows that: 
\begin{equation} \label{EstK}
\int_{1}^{\infty} \norme{K \pare{\frac{\mathcal{A}}{t}} \chi\pare{H_{m}^{s,n}} e^{-itH_{m}^{s,n}} u }^{2} \frac{dt}{t} \leq C \norme{u}^{2}
\end{equation}
for all $u \in \mathcal{H}_{s,n}$. This proves the first statement of the lemma.
\item[2)]
We next set: 
\begin{equation*}
\Phi\pare{t} = \chi\pare{H_{m}^{s,n}} J_{-}^{2}\pare{\frac{\mathcal{A}}{t}} \chi\pare{H_{m}^{s,n}}.
\end{equation*}
So, we have:
\begin{equation*}
\mathbb{D} \Phi\pare{t} \leq \frac{4\epsilon}{t} \chi\pare{H_{m}^{s,n}} \pare{J_{-}^{'}J_{-}}\pare{\frac{\mathcal{A}}{t}} \chi\pare{H_{m}^{s,n}} + O\pare{t^{-2}}
\end{equation*}
where $\supp \pare{J_{-}' J_{-}} \subset ]1-2\epsilon,1-\epsilon[$ so it is integrable by the minimal velocity estimate. Using lemma \ref{princbase} and the integrability in \ref{EstK},  this gives
\begin{equation*}
\lim_{t \to \infty} e^{itH_{m}^{s,n}} \chi\pare{H_{m}^{s,n}} J_{-}^{2} \pare{\frac{\mathcal{A}}{t}} \chi\pare{H_{m}^{s,n}} e^{-itH_{m}^{s,n}} u = 0, \hspace{1mm} \forall u \in \mathcal{H}_{s,n}.
\end{equation*}
Using the last lemma, we obtain the desired limit by a density argument.
\end{enumerate}
\end{proof}

\begin{prop} \label{prop62}
Let $g \in C^{\infty}$ such that $\supp \pare{g} \subset ]1+\epsilon,\infty[$ with $\epsilon > 0$ and such that $g\pare{x} = 1$ for all $x\in ]1+2\epsilon,\infty[$. Let $\zeta \in C^{\infty}_{0}\pare{\R}$. Then, for all $m>0$, we have:
\begin{flalign}
\hphantom{A} & \ds \int_{1}^{\infty} \norme{g\pare{\frac{\mathcal{A}}{t}} e^{-it H_{m}^{s,n}} \zeta\pare{H_{m}^{s,n}} u }^{2} \frac{\mathrm{d}t}{t} \leq C \norme{u}^{2}, \hspace{2mm} \forall u \in \mathcal{H}_{s,n} \\
& s-\underset{t\to \infty}{\lim} g\pare{\frac{A}{t}} e^{-it H_{m}^{s,n}} = 0.
\end{flalign}
\end{prop}

\begin{proof}[Proof of the proposition \ref{prop62}]
Let $J \in C^{\infty}\pare{\R}$ such that $\supp \pare{J} \subset \pare{1+ \epsilon,+\infty}$ with $\epsilon>0$ and $J\pare{x} = 1$ for all $x \in]1+2\epsilon,+\infty[$. Let $\zeta \in C^{\infty}_{0}\pare{\R}$. We define 
\begin{equation*}
F\pare{s} = \ds \int_{-\infty}^{s} J^{2}\pare{u} \mathrm du
\end{equation*}
and 
\begin{equation*}
\Phi\pare{t} = \zeta\pare{H_{m}^{s,n}} F\pare{\frac{\mathcal{A}}{t}} \zeta\pare{H_{m}^{s,n}}
\end{equation*}
so that $\Phi$ is $C^{1}$ uniformly bounded. As in the last proof, we calculate the Heisenberg derivative of $\Phi$ and thanks to the support of $J$, we obtain:
\begin{flalign}
- \mathbb{D}\Phi\pare{t} & \geq \frac{\epsilon}{t} \zeta\pare{H_{m}^{s,n}} J^{2}\pare{\frac{\mathcal{A}}{t}} \zeta\pare{H_{m}^{s,n}}  + \zeta\pare{H_{m}^{s,n}} \left (i\pare{s+\frac{1}{2}} A\pare{x} \right . \notag \\
& \quad \left . \pare{F\pare{\frac{-x}{t}}-F\pare{\frac{x}{t}}} \gamma^{2} \gamma^{1}  + i m B\pare{x} \pare{F\pare{\frac{-x}{t}}-F\pare{\frac{x}{t}}} \gamma^{1} \right )\zeta\pare{H_{m}^{s,n}},
\end{flalign}
and we have:
\begin{equation*}
\abs{F\pare{\frac{-x}{t}}-F\pare{\frac{x}{t}}} \leq \frac{-2x}{t} \underset{y\in \left [ \frac{x}{t},\frac{-x}{t} \right ]}{\sup} J^{2}\pare{y} \mathds{1}_{\left \{1+\epsilon \leq \frac{-x}{t} \right \}}.
\end{equation*}
Using the exponential decay of $A$ and $B$, we obtain:
\begin{flalign}
\hphantom{A} & \zeta\pare{H_{m}^{s,n}} \left (i\pare{s+\frac{1}{2}} A\pare{x} 
 \pare{F\pare{\frac{-x}{t}}-F\pare{\frac{x}{t}}} \gamma^{2} \gamma^{1} \right . \notag \\
 & \left . + i m B\pare{x} \pare{F\pare{\frac{-x}{t}}-F\pare{\frac{x}{t}}} \gamma^{1} \right )\zeta\pare{H_{m}^{s,n}} = O\pare{e^{-\frac{\kappa}{2}t}}
\end{flalign}
for $t$ sufficiently large. We deduce that:
\begin{equation} \label{617}
\ds \int_{1}^{\infty} \norme{J\pare{\frac{\mathcal{A}}{t}} e^{-it H_{m}^{s,n}} \zeta\pare{H_{m}^{s,n}} u }^{2} \frac{\mathrm{d}t}{t} \leq C \norme{u}^{2}, \hspace{2mm} \forall u \in \mathcal{H}_{s,n}.
\end{equation}
Next, we use:
\begin{equation*}
\Phi\pare{t} = \zeta\pare{H_{m}^{s,n}} J^{2}\pare{\frac{\mathcal{A}}{t}} \zeta\pare{H_{m}^{s,n}},
\end{equation*}
and obtain:
\begin{flalign*}
\mathbb{D}\Phi\pare{t} & = \frac{2}{t} \zeta\pare{H_{m}^{s,n}} \frac{-\mathcal{A}}{t} J\pare{\frac{\mathcal{A}}{t}} J'\pare{\frac{\mathcal{A}}{t}} \zeta\pare{H_{m}^{s,n}} + \frac{2}{t} \zeta\pare{H_{m}^{s,n}}  J\pare{\frac{\mathcal{A}}{t}} J'\pare{\frac{\mathcal{A}}{t}} \zeta\pare{H_{m}^{s,n}}  \\
& \quad + \zeta\pare{H_{m}^{s,n}} \left (i\pare{s+\frac{1}{2}} A\pare{x} 
 \pare{J^{2}\pare{\frac{-x}{t}}-J^{2}\pare{\frac{x}{t}}} \gamma^{2} \gamma^{1} \right .  \\
 & \left . \quad + i m B\pare{x} \pare{J^{2}\pare{\frac{-x}{t}}-J^{2}\pare{\frac{x}{t}}} \gamma^{1} \right )\zeta\pare{H_{m}^{s,n}}.
\end{flalign*}
The first two terms are integrable due to the support of $J$ and \eqref{617}. The last two are also integrable using the support of $J$. Consequently:
\begin{equation*}
s-\underset{t\to \infty}{\lim} J\pare{\frac{\mathcal{A}}{t}} e^{-it H_{m}^{s,n}} \zeta\pare{H_{m}^{s,n}} 
\end{equation*}
exists and is zero by \eqref{617}. The proposition follows by density.
\end{proof}

\section{Asymptotic completeness}

\subsection{Comparison operator}
Our comparison operator will be $H_{c}$ defined by:
\begin{equation}
H_{c} = i \gamma^{0} \gamma^{1} \partial_{x}
\end{equation}
where $\gamma^{0} \gamma^{1} = \text{diag}\pare{-1,1,1,-1}$ and with domain:
\begin{equation}
D\pare{H_{c}}= \left \{ \varphi \in \mathcal{H}_{s,n}; H_{c} \varphi \in \mathcal{H}_{s,n}, \hspace{1mm} \varphi_{1} \pare{0} = - \varphi_{3} \pare{0}, \hspace{1mm} \varphi_{2} \pare{0} = \varphi_{4} \pare{0} \right \}
\end{equation}
By proposition \ref{propH0aa}, this is a self-adjoint operator on its domain.

\subsection{Asymptotic completeness}

Recall that $\mathcal{A} = \Gamma x$ where $\Gamma = - \gamma^{0} \gamma^{1}$. We have:

\begin{theo}[Asymptotic completeness for fixed harmonics]
For all $m>0$ and all $\varphi \in \mathcal{H}_{s,n}$, the limits
\begin{flalign}
\hphantom{A} & \lim_{t \to \infty} e^{it H_{c}} e^{-itH_{m}^{s,n}} \varphi \\
& \lim_{t \to \infty} e^{itH_{m}^{s,n}} e^{-itH_{c}} \varphi
\end{flalign}
exist. If we denote them by:
\begin{flalign}
\hphantom{A} & \Omega_{s,n} \varphi =\lim_{t \to \infty} e^{it H_{c}} e^{-itH_{m}^{s,n}} \varphi \\
& W_{s,n} \varphi = \lim_{t \to \infty} e^{itH_{m}^{s,n}} e^{-itH_{c}} \varphi
\end{flalign}
for all $\varphi \in \mathcal{H}_{s,n}$, we have $\Omega_{s,n}^{*} = W_{s,n}$.

\end{theo}

\begin{proof}
Let $J_{-},J_{0},J_{+} \in C^{\infty}$ such that $J_{-} + J_{0} + J_{+} = 1$, the supports of $J_{-},J_{+}$ are as in \ref{prop62} and \ref{lemVitMin}, and $J_{0}=1$ on $]1-\epsilon,1+\epsilon[$, $\supp \pare{J_{0}} \subset ]1-2\epsilon,1+2\epsilon[$ with $\epsilon >0$. Using proposition \ref{prop62} and lemma \ref{lemVitMin}, it suffices to prove that, for all $\varphi \in \mathcal{H}_{s,n}$, the limit:
\begin{equation*}
\lim_{t \to \infty} e^{itH_{c}} J_{0}\pare{\frac{\mathcal{A}}{t}} e^{-itH_{m}^{s,n}} \varphi
\end{equation*}
exists. We remark that $J_{0}\pare{\frac{x}{t}} \neq 0$ if and only if $ x \geq \pare{1-2\epsilon}t >0$. Since $x<0$, $J_{0}\pare{\frac{x}{t}} = 0$, for all $t>0$ and $x<0$. We thus have:
\begin{equation*}
J_{0}\pare{\frac{\mathcal{A}}{t}} = J_{0}\pare{\frac{-x}{t}} M_{0}
\end{equation*}
where $M_{0}= \text{diag} \pare{0,1,1,0}$.  We then define:
\begin{equation*}
\Phi\pare{t} = \chi \pare{H_{c}} J_{0}\pare{\frac{\mathcal{A}}{t}} \chi \pare{H_{m}^{s,n}},
\end{equation*}
and, denoting $V\pare{x} = \pare{s + \frac{1}{2}} A\pare{x} \gamma^{1}\gamma^{2} - m B\pare{x} \gamma^{0}$, we have:
\begin{flalign*}
\mathbb{D} \Phi\pare{t}& = \frac{d}{dt} \Phi\pare{t} + i \pare{H_{c} \Phi\pare{t} - \Phi\pare{t} H_{m}^{s,n}} \\
&= \frac{2}{t} \chi\pare{H_{c}} \pare{\frac{x}{t} +1} \pare{J_{0}'J_{0}}\pare{\frac{-x}{t}} M_{0} \chi\pare{H_{m}^{s,n}}  - i \chi\pare{H_{c}} J_{0}^{2}\pare{\frac{-x}{t}} M_{0} V\pare{x} \chi\pare{H_{m}^{s,n}}.
\end{flalign*}
On the support of $J_{0}'J_{0}$, we have $\frac{x}{t} + 1 \leq 2\epsilon$. Moreover, $J_{0}\pare{\frac{-x}{t}} \neq 0$ if and only if $-\pare{1+2\epsilon} t \leq x \leq - \pare{1-2\epsilon}t$. Since $A,B$ are exponentially decreasing at $-\infty$, we obtain:
\begin{equation*}
\mathbb{D} \Phi \pare{t} \leq \frac{4\epsilon}{t} \chi\pare{H_{c}}  \pare{J_{0}'J_{0}}\pare{\frac{\mathcal{A}}{t}}  \chi\pare{H_{m}^{s,n}} + O\pare{t^{-2}}.
\end{equation*}
Using the support of $J_{0}'J_{0}$, minimal and maximal velocity estimates, the right hand side is integrable. Hence the limit exists. We can show that the second limit exists in the same way. Finally, for all $t>0$ and $\varphi,\psi \in \mathcal{H}_{s,n}$, we have $\prods{e^{itH_{c}}e^{-itH_{m}^{s,n}}\varphi}{\psi} = \prods{\varphi}{e^{itH_{m}^{s,n}}e^{-itH_{c}}\psi}$ which proves the last statement.
\end{proof}
Therefore, we obtain:
\begin{theo}[Asymptotic completeness] \label{Théorème6.5}
For all $m>0$ and all $\varphi \in \mathcal{H}$, the limits:
\begin{flalign}
\hphantom{A} & \lim_{t \to \infty} e^{it H_{c}} e^{-itH_{m}} \varphi \\
& \lim_{t \to \infty} e^{itH_{m}} e^{-itH_{c}} \varphi
\end{flalign}
exist. If we denote these limits by $\Omega \varphi$ and $W \varphi$ respectively, we have $\Omega^{*} = W$.
\end{theo}

\begin{proof}
We can decompose $\varphi = \underset{\pare{s,n}\in I}{\sum} \varphi_{s,n}$ where $\varphi_{s,n} \in \mathcal{H}_{s,n}$ and $\underset{\pare{s,n}\in I}{\sum} \norme{\varphi_{s,n}}^{2}_{\mathcal{H}_{s,n}} < \infty$. We have:
\begin{equation*}
e^{it H_{c}} e^{-itH_{m}} \varphi = \underset{\pare{s,n}\in I}{\sum} e^{it H_{c}} e^{-itH_{m}^{s,n}}\varphi_{s,n}.
\end{equation*}
Since $\lim_{t \to \infty} e^{it H_{c}} e^{-itH_{m}^{s,n}}\varphi_{s,n} = \Omega_{s,n} \varphi_{s,n}$ exists for all $\pare{s,n} \in I$ and $e^{it H_{c}} e^{-itH_{m}^{s,n}}$ is unitary, we deduce, using the dominate convergence theorem, that the limit in the theorem exists. We can do the same for the other limit. The last statement follows as in the last proof.
\end{proof}

\section{Asymptotic velocity}

\subsection{Abstract theory}
In this section, we follow the appendix $B.2$ in \cite{DeGe}. We consider a sequence $\pare{B_{n}}_{n\in \N}$ of vectors of self-adjoint operators which commute in a Hilbert space $\mathcal{H}$. More precisely:
\begin{equation*}
B_{n}= \pare{B^{1}_{n},\cdots,B^{m}_{n}},\hspace{5mm} \left [ B^{i}_{n}, B^{j}_{n} \right] = 0, \hspace{5mm} 0\leq i,j \leq m, \hspace{5mm} n=1,2,\cdots .
\end{equation*}
We have the following proposition:
\begin{prop}
Suppose that, for all $g \in C_{\infty}\pare{\R^{m}}$, there exists
\begin{equation}\label{B.2.1DEGE}
\mathrm{s}-\underset{n \to \infty}{\lim} g \pare{B_{n}}.
\end{equation}
Then there exists a unique vector of self-adjoint operators 
\begin{equation}
B= \pare{B^{1},\cdots,B^{m}}
\end{equation}
such that \eqref{B.2.1DEGE} is equal to $g\pare{B}$. $B$ is densely defined if, for some $g\in C_{\infty}\pare{\R^{m}}$ such that $g\pare{0}=1$, we have:
\begin{equation}
\mathrm{s}-\underset{R \to \infty}{\lim}\pare{\mathrm{s}- \underset{t \to \infty}{\lim} g\pare{R^{-1}B_{n}}} = \mathds{1}.
\end{equation}
\end{prop}
We then define:
\begin{defi}
Under the hypotheses of the preceding proposition, we will write:
\begin{equation}
B= \mathrm{s}-C_{\infty}- \underset{n \to \infty}{\lim}B_{n}.
\end{equation}
\end{defi}

\subsection{Asymptotic velocity for \texorpdfstring{$H_{c}$}{H0}}

\begin{theo}[Asymptotic velocity for $H_{c}$]
Let $J \in C_{\infty} \pare{\R}$. Then the limit:
\begin{equation}
\mathrm{s}-\underset{t \to \infty}{\lim} e^{itH_{c}} J\pare{\frac{\mathcal{A}}{t}} e^{-itH_{c}}
\end{equation}
exists and is equal to $J\pare{1} \mathds{1}$ where $\mathds{1}$ is the identity. Moreover, if $J\pare{0} = 1$, then
\begin{equation}
\mathrm{s}- \underset{R\to \infty}{\lim} \pare{\mathrm{s} - \underset{t \to \infty}{\lim} e ^{itH_{c}} J\pare{\frac{\mathcal{A}}{Rt}} e^{-itH_{c}}} = \mathds{1}.
\end{equation}
If we define
\begin{equation}
\mathrm{s}- \mathrm{C_{\infty}}-\underset{t \to \infty}{\lim} e^{itH_{c}} \frac{\mathcal{A}}{t} e^{-itH_{c}} =: P^{+}_{c},
\end{equation}
then the self-adjoint operator $P^{+}_{c}$ is densely defined and it commutes with $H_{c}$. $P^{+}_{c}$ is called the asymptotic velocity.
\end{theo}

\begin{proof}
Recall that $\mathcal{A} = - \gamma^{0} \gamma^{1} x$ where $ - \gamma^{0} \gamma^{1} = \text{diag}\pare{1,-1,-1,1}$. Thus, for $J \in C_{\infty}\pare{\R}$, we have $J\pare{\frac{\mathcal{A}}{t}} = \text{diag}\pare{J\pare{\frac{x}{t}}, J\pare{- \frac{x}{t}}, J\pare{-\frac{x}{t}}, J\pare{\frac{x}{t}}}$. Moreover, we have $H_{c} = i \gamma^{0} \gamma^{1} \partial_{x}$. Let $\psi^{0} \in D\pare{H_{c}}$, we wish to solve the equation 
\begin{flalign*}
\hphantom{A} & \partial_{t} \psi\pare{t,x} = i H_{c} \psi \pare{t,x},\\
& \psi\pare{0,.} = \psi^{0}\pare{.} = \pare{\psi^{0}_{1}\pare{.}, \psi^{0}_{2}\pare{.}, \psi^{0}_{3}\pare{.} , \psi^{0}_{4}\pare{.}}
\end{flalign*}
where $i H_{c} = \text{diag}\pare{1,-1,-1,1} \partial_{x}$. We will prove that the formula:
\begin{equation*}
\psi\pare{t,x} = \begin{pmatrix}
\psi_{1}^{0}\pare{x+t} \mathds{1}_{\R^{-}}\pare{x+t} - \psi_{3}^{0}\pare{-\pare{x+t}} \mathds{1}_{\R^{+}}\pare{x+t} \\
\psi_{2}^{0}\pare{x-t} \mathds{1}_{\R^{-}}\pare{x-t} + \psi_{4}^{0}\pare{-x +t} \mathds{1}_{\R^{+}}\pare{x-t} \\
\psi_{3}^{0}\pare{x-t} \mathds{1}_{\R^{-}}\pare{x-t} - \psi_{1}^{0}\pare{-x+t} \mathds{1}_{\R^{+}}\pare{x-t} \\
\psi_{4}^{0}\pare{x+t} \mathds{1}_{\R^{-}}\pare{x+t} + \psi_{2}^{0}\pare{-\pare{x+t}} \mathds{1}_{\R^{+}}\pare{x+t}
\end{pmatrix}.
\end{equation*}
gives an explicit solution for this problem. Since $x<0$ in our case, $\mathds{1}_{\R^{+}}\pare{x-t} = 0$ for all $t>0$, but we need this term for the group property of this solution. \\
We first prove that our formula gives a solution of the desired equation. Indeed, for all $t>0$, we see that $\psi_{3}\pare{t,0} = \psi_{3}^{0}\pare{-t}$ and $\psi_{1}\pare{t,0} = - \psi_{3}^{0}\pare{-t}$ since $\mathds{1}_{\R^{-}}\pare{t} = 0$ for $t>0$. Thus $\psi_{3}\pare{t,0}= - \psi_{1}\pare{t,0}$. On the other hand, we have $\psi_{2}\pare{t,0} = \psi_{2}^{0}\pare{-t}$ and $\psi_{4}\pare{t,0} = \psi_{2}^{0}\pare{-t}$ which gives us $\psi_{2}\pare{t,0} = \psi_{4}\pare{t,0}$. The boundary conditions are thus satisfied. It remains to prove that it satisfies the equation. For the first component of our formula, using the boundary consition and the derivation in the distributional sense, we obtain:
\begin{equation*}
\partial_{t} \psi_{1}\pare{t,x}  = \psi_{1}^{0\hspace{1mm}'}\pare{x+t}\mathds{1}_{\R^{-}}\pare{x+t} + \psi_{3}^{0\hspace{1mm}'}\pare{-\pare{x+t}} \mathds{1}_{\R^{+}}\pare{x+t}.
\end{equation*}
We also have:
\begin{equation*}
\partial_{x} \psi_{1}^{0} \pare{t,x}  = \psi_{1}^{0\hspace{1mm}'}\pare{x+t}\mathds{1}_{\R^{-}}\pare{x+t} + \psi_{3}^{0\hspace{1mm}'}\pare{-\pare{x+t}} \mathds{1}_{\R^{+}}\pare{x+t}
\end{equation*}
which gives $\partial_{t} \psi_{1}\pare{t,x} = \partial_{x} \psi_{1} \pare{t,x}$. For the second and third components, $\mathds{1}_{\R^{-}}\pare{x-t}$ is constant so its derivative is $0$ and we can check that $\partial_{t} \psi_{2}\pare{t,x} = - \partial_{x} \psi_{2}\pare{t,x}$ and $\partial_{t} \psi_{3}\pare{t,x} = - \partial_{x} \psi_{3}\pare{t,x}$. For the fourth component, we obtain:
\begin{equation*}
\partial_{t} \psi_{4}\pare{t,x} = \psi_{4}^{0\hspace{1mm}'}\pare{x+t} \mathds{1}_{\R^{-}}\pare{x+t} - \psi_{2}^{0\hspace{1mm}'}\pare{-\pare{x+t}}.
\end{equation*}
We have the same for $ \partial_{x} \psi_{4}\pare{t,x}$ so that $\partial_{t} \psi_{4}\pare{t,x} = \partial_{x} \psi_{4}\pare{t,x}$. So $\partial_{t} \psi \pare{t,x} = iH_{c}\psi \pare{t,x}$ in the sense of distribution. Since $\psi^{0} \in D\pare{H_{c}}$, the derivatives are, in fact, well defined in $\mathcal{H}_{s,n}$ and the equality is satisfied in $\mathcal{H}_{s,n}$. We thus have a solution.

~\\
We then turn our attention to the asymptotic velocity. We have:
\begin{flalign*}
\hphantom{A}& e^{itH_{c}}J\pare{\frac{\mathcal{A}}{t}} e^{-itH_{c}} \psi^{0} \\
& = \begin{pmatrix}
J\pare{\frac{x}{t}+1}\pare{\psi_{1}^{0}\pare{x} \mathds{1}_{\R^{-}}\pare{x} \mathds{1}_{\R^{-}}\pare{x+t} + \psi_{1}^{0}\pare{x} \mathds{1}_{\R^{+}}\pare{-x} \mathds{1}_{\R^{+}}\pare{x+t}} \\
J\pare{-\frac{x}{t}+1}\pare{\psi_{2}^{0}\pare{x} \mathds{1}_{\R^{-}}\pare{x} \mathds{1}_{\R^{-}}\pare{x-t}} \\
J\pare{-\frac{x}{t}+1}\pare{\psi_{3}^{0}\pare{x} \mathds{1}_{\R^{-}}\pare{x} \mathds{1}_{\R^{-}}\pare{x-t}} \\
J\pare{\frac{x}{t}+1}\pare{\psi_{4}^{0}\pare{x} \mathds{1}_{\R^{-}}\pare{x} \mathds{1}_{\R^{-}}\pare{x+t} + \psi_{4}^{0}\pare{x} \mathds{1}_{\R^{+}}\pare{-x} \mathds{1}_{\R^{+}}\pare{x+t}}
\end{pmatrix}.
\end{flalign*}
This last term converges pointwise to $J\pare{1} \psi^{0}\pare{x}$ as $t \to  \infty$. Since $J$, $\mathds{1}_{\R^{-}}$, $\mathds{1}_{\R^{+}}$, $\mathds{1}_{\R^{-}}$ are bounded and $\psi^{0} \in \mathcal{H}_{s,n}$, we can use the dominate convergence theorem to conclude that: 
\begin{equation*}
\lim_{t\to \infty} e^{itH_{c}}J\pare{\frac{\mathcal{A}}{t}} e^{-itH_{c}} \psi^{0} = J\pare{1} \psi^{0}.
\end{equation*}
If $J\in C_{\infty}\pare{\R}$ with $J\pare{0} =1$, then
\begin{equation*}
\lim_{t\to \infty} e^{itH_{c}}J\pare{\frac{\mathcal{A}}{Rt}} e^{-itH_{c}} \psi^{0} = J\pare{\frac{1}{R}} \psi^{0},
\end{equation*}
and the last term goes to $J\pare{0} \psi^{0} = \psi^{0}$. So
\begin{equation*}
\mathrm{s}- \underset{R\to \infty}{\lim} \pare{\mathrm{s} - \underset{t \to \infty}{\lim} e ^{itH_{c}} J\pare{\frac{\mathcal{A}}{Rt}} e^{-itH_{c}}} = \mathds{1}.
\end{equation*}
The last part of the theorem follows from the abstract theory.
\end{proof}

We can know study the spectrum of $P^{+}_{c}$:
\begin{prop}
$\sigma \pare{P^{+}_{c}} = \{1 \}$
\end{prop}

\begin{proof}
Let $J \in C_{\infty}\pare{\R}$ such that $J\pare{1} = 0$. We can approach $J$ by a sequence $\pare{J_{n}}_{n\in \N}$ of $C^{\infty}_{0}\pare{\R}$ functions which are zero in a neighbourhood of $1$ in $L^{\infty}$. By density, we can suppose that $J \in C^{\infty}_{0}\pare{\R}$ and $J$ is zero in a neighbourhood of $1$. Using minimal and maximal velocity estimates, we obtain:
\begin{equation}
J\pare{P^{+}_{c}} = \mathrm{s}-\underset{t \to \infty}{\lim} e^{itH_{c}} J\pare{\frac{\mathcal{A}}{t}} e^{-itH_{c}} = 0
\end{equation}
Now, if we have $J\pare{1} \neq 0$, we can suppose that $J \in C^{\infty}_{0}\pare{\R}$ is constant, non zero, in a neighbourhood of $1$. Then, for all $\varphi \in \mathcal{H}$, we have:
\begin{equation*}
J\pare{P^{+}_{c}} \varphi - J\pare{1} \varphi = \mathrm{s}-\underset{t \to \infty}{\lim} e^{itH_{c}} \pare{J\pare{\frac{\mathcal{A}}{t}}-J\pare{1}} e^{-itH_{c}} \varphi.
\end{equation*}
Since $J\pare{x} - J\pare{1}$ is zero in a neighbourhood of $1$, we obtain $J\pare{P^{+}_{c}} \varphi = J\pare{1} \varphi \neq 0$. This ends the proof.
\end{proof}
The following consequence is immediate:
\begin{cor}
$P^{+}_{c}= \mathds{1}$
\end{cor}

\subsection{Asymptotic velocity for \texorpdfstring{$H_{m}$}{H0}}

\begin{theo}[Asymptotic velocity for $H_{m}$]
Let $J \in C_{\infty} \pare{\R}$. Then, for all $m >0$, the limit:
\begin{equation}
\mathrm{s}-\underset{t \to \infty}{\lim} e^{itH_{m}} J\pare{\frac{\mathcal{A}}{t}} e^{-itH_{m}}
\end{equation}
exists. Moreover, if $J\pare{0} = 1$, then
\begin{equation}
\mathrm{s}- \underset{R\to \infty}{\lim} \pare{\mathrm{s} - \underset{t \to \infty}{\lim} e ^{itH_{m}} J\pare{\frac{\mathcal{A}}{Rt}} e^{-itH_{m}}} = 1
\end{equation}
If we define
\begin{equation}
\mathrm{s}- \mathrm{C_{\infty}}-\underset{t \to \infty}{\lim} e^{itH_{m}} \frac{\mathcal{A}}{t} e^{-itH_{m}} =: P^{+}_{m},
\end{equation}
then the self-adjoint operator $P^{+}_{m}$ is densely defined and commutes with $H_{m}$. The operator $P^{+}_{m}$ is called the asymptotic velocity.
\end{theo}

\begin{proof}
We can write
\begin{equation*}
e^{itH_{m}} J\pare{\frac{\mathcal{A}}{t}} e^{-itH_{m}} = e^{itH_{m}} e^{itH_{c}} e^{itH_{c}} J\pare{\frac{\mathcal{A}}{t}} e^{-itH_{c}}e^{itH_{c}} e^{-itH_{m}}
\end{equation*}
Using uniform boundedness of our operators and introducing $\Omega$ and $W$ at the right place, this limit is equal to $W J\pare{P^{+}_{c}} \Omega$ where $W,\Omega$ are defined in theorems \ref{Théorème6.5}. We can use the same argument for the second limit and the existence of $P^{+}_{m}$ follows by the abstract theory and we have:
\begin{equation}
J\pare{P^{+}_{m}} = W J\pare{P^{+}_{c}} \Omega
\end{equation}

\end{proof}
We deduce:
\begin{prop}
For all $m>0$, $\sigma \pare{P^{+}_{m}} = \{1\}$
\end{prop}

\begin{proof}
Using the last proof, we have:
\begin{equation*}
J\pare{P^{+}_{m}} = W J\pare{P^{+}_{c}} \Omega
\end{equation*}
for all $J \in C_{\infty}\pare{\R}$ where  $\Omega,W$ are unitary and $\Omega^{-1} = W$.
\end{proof}
We then have the following consequence:
\begin{cor}
For all $m>0$, $P^{+}_{m} = \mathds{1}$.
\end{cor}

\renewcommand{\refname}{References}
\bibliographystyle{plain}
\bibliography{BiblioSadSDirac2}

\begin{thebibliography}{10}

\bibitem{ABG}
W.O. Amrein, A.~Boutet~de Monvel, and V.~Georgescu.
\newblock {\em "{C0-Groups}, {Commutator} {Methods} {and} {Spectral} {Theory}
  {of} {N-Body} {Hamiltonians}"}, volume 135 of {\em Progress in Mathematics}.
\newblock Birkh\"auser, 1996.

\bibitem{AIS}
S.J. Avis, C.J. Isham, and D.~Storey.
\newblock Quantum field theory in {anti}-{de} sitter {space}-{time}.
\newblock {\em Phys. Rev. D}, 18(10):3565--3576, 1978.

\bibitem{Bachelot2}
A.~Bachelot.
\newblock Gravitational scattering of elecromagnetic field by a {Schwarzschild}
  black hole.
\newblock {\em Ann. Inst. H. Poincar\'{e} Phys. Th\'{e}or.}, 54:261--320, 1991.

\bibitem{Bachelot3}
A.~Bachelot.
\newblock Asymptotic completeness for the {Klein}-{Gordon} equation on the
  {Schwarzschild} metric.
\newblock {\em Ann. Inst. H. Poincar\'{e} Phys. Th\'{e}or.}, 61(4):411--441,
  1994.

\bibitem{Bac97}
A.~Bachelot.
\newblock Quantum vacuum polarization at the black hole horizon.
\newblock {\em Ann. Inst. H. Poincar\'{e} Phys. Th\'{e}or.}, 67(2):182--222,
  1997.

\bibitem{Ba99}
A.~Bachelot.
\newblock The {Hawking} effect.
\newblock {\em Ann. Inst. H. Poincar\'{e} Phys. Th\'{e}or.}, 70:41--99, 1999.

\bibitem{Ba00}
A.~Bachelot.
\newblock Creation of fermions at the charged black hole horizon.
\newblock {\em Ann. H. Poincar\'{e}}, 1(6):1043--1095, 2000.

\bibitem{Bachelot}
A.~Bachelot.
\newblock "{The} {Dirac} {System} {On} {The} {Anti}-{De} {Sitter} {Universe}".
\newblock {\em Commun. Math. Phys}, 283(1):127--167, 2008.

\bibitem{Bachelot7}
A.~Bachelot.
\newblock The {Klein}-{Gordon} equation in {Anti}-de {Sitter} cosmology.
\newblock {\em J. Math. Pures Appl.}, 96:527--554, 2011.

\bibitem{BaMBa93}
A.~Bachelot and A.~Motet-Bachelot.
\newblock Les résonances d'un trou noir de schwarzschild.
\newblock {\em Ann. Inst. H. Poincar\'{e} Phys. Th\'{e}or.}, 59(1):3--68, 1993.

\bibitem{Bat06}
D.~Batic.
\newblock Scattering for massive {Dirac} fields on the {Kerr} metric.
\newblock {\em J. Math. Phys.}, 48, 2007.

\bibitem{BoHa}
J.F. Bony and D.~H\"{a}fner.
\newblock {Decay} and {Non}-{Decay} of the {Local} {Energy} for the {Wave}
  {Equation}.
\newblock {\em Comm. Math. Phys.}, 282:697--719, 2008.

\bibitem{BrFr82b}
P.~Breitenlohner and D.Z. Freedman.
\newblock {Positive} {energy} {in} {Anti-de} {Sitter} {backgrounds} {and}
  {gauged} {extended} {supergravity}.
\newblock {\em Phys. Lett. B}, 115.

\bibitem{BrFr82a}
P.~Breitenlohner and D.Z. Freedman.
\newblock {Stability} {in} {gauged} {extended} {supergravity}.
\newblock {\em Ann. Phys.}, 144.

\bibitem{CFKS}
H.L. Cycon, R.G. Froese, W.~Kirsch, and B.~Simon.
\newblock {\em {Schr\"{o}dinger} {Operators}}.
\newblock {Springer}, 2008.

\bibitem{DaHoRod14}
M.~Dafermos, G.~Holzegel, and I.~Rodnianski.
\newblock Scattering theory construction of dynamical vacuum black holes.
\newblock arXiv: 1306.5534.

\bibitem{DaRodRoth14}
M.~Dafermos, I.~Rodnianski, and Y.~Shlapentokh-Rothman.
\newblock Scattering theory for the wave equation on {Kerr} black hole
  exteriors.
\newblock arXiv: 1412.8379.

\bibitem{Dau04}
T.~Daud\'{e}.
\newblock Time-dependent scattering theory for massive charged {Dirac} fields
  by a {Kerr}-{Newman} black hole.
\newblock th\`{e}se de doctorat, Universit\'{e} Bordeaux 1, available online at
  http://tel.archives-ouvertes.fr, 2004.

\bibitem{Dau10}
T.~Daud\'{e}.
\newblock Time-dependent scattering theory for massive charged {Dirac} fields
  by a {Reissner}-{Nordstr\"{o}m} black hole.
\newblock {\em J. Math. Phys.}, 51, 2010.

\bibitem{DeBHS}
S.~De~Bi\`{e}vre, P.~Hislop, and I.M. Sigal.
\newblock Scattering theory for the wave equation on non-compact manifolds.
\newblock {\em Rev. Math. Phys.}, 4:575--618, 1992.

\bibitem{DeGe}
J.~Derezinski and C.~G\'erard.
\newblock {\em {Scattering} {Theory} {of} {Classical} {and} {Quantum}
  {N-Particle} {Systems}}.
\newblock {Springer}, 1997.

\bibitem{Dim85}
J.~Dimock and B.S. Kay.
\newblock Scattering for the wave equation on the schwarzschild metric.
\newblock {\em Gen. Rel. Grav.}, 17(4):353--359, 1985.

\bibitem{DimKay86b}
J.~Dimock and B.S. Kay.
\newblock Classical and quantum scattering theory for linear scalar fields on
  the {Schwarzschild} metric ii.
\newblock {\em J. Math. Phys.}, 27:2520--2525, 1986.

\bibitem{DimKay86a}
J.~Dimock and B.S. Kay.
\newblock Scattering for massive scalar fields on {Coulomb} potentials and
  {Schwarzschild} metrics.
\newblock {\em Class. Quantum Grav.}, 3:71--80, 1986.

\bibitem{DimKay87}
J.~Dimock and B.S. Kay.
\newblock Classical and quantum scattering theory for linear scalar fields on
  the {Schwarzschild} metric i.
\newblock {\em Ann. Phys.}, 175:366--426, 1987.

\bibitem{Dya13}
S.~Dyatlov.
\newblock Asymptotics of linear waves and resonances with applications to black
  holes.
\newblock arXiv: 1305.1723.

\bibitem{Dya11b}
S.~Dyatlov.
\newblock Exponential energy decay for kerr-de sitter black holes beyond event
  horizons.
\newblock {\em Math. Res. Lett.}, 18:1023--1035, 2011.

\bibitem{Dya11a}
S.~Dyatlov.
\newblock Quasi-normal modes and exponential energy decay for the kerr-de
  sitter black hole.
\newblock {\em Comm. Math. Phys.}, 306:119--163, 2011.

\bibitem{Dya12}
S.~Dyatlov.
\newblock Asymptotic distribution of quasi-normal modes for kerr-de sitter
  black holes.
\newblock {\em Ann. Henri Poincar\'{e}}, 13:1101--1166, 2012.

\bibitem{DyaZw16}
S.~Dyatlov and M.~Zworski.
\newblock {Mathematical} {Theory} of {Scattering} {Resonances}.
\newblock http: math.mit.edu/~dyatlov/res/res.pdf.

\bibitem{Hawking}
G.~F.~R. Ellis and S.~W. Hawking.
\newblock {\em {The} {Large} {Scale} {Structure} {of} {Space-time}}.
\newblock {Cambridge} {University} {Press}, 1973.

\bibitem{EncKam}
A.~Enciso and N.~Kamran.
\newblock A singular initial-boundary value problem for nonlinear wave
  equations and holography in asymptotically {Anti}-de {Sitter} spaces.
\newblock {\em J. Maths. Pure. Appl.}, 103:1053--1091, 2015.

\bibitem{Fried80}
F.G. Friedlander.
\newblock Radiation fields and hyperbolic scattering theory.
\newblock {\em Mat. Proc. Cambridge Philos. Soc.}, 88:483--515, 1980.

\bibitem{FLN05}
S.~Fujii\'{e}, C.~Lasser, and L.~N\'{e}d\'{e}lec.
\newblock Semiclassical resonances for two-level {Schr\"{o}dinger} operator
  with a conical intersection, 2005.

\bibitem{FuRa98}
S.~Fujii\'{e} and T.~Ramond.
\newblock Matrice de scattering et r\'{e}sonances associ\'{e}es \`{a} une
  orbite h\'{e}t\'{e}rocline.
\newblock {\em Ann. Inst. H. Poincar\'{e}}, 69:31--82, 1998.

\bibitem{FuRa99}
S.~Fujii\'{e} and T.~Ramond.
\newblock Exact wkb analysis and the langer modification with application to
  barrier top resonances.
\newblock In C.~Howls, editor, {\em Towards the Exact WKB Analysis of
  Differential Equations, Linear or Non-Linear}, pages 15--31. Kyoto University
  Press, 1999.

\bibitem{Gan16}
O.~Gannot.
\newblock Existence of quasinormal modes for {Kerr}-{AdS} black holes.
\newblock arXiv: 1602.08147.

\bibitem{Gan14}
O.~Gannot.
\newblock Quasinormal modes for {Schwarzschild}-{ADS} black holes: exponential
  convergence to the real axis.
\newblock {\em Commun. Math. Phys}, 330:771--799, 2014.

\bibitem{GGH14}
V.~Georgescu, C.~G\'{e}rard, and D.~H\"{a}fner.
\newblock Asymptotic completeness for superradiant {Klein}-{Gordon} equations
  and applications to the {De} {Sitter} {Kerr} metric.
\newblock arXiv:1405.5304, 2014.

\bibitem{GeGr88}
C.~G\'{e}rard and A.~Grigis.
\newblock Precise estimates of tunneling and eigenvalues near a potential
  barrier.
\newblock {\em J. Differ. Eqs.}, 42:149--177, 1988.

\bibitem{GeNi}
C.~G\'erard and F.~Nier.
\newblock Scattering theory for the perturbations of periodic {Schr\"{o}dinger}
  operators.
\newblock {\em J. Math. Kyoto Univ.}, 38(4):595--634, 1998.

\bibitem{GuiMoPa}
C.~Guillarmou, S.~Moroianu, and J.~Park.
\newblock Eta invariant and selberg zeta function of odd type over convex
  co-compact hyperbolic manifolds.
\newblock {\em Advances in Math.}, 225:2464--2516, 2010.

\bibitem{Ha01}
D.~H\"{a}fner.
\newblock Compl\'{e}tude asymptotique pour l'\'{e}quation des ondes dans une
  classe d'espaces-temps stationnaires et asymptotiquement plats.
\newblock {\em Ann. Inst. Fourier}, 51(3):779--833, 2001.

\bibitem{Ha03}
D.~H\"{a}fner.
\newblock {Sur} {la} {Th\'{e}orie} {de} {la} {diffusion} {pour}
  {l'\'{e}quation} {de} {Klein-Gordon} {dans} {la} {m\'{e}trique} {de} {Kerr}.
\newblock {\em Dissertationes Mathematicae}, 421, 2003.

\bibitem{Ha09}
D.~H\"{a}fner.
\newblock Creation of fermions by rotating charged black holes.
\newblock {\em M\'{e}m. Soc. Math. Fr.}, (117):158 pp, 2009.

\bibitem{HaNi04}
D.~H\"{a}fner and J.~P. Nicolas.
\newblock {Scattering} {of} {massless} {Dirac} {fields} {by} {a} {Kerr} {black}
  {Hole}.
\newblock {\em Rev. Math. Phys.}, 16(1):29--123, 2004.

\bibitem{HeSj86}
B.~Helffer and J.~Sj\"{o}strand.
\newblock {\em R\'{e}sonances en limite semi-classique}.
\newblock M\'{e}moires de la {Soc\'{e}t\'{e}} {Mathématique} de {France},
  1986.

\bibitem{Holsmu2}
G.~Holzegel and J.~Smulevici.
\newblock Decay properties of {Klein}-{Gordon} fields on {Kerr}-{AdS}
  spacetimes.
\newblock {\em Comm. Pure Appl. Math.}, 66(11):1751--1802, 2013.

\bibitem{Holsmu}
G.~Holzegel and J.~Smulevici.
\newblock Stability of {Schwarzschild}-{AdS} for the spherically symmetric
  {Einstein}-{Klein}-{Gordon} system.
\newblock {\em Commun. Math. Phys}, 317(1):205--251, 2013.

\bibitem{HolSmu3}
G.~Holzegel and J.~Smulevici.
\newblock Quasimodes and a lower bound on the uniform energy decay rate for
  {Kerr}-{AdS} spacetimes.
\newblock {\em Anal. PDE}, 7, 2014.

\bibitem{HolWar}
G.~Holzegel and C.M. Warnick.
\newblock Boundedness and growth for the massive wave equation on
  asymptotically {Anti}-de {Sitter} black holes.
\newblock {\em J. Funct. Anal.}, 266:2436--2485.

\bibitem{Hormander}
L.~H\"{o}rmander.
\newblock {\em {The} {Analysis} {of} {Linear} {Partial} {Differential}
  {Operators} {1-3}}.
\newblock {Springer}, 1985.

\bibitem{Ia14}
A.~Iantchenko.
\newblock Quasi-normal modes for de {Sitter}-{Reissner}-{Nordstr\" {o}m} black
  holes.
\newblock arXiv: 1407.3654.

\bibitem{IaKo14}
A.~Iantchenko and E.~Korotyaev.
\newblock Resonances for 1d massless dirac operators.
\newblock {\em J. Difference Equ.}, 256:3038--3066, 2014.

\bibitem{IaKo14bis}
A.~Iantchenko and E.~Korotyaev.
\newblock Resonances for dirac operators on the half line.
\newblock {\em J. Math. Anal. Appl.}, 420:279--313, 2014.

\bibitem{IshWa03}
A.~Ishibashi and R.M. Wald.
\newblock {Dynamics} {in} {non-globally} {hyperbolic}, {static} {space-times}:
  {II}. {General} {analysis} {of} {prescriptions} {for} {dynamics}.
\newblock {\em Class. Quantum Grav.}, 20, 2003.

\bibitem{IshWa04}
A.~Ishibashi and R.M. Wald.
\newblock {Dynamics} {in} {non-globally} {hyperbolic}, {static} {space-times}:
  {III}. {Anti-de} {Sitter} {space-time}.
\newblock {\em Class. Quantum Grav.}, 21, 2004.

\bibitem{Jin98}
W.M. Jin.
\newblock Scattering of massive dirac fields on the schwarzschild black hole.
\newblock {\em Class. Quant. Grav.}, 15:3163--3175, 1998.

\bibitem{Kh07}
A.~Khochman.
\newblock Resonances and spectral shift function for the semi-classical dirac
  operator.
\newblock {\em Rev. Math. Phys.}, 19:1071--1115, 2007.

\bibitem{Kh09}
A.~Khochman.
\newblock Klein paradox and scattering theory for the semi-classical dirac
  equation.
\newblock {\em Asymptot. Anal.}, 65:223--249, 2009.

\bibitem{Me00}
F.~Melnyk.
\newblock Waves operators for the massive charged linear fields on the
  {Reissner}-{Nordstr\"{o}m} metric.
\newblock {\em Class. Quant. Grav.}, 17:2281--2296, 2000.

\bibitem{Me03}
F.~Melnyk.
\newblock Scattering on {Reissner}-{Nordstr\"{o}m} metric for massive charged
  spin-$\frac{1}{2}$ fields.
\newblock {\em Ann. Henri Poincar\'{e}}, 4(5):813--846, 2003.

\bibitem{Me04b}
F.~Melnyk.
\newblock The hawking effect for collapsing star in an initial state of {KMS}
  type.
\newblock {\em J. Phys. A: Math. Gen.}, 37:9225--9249, 2004.

\bibitem{Me04a}
F.~Melnyk.
\newblock The hawking effect for spin $\frac{1}{2}$ fields.
\newblock {\em Comm. Math. Phys.}, 244(3):483--525, 2004.

\bibitem{MeSBV14}
R.~Melrose, A.~S\`{a}~Barreto, and A.~Vasy.
\newblock Asymptotics of solutions of the wave equation on de
  sitter-schwarzschild space.
\newblock {\em Comm. in PDE.}, 39:512--529, 2014.

\bibitem{JPN13}
J.-P. Nicolas.
\newblock {Conformal} scattering on the {Schwarzschild} metric.
\newblock arXiv:1312.1386, to appear in {Annales} de l'{Institut} {Fourier}.

\bibitem{JPN}
J.-P. Nicolas.
\newblock Scattering of linear {Dirac} fields by a spherically symmetric black
  hole.
\newblock {\em Ann. Inst. H. Pioncar\'{e} Phys. Th\'{e}or.}, 62(2):35--58,
  1995.

\bibitem{BON}
B.~O'Neill.
\newblock {\em {Semi-Riemannian} {Geometry} {With} {Applications} {to}
  {Relativity}}.
\newblock {Academic} {Press}, 1983.

\bibitem{Pais}
A.~Pais.
\newblock {\em Subtle is the lord}.
\newblock Oxford University Press, 2005.

\bibitem{Penrose}
R.~Penrose and W.~Rindler.
\newblock {\em Spinors and {Space}-time}, volume~I of {\em Cambridge Monographs
  on Mathematical Physics}.
\newblock Cambridge University Press, 1984.

\bibitem{Ra96}
T.~Ramond.
\newblock Semiclassical study of quantum scattering on the line.
\newblock {\em Commun. Math. Phys}, 177:221--254, 1996.

\bibitem{Reedsimon}
M.~Reed and B.~Simon.
\newblock {\em {Methods} {of} {Modern} {Mathematical} {Physics}}, volume 1-4.
\newblock {Academic} {Press}, 1975.

\bibitem{SBZ97}
A.~S\`{a}~Barreto and M.~Zworski.
\newblock Distibution of resonances for spherical black holes.
\newblock {\em Math. Res. Lett.}, 4:103--122, 1997.

\bibitem{TaZw98}
S.~Tang and M.~Zworski.
\newblock From quasimodes to resonances.
\newblock {\em Math. Res. Lett.}, 5:261--272, 1998.

\bibitem{Thaller}
B.~Thaller.
\newblock {\em {The} {Dirac} {Equation}}.
\newblock {Springer}, 1992.

\bibitem{Va13}
A.~Vasy.
\newblock Microlocal analysis of asymptotically hyperbolic and kerr-de sitter
  spaces.
\newblock {\em Inv. Math.}, 194:381--513, 2013.

\bibitem{Wald}
R.M. Wald.
\newblock {\em {General} {Relativity}}.
\newblock {The} {University} {of} {Chicago} {Press}, 1984.

\bibitem{War14}
C.~Warnick.
\newblock {On} quasinormal modes of asymptotically anti-de sitter black holes.
\newblock arXiv: 1306.5760, to appear in {Comm}. {Math}. {Phys}.

\bibitem{Weinb}
S.~Weinberg.
\newblock {\em Gravitation and Corsmology}.
\newblock John Wiley and Sons, 1972.

\end{thebibliography}
\nocite{*}

Université de Grenoble, Institut Fourier, 100, rue des maths, BP 74, 38402 Saint-Martin d'Hères Cedex (France)

~\
E-mail address: guillaume.idelon-riton@ujf-grenoble.fr

\end{document}